\documentclass[letterpaper,10pt]{article}

\usepackage[dvipsnames]{xcolor}
\usepackage{
  amsmath, amsthm, amssymb, mathtools, dsfont, units,       % Math typesetting
  graphicx, wrapfig, subfig, float,                         % Figures and graphics formatting
  listings, color, inconsolata, pythonhighlight,            % Code formatting
  fancyhdr, hyperref, framed, authblk                      % Headers/footers, section fonts
}
\usepackage{newpxtext, newpxmath, inconsolata}
\usepackage[cal=cm,scr=boondox,bb=dsfontserif]{mathalpha}
\usepackage[inline]{enumitem}
\usepackage{array,longtable}

\usepackage[
  backend=biber,
  bibencoding=utf8,
  style=numeric,
  giveninits=true,
  maxnames = 999
]{biblatex}
\usepackage[left=1in, right=1in, top=1.0in, bottom=.9in, headsep=.2in, footskip=0.35in]{geometry}

\usepackage[bottom]{footmisc}

\allowdisplaybreaks

\usepackage[font={it,footnotesize}]{caption}

\hypersetup{colorlinks=true, linkcolor=RoyalBlue, citecolor=RedOrange, urlcolor=RoyalBlue}

\usepackage{titlesec}
\titleformat{\section}{\large\bfseries\selectfont}{\thesection\;\;\;}{0em}{}
\titleformat{\subsection}{\normalsize\bfseries\selectfont}{\thesubsection\;\;\;}{0em}{}

\setlist[itemize]{wide=0pt, leftmargin=16pt, labelwidth=10pt, align=left}

\usepackage[subfigure]{tocloft}

\graphicspath{{Images/}{../Images/}}

\theoremstyle{plain}
\newtheorem{theorem}{Theorem}
\newtheorem{proposition}[theorem]{Proposition}

\newtheorem{lemma}[theorem]{Lemma}
\newtheorem{corollary}[theorem]{Corollary}

\theoremstyle{definition}
\newtheorem{definition}[theorem]{Definition}
\newtheorem{remark}[theorem]{Remark}

\newtheorem*{notations*}{Notations}
\numberwithin{theorem}{section}
\numberwithin{equation}{section}

\newcommand{\bba}{\mathbb{A}}
\newcommand{\bbd}{\mathbb{D}}
\newcommand{\bbe}{\mathbb{E}}
\newcommand{\bbg}{\mathbb{G}}
\newcommand{\bbk}{\mathbb{K}}
\newcommand{\bbn}{\mathbb{N}}
\newcommand{\bbp}{\mathbb{P}}
\newcommand{\bbr}{\mathbb{R}}
\newcommand{\bbt}{\mathbb{T}}
\newcommand{\bbone}{\mathbb{1}}

\newcommand{\bfa}{\mathbf{a}}
\newcommand{\bfb}{\mathbf{b}}

\newcommand{\bfe}{\mathbf{e}}

\newcommand{\bfs}{\mathbf{s}}

\newcommand{\bfu}{\mathbf{u}}
\newcommand{\bfv}{\mathbf{v}}
\newcommand{\bfx}{\mathbf{x}}
\newcommand{\bfy}{\mathbf{y}}

\newcommand{\bfDelta}{\boldsymbol{\Delta}}
\newcommand{\bfPi}{\boldsymbol{\Pi}}
\newcommand{\bfSigma}{\boldsymbol{\Sigma}}
\newcommand{\bfUps}{\boldsymbol{\Upsilon}}
\newcommand{\bfA}{\mathbf{A}}
\newcommand{\bfB}{\mathbf{B}}

\newcommand{\bfD}{\mathbf{D}}
\newcommand{\bfF}{\mathbf{F}}
\newcommand{\bfG}{\mathbf{G}}
\newcommand{\bfH}{\mathbf{H}}
\newcommand{\bfI}{\mathbf{I}}
\newcommand{\bfL}{\mathbf{L}}
\newcommand{\bfM}{\mathbf{M}}
\newcommand{\bfP}{\mathbf{P}}
\newcommand{\bfQ}{\mathbf{Q}}
\newcommand{\bfR}{\mathbf{R}}
\newcommand{\bfS}{\mathbf{S}}

\newcommand{\bfX}{\mathbf{X}}
\newcommand{\bfY}{\mathbf{Y}}

\newcommand{\fkc}{\mathfrak{c}}
\newcommand{\fkd}{\mathfrak{d}}
\newcommand{\fke}{\mathfrak{e}}
\newcommand{\fkr}{\mathfrak{r}}
\newcommand{\fks}{\mathfrak{s}}
\newcommand{\fkm}{\mathfrak{m}}
\newcommand{\fkn}{\mathfrak{n}}
\newcommand{\fku}{\mathfrak{u}}
\newcommand{\fkw}{\mathfrak{w}}
\newcommand{\fkx}{\mathfrak{x}}
\newcommand{\fkq}{\mathfrak{q}}
\newcommand{\fkD}{\mathfrak{D}}
\newcommand{\fkJ}{\mathfrak{J}}
\newcommand{\fkS}{\mathfrak{S}}
\newcommand{\fkW}{\mathfrak{W}}
\newcommand{\fkX}{\mathfrak{X}}
\newcommand{\fkY}{\mathfrak{Y}}

\newcommand{\bfke}{\boldsymbol{\mathfrak{e}}}
\newcommand{\bfku}{\boldsymbol{\mathfrak{u}}}
\newcommand{\bfkw}{\boldsymbol{\mathfrak{w}}}
\newcommand{\bfks}{\boldsymbol{\mathfrak{s}}}

\newcommand{\caA}{\mathcal{A}}
\newcommand{\caB}{\mathcal{B}}
\newcommand{\caC}{\mathcal{C}}
\newcommand{\caE}{\mathcal{E}}
\newcommand{\caF}{\mathcal{F}}
\newcommand{\caG}{\mathcal{G}}
\newcommand{\caH}{\mathcal{H}}
\newcommand{\caI}{\mathcal{I}}
\newcommand{\caJ}{\mathcal{J}}

\newcommand{\caL}{\mathcal{L}}
\newcommand{\caS}{\mathcal{S}}
\newcommand{\caT}{\mathcal{T}}

\newcommand{\scA}{\mathscr{A}}
\newcommand{\scN}{\mathscr{N}}
\newcommand{\scP}{\mathscr{P}}

\newcommand{\rmd}{\mathrm{d}}
\newcommand{\rmi}{\mathrm{i}}
\renewcommand{\Re}{\operatorname{Re}}
\renewcommand{\Im}{\operatorname{Im}}
\newcommand{\Var}{\operatorname{Var}}
\newcommand{\diag}{\operatorname{diag}}

\newcommand{\rank}{\operatorname{rank}}

\newcommand{\PRM}{\operatorname{PRM}}
\newcommand{\pconv}{\xrightarrow{\bbp}}

\DeclarePairedDelimiter\abs{\lvert}{\rvert}
\DeclarePairedDelimiter\angles{\langle}{\rangle}
\DeclarePairedDelimiter\braks{\lbrack}{\rbrack}
\DeclarePairedDelimiter\curls{\lbrace}{\rbrace}
\DeclarePairedDelimiter\dbraks{\llbracket}{\rrbracket}
\DeclarePairedDelimiter\norm{\lVert}{\rVert}
\DeclarePairedDelimiter\pars{\lparen}{\rparen}
\NewDocumentCommand{\bigO}{o m}{%
  O\IfNoValueTF{#1}
    {\pars{#2}}        % no optional argument
    {\pars[#1]{#2}}    % with optional argument
}
\NewDocumentCommand{\smallo}{o m}{%
  o\IfNoValueTF{#1}
    {\pars{#2}}        % no optional argument
    {\pars[#1]{#2}}    % with optional argument
}

\newcommand{\eps}{\varepsilon}

\newcommand{\shortpara}[1]{\noindent\underline{\textit{#1}}}

\begin{document}

% --- Main title and subtitle ---

\title{Staircase phase transitions for the extreme eigenvalues of heavy-tailed sample correlation matrices}

% --- Author and date of last update ---

\author[1]{Yanpeng Li\textsuperscript{a,}}
\author[2]{Zeqin Lin\textsuperscript{b,}}
\author[3]{Yiming Liu\textsuperscript{c,}}
\author[4]{Jiahui Xie\textsuperscript{d,}}
\author[5]{Haozhu Zhao\textsuperscript{e,}}

\affil[1]{Harbin Institute of Technology \textsuperscript{a}\texttt{\href{mailto:20230256@hit.edu.cn}{20230256@hit.edu.cn}}}
\affil[2]{Nanyang Technological University \textsuperscript{b}\texttt{\href{mailto:zeqin.lin@ntu.edu.sg}{zeqin.lin@ntu.edu.sg}}}
\affil[3]{Jinan University \textsuperscript{c}\texttt{\href{mailto:liuyiming@jnu.edu.cn}{liuyiming@jnu.edu.cn}}}
\affil[4]{National University of Singapore \textsuperscript{d}\texttt{\href{mailto:jhxie@nus.edu.sg}{jhxie@nus.edu.sg}}}
\affil[5]{Changchun University of Science and Technology \textsuperscript{e}\texttt{\href{mailto:zhaohz24@163.com}{zhaohz24@163.com}}}

\date{\vspace{-2.5em}}

\renewcommand\Authfont{\normalsize}
\renewcommand\Affilfont{\footnotesize}

% --- Add title and table of contents ---

\maketitle
\begin{abstract}
We establish phase transitions in the first-order asymptotics of the largest and smallest nonzero eigenvalues of high-dimensional sample correlation matrices formed from data with i.i.d. centered real entries of finite variance, allowing an infinite fourth moment. Because of self-normalization, spectral outliers are driven by collisions of large entries in distinct rows of a common column rather than by individual large entries. The resulting phase transitions depend jointly on the aspect ratio and the entry tail, producing staircase boundaries in the phase diagram. We prove convergence to the corresponding Mar\v{c}enko--Pastur edges in subcritical regimes and to deterministic outlier levels in supercritical regimes. In critical regimes, Poisson limits for the point processes of spectral outliers and nondegenerate limiting laws for the extreme eigenvalues are established.
\end{abstract}
\tableofcontents \label{sec:contents}

% --- Main content: import lectures as subfiles ---

% *******************************************
\section{Introduction}
\label{sec:introduction}
% *******************************************

Let $\xi$ be a real-valued random variable, and let $(p_n)_{n \geq 1} \subset \bbn$. For each $n$, consider a $p_n \times n$ data matrix with i.i.d. entries distributed as $\xi$,
\begin{equation}
    \bfX_n = (x_{i \mu})_{i \in \dbraks{p_n}, \, \mu \in \dbraks{n}}
    \in \bbr^{p_n \times n},
    \qquad
    \bfX_n^\top = (\bfx_1, \ldots, \bfx_{p_n}).
    \label{def:data-matrix}
\end{equation}
Following the conventions of multivariate statistics, we regard $p_n$ as the data dimension and $n$ as the sample size. Hence, the $i$-th row $\bfx_i = (x_{i 1}, \ldots, x_{i n})^\top \in \bbr^n$ collects the observations of the $i$-th feature. The basic assumptions used in this manuscript are
\begin{equation}
    \bbe \xi = 0,
    \qquad
    \bbe \abs{\xi}^2 = 1,
    \qquad
    p_n / n \to \phi \in (0, \infty).
    \label{eqn:basic-assumption}
\end{equation}
Here the last condition on the aspect ratio places us in the proportional high-dimensional regime $p_n \asymp n$. The \emph{uncentered sample covariance matrix} associated with $\bfX_n$ is
\begin{equation}
    \bfS_n := \frac{1}{n} \bfX_n \bfX_n^\top = (S_{ij})_{i,j=1}^{p_n},
    \qquad
    S_{ij} := \frac{1}{n} \angles{\bfx_i, \bfx_j}
    = \frac{1}{n} \sum\nolimits_{\mu=1}^n x_{i \mu} x_{j \mu}.
    \label{def:uncentered-covariance}
\end{equation}
Let $\bfD_n$ be the diagonal matrix formed from the diagonal entries of $\bfS_n$,
\begin{equation}
    \bfD_n = \diag (D_1, \ldots, D_{p_n}),
    \qquad
    D_i = \norm{\bfx_i}^2 / n.
\end{equation}
The main object studied in this manuscript is the \emph{uncentered sample correlation matrix}
\begin{equation}
    \bfR_n = \bfD_n^{-1 / 2} \bfS_n \bfD_n^{-1 / 2}
    = \bfY_n \bfY_n^\top,
    \qquad
    \bfY_n^\top = (\bfy_1, \ldots, \bfy_{p_n}),
    \qquad
    \bfy_i = {\bfx_i} / {\norm{\bfx_i}}.
    \label{def:uncentered-correlation}
\end{equation}
Specifically, its entries are
\begin{equation*}
    \bfR_n := (R_{ij})_{i,j=1}^{p_n},
    \qquad
    R_{i j} = \frac{\angles{\bfx_i, \bfx_j}}
    {\norm{\bfx_i} \, \norm{\bfx_j}}
    = \frac{\sum_{\mu = 1}^n x_{i \mu} x_{j \mu}}
    {\sqrt{\pars[\big]{\sum_{\mu = 1}^n \abs{x_{i \mu}}^2}
    \pars[\big]{\sum_{\mu = 1}^n \abs{x_{j \mu}}^2}}}.
\end{equation*}
By \eqref{def:uncentered-correlation}, the matrix $\bfY_n$ is obtained by normalizing each row of $\bfX_n$ to have unit norm, and $\bfR_n$ is the corresponding Gram matrix. The correlation matrix $\bfR_n$ is invariant under positive rescaling of individual rows, and therefore the normalization to unit variance in \eqref{eqn:basic-assumption} entails no loss of generality for entry distributions with zero mean and finite variance $\bbe \abs{\xi}^2 \in (0, \infty)$.

For entry distributions with nonzero population mean, $\bbe \xi \neq 0$, a more standard choice in multivariate statistics is the \emph{Pearson sample correlation matrix}, which centers each feature by its sample mean before normalization. Specifically, define the sample means
\begin{equation}
    \bar{x}_i := \frac{1}{n} \angles{\mathbf{1}_n, \bfx_i}
    = \frac{1}{n} \sum\nolimits_{\mu=1}^n x_{i \mu},
    \qquad
    \mathbf{1}_n = (1, \ldots, 1)^\top \in \bbr^n.
    \label{def:pearson-projection}
\end{equation}
The \emph{Pearson sample correlation matrix} is then defined as
\begin{equation}
    \bfR_n^\circ
    = \bfY_n^\circ (\bfY_n^\circ)^\top,
    \qquad
    \bfy_i^\circ
    = \bfx_i^{\circ} / \norm{\bfx_i^{\circ}},
    \qquad
    \bfx_i^{\circ} 
    := \bfx_i - \bar{x}_i \mathbf{1}_n.
    \label{def:pearson-correlation}
\end{equation}
The Pearson correlation matrix $\bfR_n^\circ$ is invariant under shifts of individual features. Here, to simplify the presentation, we first state our main results for the uncentered sample correlation matrix $\bfR_n$ under \eqref{eqn:basic-assumption}, and then extend them to $\bfR_n^\circ$ in Theorem \ref{cor:pearson-extension}.

For a real symmetric $\bfA \in \bbr^{d \times d}$, write $\lambda_1 (\bfA) \geq \cdots \geq \lambda_d (\bfA)$ for its eigenvalues in nonincreasing order. This manuscript studies the first-order asymptotics of the largest eigenvalue $\lambda_1 (\bfR_n)$ and the smallest nonzero eigenvalue $\lambda_{p_n \wedge n} (\bfR_n)$ under \eqref{eqn:basic-assumption}. We allow entries with heavy tails and, in particular, do not require the finite fourth moment condition $\bbe \abs{\xi}^4 < \infty$. Our objective is to determine the tail conditions on $\xi$ that govern whether these eigenvalues converge to the respective edges of the bulk spectrum or separate from them. The interaction between heavy tails and self-normalization produces phase transitions at both spectral edges whose structure depends on the aspect ratio $\phi$. Before presenting our results, we first briefly review the related literature on limiting spectral distributions (LSDs) and extreme eigenvalues for $\bfS_n$ and $\bfR_n$.

Under \eqref{eqn:basic-assumption}, the LSD of the sample covariance matrix $\bfS_n$ is described by the celebrated Mar\v{c}enko--Pastur (MP) law \cite{marcenkoDistributionEigenvaluesSets1967}. For a symmetric matrix $\bfA \in \bbr^{m \times m}$, let
\begin{equation*}
    F^{\bfA} (t) = \frac{1}{m} \sum\nolimits_{i=1}^m \bbone \{ \lambda_i (\bfA) \leq t \}
\end{equation*}
denote its empirical spectral distribution (ESD). Under \eqref{eqn:basic-assumption}, the ESD $F^{\bfS_n}$ converges weakly almost surely to the MP law with aspect ratio $\phi$, whose distribution function is
\begin{equation}
    \tilde{F}_\phi (t)
    = \int_{-\infty}^t
    \frac{\sqrt{(\lambda_+ - x) (x - \lambda_-)}}{2 \pi \phi x}
    \bbone \{\lambda_- < x < \lambda_+\} \, \rmd x
    + \pars*{1 - \frac{1}{\phi}}_+ \bbone \{t \geq 0\},
    \label{def:MP-law}
\end{equation}
where the lower and upper MP edges are respectively given by
\begin{equation}
    \lambda_{-} := (1 - \sqrt{\phi})^2,
    \qquad
    \lambda_{+} := (1 + \sqrt{\phi})^2.
    \label{def:MP-edges}
\end{equation}
See also \cite[Theorem 3.6]{baiSpectralAnalysisLarge2010}. The continuous part of $\tilde{F}_\phi$ is supported on $[\lambda_-, \lambda_+]$. When $\phi > 1$, the law also has an atom of mass $1 - 1 / \phi$ at zero. 

The MP law \eqref{def:MP-law} describes the global distribution of eigenvalues but does not by itself determine the limiting behavior of the extreme eigenvalues of $\bfS_n$. This behavior is characterized by the classical Bai--Yin theorem \cite{baiLimitSmallestEigenvalue1993,yinLimitLargestEigenvalue1988}, which states that, under \eqref{eqn:basic-assumption} and the additional fourth-moment condition $\bbe \abs{\xi}^4 < \infty$,
\begin{equation}
    \lambda_1(\bfS_n) \to \lambda_+,
    \qquad
    \lambda_{p_n \wedge n}(\bfS_n) \to \lambda_-,
    \qquad \text{a.s.}
    \label{eqn:classical-covariance-edges}
\end{equation}
Here, the upper edge limit was established by Yin, Bai, and Krishnaiah \cite{yinLimitLargestEigenvalue1988}, and the lower edge limit by Bai and Yin \cite{baiLimitSmallestEigenvalue1993}. Bai, Silverstein, and Yin \cite{baiNoteLargestEigenvalue1988} further showed that the condition of finite fourth moment is necessary for the almost-sure convergence of $\lambda_1(\bfS_n)$ to $\lambda_+$.

For convergence to the lower edge, however, the moment assumption can be substantially weakened. This breakthrough was achieved by Tikhomirov \cite{tikhomirovLimitSmallestSingular2015}, who established that, for $\phi \ne 1$, assumption \eqref{eqn:basic-assumption} alone already ensures the almost-sure convergence $\lambda_{p_n \wedge n} (\bfS_n) \to \lambda_-$. Thus, the lower edge of the sample covariance matrix is more robust to heavy tails: its almost-sure convergence requires only a finite second moment, whereas a finite fourth moment is necessary at the upper edge.

For the uncentered sample correlation matrix $\bfR_n$, Jiang \cite{jiangLimitingDistributionsEigenvalues2004} established the MP law under \eqref{eqn:basic-assumption} and the almost-sure convergence of both extreme eigenvalues when $\bbe \abs{\xi}^4 < \infty$,
\begin{equation}
    \lambda_1(\bfR_n) \to \lambda_+,
    \qquad
    \lambda_{p_n \wedge n}(\bfR_n) \to \lambda_-,
    \qquad \text{a.s.}
    \label{eqn:classical-correlation-edges}
\end{equation}
For the Pearson correlation matrix $\bfR_n^\circ$, Jiang \cite{jiangLimitingDistributionsEigenvalues2004} also established the MP law under \eqref{eqn:basic-assumption} and convergence at the upper edge by additionally requiring $\bbe \abs{\xi}^4 < \infty$, while leaving convergence at the lower edge as a conjecture. This was subsequently resolved by Xiao and Zhou \cite{xiaoAlmostSureLimit2010}. The key step in Jiang's proof of \eqref{eqn:classical-correlation-edges} for the uncentered matrix is the diagonal relation \eqref{def:uncentered-correlation}. Under $\bbe \abs{\xi}^4 < \infty$, the rowwise empirical second moments $D_i = \norm{\bfx_i}^2 / n$ converge uniformly to one, allowing the covariance edge limits in \eqref{eqn:classical-covariance-edges} to be transferred to $\bfR_n$. More recently, Jiang, Zhu, and Li \cite{jiangEigenvaluesRenormalizedSample2026} established almost sure convergence of the largest eigenvalue of a renormalized Pearson correlation matrix when $p_n \asymp n^\beta$ for $\beta \geq 1$, under the moment condition $\bbe \abs{\xi}^{2 (\beta + 1)} < \infty$.

Motivated by Tikhomirov's result \cite{tikhomirovLimitSmallestSingular2015} for the smallest eigenvalue of the sample covariance matrix, it is natural to ask whether the extreme eigenvalue convergence in \eqref{eqn:classical-correlation-edges} can be established under assumptions weaker than a finite fourth moment. Heiny and Mikosch \cite{heinyAlmostSureConvergence2018} gave a criterion under which both limits in \eqref{eqn:classical-correlation-edges} remain valid. More precisely, write $\bfy = (y_1, \ldots, y_n)^\top = \bfx / \norm{\bfx}$ for a normalized row. For $\phi \in (0, 1]$ and symmetric $\xi$, their key assumption is the following recursive moment bound: there exist sequences $q_n \to \infty$ and $k_n \in \bbn$ such that $k_n / \log n \to \infty$, $k_n^3 q_n / n \to 0$, and
\begin{equation}
    \bbe \braks[\big]{\abs{y_1}^{2m_1} \cdots
    \abs{y_{r-1}}^{2m_{r-1}} \abs{y_r}^{2m_r}}
    \leq \frac{q_n}{n}
    \bbe \braks[\big]{\abs{y_1}^{2m_1} \cdots
    \abs{y_{r-1}}^{2m_{r-1}} \abs{y_r}^{2m_r-2}},
    \label{cond:Heiny-Mikosch}
\end{equation}
for all $1 \leq r \leq \ell - 1$ and positive integers $m_1, \ldots, m_r$ satisfying $m_1 + \cdots + m_r = \ell \leq k_n$. Condition \eqref{cond:Heiny-Mikosch} is expressed through mixed moments of a self-normalized row and can be difficult to verify directly from the law of $\xi$. Heiny and Mikosch \cite[Section 3.2]{heinyAlmostSureConvergence2018} note that such verification may require evaluating high-dimensional integrals, which could be difficult even for classical heavy-tailed distributions such as the Student $t$ and symmetrized Pareto distributions. In particular, they do not provide a concrete example of a distribution with $\bbe \abs{\xi}^4 = \infty$ that satisfies condition \eqref{cond:Heiny-Mikosch}. Thus, although condition \eqref{cond:Heiny-Mikosch} permits an infinite fourth moment, it does not translate into a simple moment or tail criterion on $\xi$.

A central difficulty in analyzing the extreme eigenvalues of $\bfR_n$ is the dependence induced by self-normalization. Although the rows of $\bfY_n$ are independent, the coordinates within each normalized row $\bfy_i = \bfx_i / \norm{\bfx_i}$ share the random normalizer $\norm{\bfx_i}$. When $\bbe \abs{\xi}^4 < \infty$, Jiang's comparison \cite{jiangLimitingDistributionsEigenvalues2004} between $\bfS_n$ and $\bfR_n$ bypasses this dependence, but the same reduction need not remain valid for heavy-tailed entries. Chafa{\"i} and Tikhomirov \cite{chafaiConvergenceExtremalEigenvalues2018} developed a framework for empirical covariance matrices built from independent isotropic random vectors with possibly dependent coordinates, establishing convergence in probability at the lower and upper MP edges under the weak and strong tail projection properties, respectively. In the tall regime $p_n / n \to \phi > 1$, their lower edge result applies to the companion matrix $\bfQ_n = \bfY_n^\top \bfY_n$ whenever the rescaled rows $\sqrt{n} \bfy_i$ satisfy the isotropy and weak tail projection hypotheses. The latter may hold under assumptions weaker than $\bbe \abs{\xi}^4 < \infty$. In the wide regime $p_n / n \to \phi < 1$, however, their result cannot be applied directly to $\lambda_{p_n}(\bfR_n)$, since doing so would require treating the dependent columns of $\bfY_n$ as observations.

Finally, for entries with regularly varying tails of index $\alpha \in (0, 4)$, Auffinger, Ben Arous, and P\'{e}ch\'{e} \cite{auffingerPoissonConvergenceLargest2009} established Poisson limits for the rescaled upper extreme eigenvalues of Wigner and sample covariance matrices, yielding Fr\'{e}chet limits for their largest eigenvalues. Recently, at the crossover scale $\lim_{t \to \infty} t^4 \bbp \{\abs{\xi} > t\} \in (0, \infty)$, Diaconu \cite{diaconuMoreLimitingDistributions2023} used the moment method to derive a deformed Fr\'{e}chet law for the largest eigenvalue of Wigner matrices with symmetric entry distributions. The resulting limiting law has an atom at the upper spectral edge. Han \cite{hanDeformedFrechetLaw2025} subsequently proved the corresponding limiting law for $\lambda_1 (\bfS_n)$ by combining the resampling technique with a resolvent analysis of finite-rank perturbations. In these results, the behavior of the largest eigenvalues is driven by individual large entries.

In this manuscript, we establish the first-order asymptotics of both the largest and smallest nonzero eigenvalues of $\bfR_n$ under conditions stated directly in terms of the distribution of $\xi$. Section \ref{subsec:main-results} presents our main results for $\lambda_1 (\bfR_n)$, which reveal a hierarchy of transitions generated by collisions of large entries in distinct rows and a common column. Depending on the tail of $\xi$, we distinguish four regimes: convergence to the upper MP edge $\lambda_+$, convergence to deterministic outlier levels, nondegenerate critical limits, and divergence. The mechanism producing these outliers differs from that in \cite{diaconuMoreLimitingDistributions2023,hanDeformedFrechetLaw2025}: because self-normalization couples the entries within each row, individual large entries alone do not suffice to produce spectral outliers. Outliers instead require collisions involving sufficiently many rows. A remarkable feature of this phase transition is the dependence on the aspect ratio $\phi$, which determines both the minimum number of colliding rows needed to produce an outlier above $\lambda_+$ and the locations of the deterministic outlier levels.

In Section \ref{subsec:lower-main-results}, we establish corresponding results for the smallest nonzero eigenvalue $\lambda_{p_n \wedge n} (\bfR_n)$. The phase diagram again depends on the aspect ratio $\phi$ but has a simpler structure. In the tall regime $\phi > 1$, we prove almost sure convergence to the lower MP edge $\lambda_-$ under the finite variance condition alone, matching Tikhomirov's result \cite{tikhomirovLimitSmallestSingular2015} for $\bfS_n$. In contrast, in the wide regime $0 < \phi < 1$, we show that $\lambda_{p_n} (\bfR_n)$ converges in probability to the lower MP edge $\lambda_-$ or zero according to whether $t^3 \bbp \{\abs{\xi} > t\}$ tends to zero or infinity, respectively. At the crossover scale $t^3 \bbp \{\abs{\xi} > t\} \to \omega \in (0, \infty)$, we obtain a nondegenerate limiting law analogous to the critical laws for the largest eigenvalue. We emphasize that no symmetry assumption on the distribution of $\xi$ is needed for our results.

% *******************************************
\subsection{Largest eigenvalue}
\label{subsec:main-results}
% *******************************************

We begin with the subcritical regime, in which $\lambda_1 (\bfR_n)$ still converges in probability to $\lambda_+$. Define
\begin{equation*}
    \bar{F}_{\xi} (t) := \bbp \{\abs{\xi} > t\},
    \qquad
    \alpha_k := 2 + 2 / k,
    \qquad
    k \geq 2.
\end{equation*}
Here the exponent $\alpha_k$ arises from counting $k$-row collisions, each consisting of $k$ entries in distinct rows of a common column, all exceeding $\sqrt{n}$ in absolute value. In fact, the expected number of such collisions is
\begin{equation*}
    n \binom{p_n}{k} \bar{F}_{\xi} (\sqrt{n})^k
    = \pars*{\frac{\phi^k}{k!} + o (1)}
    \braks[\big]{n^{\alpha_k / 2} \bar{F}_{\xi} (\sqrt{n})}^k.
\end{equation*}
We call the tail of $\xi$ \emph{subcritical}, \emph{supercritical}, or \emph{critical} at order $k$ if, as $t \to \infty$, the quantity $t^{\alpha_k} \bar{F}_{\xi} (t)$ tends to zero, infinity, or a finite positive constant, respectively.

The tail condition for convergence to $\lambda_+$ depends on the minimum number of colliding rows needed to produce a separated outlier. We call this number the \emph{first collision order} and define it, together with the corresponding critical tail exponent, by
\begin{equation}
    k_* (\phi) 
    := \lfloor \sqrt{\phi} \rfloor + 2,
    \qquad
    \alpha_* (\phi) 
    := \alpha_{k_* (\phi)}
    = 2 + \frac{2}{k_* (\phi)}.
    \label{def:first-collision-order}
\end{equation}

% *******************************************
\begin{theorem}[Subcritical regime]
\label{thm:upper-edge-convergence}
% *******************************************
Suppose \eqref{eqn:basic-assumption} holds and
\begin{equation}
    t^{\alpha_* (\phi)} \bar{F}_{\xi} (t) \to 0,
    \qquad t \to \infty.
    \label{cond:upper-tail-subcritical}
\end{equation}
Then, as $n \to \infty$,
\begin{equation}
    \lambda_1 (\bfR_n) \pconv \lambda_+ (\phi) = (1 + \sqrt{\phi})^2.
    \label{eqn:upper-edge-convergence}
\end{equation}
\end{theorem}

A distinctive feature of this subcritical result is the dependence of the tail condition \eqref{cond:upper-tail-subcritical} on the aspect ratio $\phi$. The table below lists the first few values of $k_* (\phi)$ and $\alpha_* (\phi)$.
% *******************************************
\begin{center}
\renewcommand{\arraystretch}{1.2}
\begin{tabular}{c c c}
\hline
Aspect ratio & First collision order $k_* (\phi)$ & Critical exponent $\alpha_* (\phi)$\\
\hline
    $0 < \phi < 1$ & $2$ & $3$\\
    $1 \leq \phi < 4$ & $3$ & $8 / 3$\\
    $4 \leq \phi < 9$ & $4$ & $5 / 2$\\
\hline
\end{tabular}
\end{center}
% *******************************************
For every $\phi > 0$, the subcritical condition \eqref{cond:upper-tail-subcritical} is weaker than the moment assumption $\bbe \abs{\xi}^{\alpha_* (\phi)} < \infty$. For example, a distribution with tail decaying as $t^{-\alpha_* (\phi)} / \log t$ satisfies this condition, even though $\bbe \abs{\xi}^{\alpha_* (\phi)} = \infty$.

We next turn to the regime in which $\lambda_1 (\bfR_n)$ converges to a deterministic level above $\lambda_+$. This occurs when the tail is supercritical at some order $k \geq k_* (\phi)$ while subcritical at the next order $k + 1$. To specify the limit, we introduce the outlier map
\begin{equation}
    \theta_\phi (x)
    := x \pars*{1 + \frac{\phi}{x - 1}},
    \qquad x \in (0, 1-\sqrt{\phi}] \cup [1 + \sqrt{\phi}, \infty).
    \label{def:upper-outlier-map}
\end{equation}
In Johnstone's spiked population model \cite{johnstoneDistributionLargestEigenvalue2001}, this map relates a population spike to the limiting location of its sample outlier in the supercritical regime of the BBP transition \cite{baikPhaseTransitionLargest2005,baikEigenvaluesLargeSample2006}. The map $\theta_\phi$ equals $\lambda_+$ at $1 + \sqrt{\phi}$ and is strictly increasing above this threshold. In our context, the number $k$ of colliding rows plays the role of the population spike, yielding the deterministic limit $\theta_\phi (k)$ in the following result. Note that this also explains our definition of $k_{*} (\phi)$ in \eqref{def:first-collision-order} as it is the smallest integer that exceeds the threshold $1 + \sqrt{\phi}$.

% *******************************************
\begin{theorem}[$k$-supercritical regime]
\label{thm:collision-layer}
% *******************************************
Suppose that \eqref{eqn:basic-assumption} holds and
\begin{equation}
    t^{\alpha_k} \bar{F}_{\xi} (t) \to \infty,
    \qquad
    t^{\alpha_{k + 1}} \bar{F}_{\xi} (t) \to 0,
    \qquad t \to \infty
    \label{cond:deterministic-layer}
\end{equation}
for an fixed integer $k \geq k_* (\phi)$. Then, as $n \to \infty$,
\begin{equation}
    \lambda_1 (\bfR_n) \pconv \theta_\phi (k)
    = k \pars*{1 + \frac{\phi}{k - 1}}.
    \label{eqn:deterministic-layer-limit}
\end{equation}
\end{theorem}

As a consequence of Theorem \ref{thm:collision-layer}, 
\begin{alignat*}{3}
    0 < {} & \phi < 1, \qquad
    & t^3 \bar{F}_{\xi} (t) & \to \infty, \qquad
    t^{8 / 3} \bar{F}_{\xi} (t) \to 0
    & \qquad \Longrightarrow \qquad
    \lambda_1 (\bfR_n) & \pconv 2 (1 + \phi), \\
    1 \leq {} & \phi < 4, \qquad
    & t^{8 / 3} \bar{F}_{\xi} (t) & \to \infty, \qquad
    t^{5 / 2} \bar{F}_{\xi} (t) \to 0
    & \qquad \Longrightarrow \qquad
    \lambda_1 (\bfR_n) & \pconv 3 (1 + \phi / 2).
\end{alignat*}

On the other hand, the next result shows that the largest eigenvalue diverges in probability when the tail of $\xi$ is supercritical at every fixed order $k \geq k_* (\phi)$.

% *******************************************
\begin{theorem}[Ultra-supercritical regime]
\label{thm:all-orders-divergence}
% *******************************************
Suppose that \eqref{eqn:basic-assumption} holds and
\begin{equation}
    t^{\alpha_k} \bar{F}_{\xi} (t) \to \infty,
    \qquad t \to \infty,
    \qquad \forall \, k \geq k_* (\phi).
    \label{cond:all-orders-supercritical}
\end{equation}
Then, as $n \to \infty$,
\begin{equation}
    \lambda_1 (\bfR_n) \pconv \infty.
    \label{eqn:divergence-max-eig}
\end{equation}
\end{theorem}

The ultra-supercritical condition \eqref{cond:all-orders-supercritical} can hold even when $\bbe \abs{\xi}^2 = 1$. For example, a centered distribution with $\bar{F}_{\xi} (t) \asymp {t^{-2} (\log t)^{-2}}$ has a finite second moment and satisfies this condition. Thus finite variance alone does not ensure boundedness of $\lambda_1 (\bfR_n)$.

Figure \ref{fig:phase-diagram} summarizes how the tail of $\xi$ and the aspect ratio $\phi$ determine the first-order behavior of $\lambda_1 (\bfR_n)$. For this illustration, we consider regularly varying tails of the form
\begin{equation}
    \bar{F}_{\xi} (t) = t^{-\alpha} \ell (t),
    \qquad t \to \infty,
    \qquad \alpha > 2,
    \label{cond:regular-variation}
\end{equation}
where $\ell$ is slowly varying at infinity and $\alpha > 2$ ensures finite variance. The unhatched region in Figure \ref{fig:phase-diagram} represents the subcritical regime $\alpha > \alpha_* (\phi)$, in which Theorem \ref{thm:upper-edge-convergence} gives convergence to $\lambda_+$. Its staircase boundary reflects the increase in the first collision order $k_* (\phi)$ at $\phi = 1, 4, 9, \ldots$ and the corresponding decrease in the critical exponent $\alpha_* (\phi)$.

The diagonally hatched bands indicate the $k$-supercritical regimes. The interior of the $k$-th band corresponds to $\alpha_{k + 1} < \alpha < \alpha_k$ and $0 < \phi < (k - 1)^2$, and Theorem \ref{thm:collision-layer} gives convergence to $\theta_\phi (k)$. On the right vertical boundary $\phi = (k - 1)^2$, this outlier level coincides with the upper MP edge, $\theta_\phi (k) = \lambda_+$. Since $k_* (\phi) = k + 1$ on this boundary, Theorem \ref{thm:upper-edge-convergence} gives convergence to $\lambda_+$ for $\alpha_{k + 1} < \alpha \leq \alpha_k$. Thus each right vertical boundary belongs to the subcritical regime, including the upper endpoint marked by an open circle but excluding the lower endpoint marked by a filled circle. The lower endpoint, at $\alpha = \alpha_{k + 1}$, lies in the critical regime of order $k + 1$ described below.

\begin{figure}[tb]
    \centering
    \includegraphics[width=0.9\textwidth]{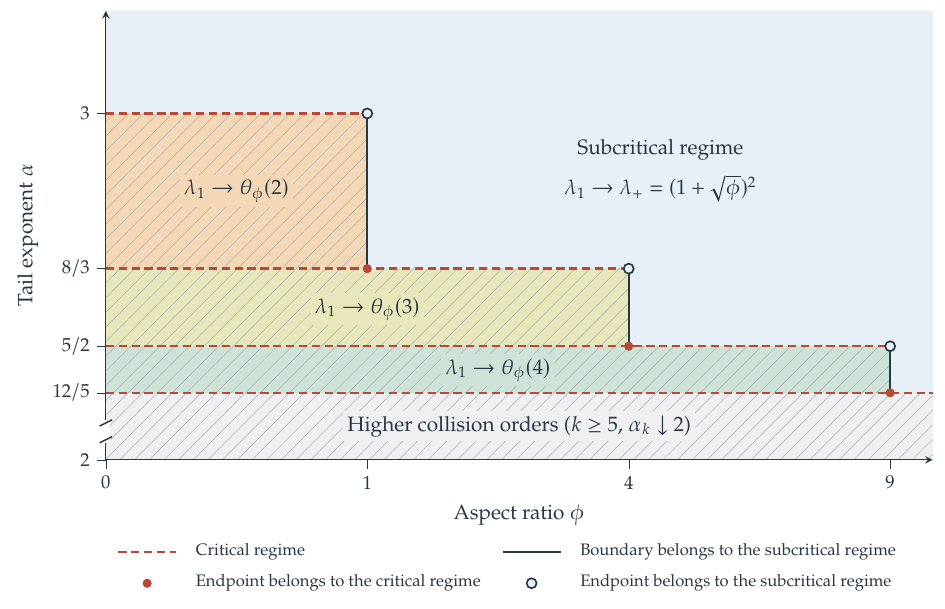}
    \caption{Phase diagram for the largest eigenvalue}
    \label{fig:phase-diagram}
\end{figure}

% *******************************************
\begin{remark}
% *******************************************
A related staircase phenomenon appears in the spectral deviations of random regular graphs, as recently established in \cite{chenNewApproachStrong2026}. Fix an integer $d \geq 2$, and let $\bar{\bfA}_n$ be the adjacency matrix of a random $2d$-regular graph on $n$ vertices in the permutation model. Write $\bfA_n = \bar{\bfA}_n |_{\{\mathbf{1}_n\}^\perp}$, so that $\norm{\bfA_n}$ gives the largest absolute value among the nontrivial eigenvalues of $\bar{\bfA}_n$. Set $m_* = \lfloor (\sqrt{2d - 1} + 1)/2 \rfloor$ and define 
\begin{equation*}
    \rho (m_*) = 2\sqrt{2d - 1},
    \qquad
    \rho (m) = 2m - 1 + \frac{2d - 1}{2m - 1},
    \qquad
    m_* < m \leq d.
\end{equation*}
For each integer $m$ with $m_* \leq m \leq d - 1$ and each fixed $\varepsilon \in (0, \rho (m + 1) - \rho (m))$, the staircase theorem \cite[Theorem~3.6]{chenNewApproachStrong2026} yields the deviation estimate
\begin{equation}
    \bbp \{\norm{\bfA_n} \geq \rho (m) + \varepsilon\}
    \asymp_{d,\varepsilon} n^{-m},
    \qquad n \to \infty.
    \label{eqn:staircase-theorem}
\end{equation}
Thus, the polynomial decay exponent increases in steps as the spectral threshold rises.

The analogy of \eqref{eqn:staircase-theorem} with our results lies in the local structures that produce distinct outlier levels. In our setting, $k$-row collisions of large entries form stars in the bipartite large-entry graph $\caG$, with a column vertex at the center and $k$ row vertices as leaves, which produce the levels $\theta_\phi (k)$. See Section \ref{sec:collisions} for illustration of this mechanism. For random regular graphs, vertices with $m$ self-loops can produce outliers near $\rho (m)$ for $m_* < m \leq d$, as originally shown by Friedman \cite{friedmanProofAlonsSecond2008}. In both models, the levels are images under the respective outlier maps. The different forms of these maps reflect the different limiting bulk spectra. Nevertheless, the two staircases  describe different phenomena. Our results describe changes in the first-order behavior of the largest eigenvalue as the entry tail and aspect ratio vary, whereas \eqref{eqn:staircase-theorem} concerns rare event probabilities as the spectral threshold varies.
\end{remark}

We now turn to the critical regime, where transitions between the phases described above occur. Fix an integer $k \geq k_* (\phi)$ and assume the critical tail condition
\begin{equation}
    t^{\alpha_k} \bar{F}_{\xi} (t) \to \omega \in (0, \infty),
    \qquad t \to \infty.
    \label{cond:critical-collision}
\end{equation}
In Figure \ref{fig:phase-diagram}, this condition corresponds to the dashed horizontal segment $\alpha = \alpha_k$ with $0 < \phi < (k - 1)^2$. We first define the thresholds
\begin{equation}
    \beta_\phi (k - 1) := (1 + \sqrt{\phi}) \vee (k - 1),
    \qquad
    \vartheta_\phi (k - 1)
    := \theta_\phi (\beta_\phi (k - 1)).
    \label{def:critical-baseline}
\end{equation}
In particular, the spectral baseline $\vartheta_\phi (k - 1)$ equals the upper MP edge $\lambda_+$ at the first critical order $k = k_* (\phi)$, while it equals the level $\theta_\phi (k - 1)$ at higher orders $k > k_* (\phi)$. Under \eqref{cond:critical-collision}, the eigenvalues of $\bfR_n$ above this baseline have a Poisson limit. Specifically, define the following point process on $(\vartheta_\phi (k - 1), \infty)$,
\begin{equation}
    \scN_{k, n}^{+}
    := \sum\nolimits_{i = 1}^{p_n}
    \bbone \{\lambda_i (\bfR_n) > \vartheta_\phi (k - 1)\}
    \, \delta_{\lambda_i (\bfR_n)}.
    \label{def:upper-spectral-process}
\end{equation}

To describe the limiting process, we introduce
\begin{subequations}
\begin{alignat}{2}
    \varphi (a) & := \frac{a}{\sqrt{1 + a^2}},
    & \qquad 
    a & \in \bbr_+, 
    \label{def:star-amplitude-single} \\
    \varphi (\bfa) & := (\varphi (a_1), \ldots, \varphi (a_k))^\top,
    & \qquad
    \bfa & = (a_r)_{r=1}^k \in \bbr_+^k.
    \label{def:star-amplitude-normalization}
\end{alignat}    
\end{subequations}
For a $k$-row collision, define the associated correlation matrix and its largest eigenvalue by
\begin{equation}
    \bfM_k (\bfa)
    := \diag \pars[\bigg]{\frac{1}{1 + a_1^2}, \ldots, \frac{1}{1 + a_k^2}}
    + \varphi (\bfa) \varphi (\bfa)^\top,
    \qquad
    \gamma_k^{+} (\bfa)
    := \lambda_1 (\bfM_k (\bfa)).
    \label{def:star-Gram}
\end{equation}
The matrix $\bfM_k (\bfa)$ has unit diagonal. Under \eqref{cond:critical-collision}, with probability tending to one there is no collision involving more than $k$ rows, while those involving fewer than $k$ rows cannot produce limiting outliers above $\vartheta_\phi (k - 1)$. See Lemmas \ref{lem:unified-graph} and \ref{lem:collision-components} below. Thus the limiting outliers must come from collisions involving exactly $k$ rows. For such a collision with normalized entry magnitudes $\bfa \in \bbr_+^k$, the matrix $\bfM_k (\bfa)$ approximates the corresponding principal submatrix of $\bfR_n$ up to conjugation by a diagonal sign matrix that does not affect the eigenvalues. When $\gamma_k^{+} (\bfa) > \beta_\phi (k - 1)$, the resulting outlier has asymptotic location $\theta_\phi (\gamma_k^{+} (\bfa))$, which lies above the spectral baseline $\vartheta_\phi (k - 1)$.

Finally, to describe the limiting distribution of the normalized entry magnitudes $\bfa \in \bbr_+^k$ associated with a $k$-row collision, introduce the tail measure
\begin{equation}
    \varrho_{k, \omega} (\rmd t)
    := \omega \alpha_k t^{-\alpha_k - 1} \, \rmd t,
    \qquad t > 0.
    \label{def:critical-tail-measure}
\end{equation}
Recall that $\PRM (\Lambda)$ denotes the law of a Poisson random measure with intensity $\Lambda$.

% *******************************************
\begin{theorem}[Poisson limit for upper spectral spikes]
\label{thm:upper-poisson}
% *******************************************
Assume \eqref{eqn:basic-assumption} and suppose that \eqref{cond:critical-collision} holds for a fixed integer $k \geq k_* (\phi)$. Then, in the vague topology on $(\vartheta_\phi (k - 1), \infty)$,
\begin{equation}
    \scN_{k, n}^{+} \Rightarrow \scN_k^{+},
    \qquad
    \scN_k^{+} \sim \PRM (\Lambda_{k, \phi, \omega}^{+}).
    \label{eqn:upper-spectral-poisson}
\end{equation}
For every Borel set $B \subset (\vartheta_\phi (k - 1), \infty)$, the limiting intensity is 
\begin{equation}
    \Lambda_{k, \phi, \omega}^{+} (B)
    := \frac{\phi^k}{k!}
    \int_{\bbr_+^k}
    \bbone \curls[\big]{
    \gamma_k^{+} (\bfa) > \beta_\phi (k - 1), \,
    \theta_\phi (\gamma_k^{+} (\bfa)) \in B} \,
    \varrho_{k, \omega}^{\otimes k} (\rmd \bfa).
    \label{def:critical-spectral-intensity}
\end{equation}
The intensity is atomless and has finite positive total mass,
\begin{equation}
    \sup\nolimits_{\lambda >\vartheta_\phi (k - 1)} 
    \Lambda_{k, \phi, \omega}^{+} (\{\lambda\}) = 0,
    \qquad
    \Lambda_{k, \phi, \omega}^{+} ((\vartheta_\phi (k - 1), \infty)) \in (0, \infty)
\end{equation}
Moreover, for every fixed $\lambda > \vartheta_\phi (k - 1)$,
\begin{equation}
    \abs[\big]{\{i \in \dbraks{p_n} : \lambda_i (\bfR_n) > \lambda\}}
    \Rightarrow \operatorname{Poisson}
    \pars[\big]{\Lambda_{k, \phi, \omega}^{+} ((\lambda, \infty))}.
    \label{eqn:fixed-gap-poisson}
\end{equation}
\end{theorem}

From this Poisson limit we can deduce the limiting distribution of $\lambda_1 (\bfR_n)$ in the critical regime.

% *******************************************
\begin{theorem}[$k$-critical regime]
\label{thm:critical-maximum}
% *******************************************
Assume \eqref{eqn:basic-assumption} and suppose that \eqref{cond:critical-collision} holds for a fixed integer $k \geq k_* (\phi)$. Then
\begin{equation}
    \lambda_1 (\bfR_n) \Rightarrow H_{k, \phi, \omega}^{+},
    \label{eqn:weak-convergence-critical}
\end{equation}
where the distribution function $H_{k, \phi, \omega}^{+}$ is given by
\begin{equation}
    H_{k, \phi, \omega}^{+} (\lambda)
    := \begin{cases}
        0, & \quad \lambda < \vartheta_\phi (k - 1), \\
        \exp \curls[\big]{-\Lambda_{k, \phi, \omega}^{+} ((\lambda, \infty))},
        & \quad \vartheta_\phi (k - 1) \leq \lambda < \theta_\phi (k), \\
        1, & \quad \lambda \geq \theta_\phi (k).
    \end{cases}
    \label{eqn:critical-maximum-cdf}
\end{equation}
The distribution function $H_{k, \phi, \omega}^{+}$ is continuous except at the spectral baseline $\vartheta_\phi (k - 1)$, where the limiting law has an atom of mass $H_{k, \phi, \omega}^{+} (\vartheta_\phi (k - 1)) \in (0,1)$.
\end{theorem}

The limiting law $H_{k, \phi, \omega}^{+}$ in Theorem \ref{thm:critical-maximum} assigns probability one to the interval $[\vartheta_\phi (k - 1), \theta_\phi (k))$. It therefore connects the deterministic limits of the phases on either side of the critical line. The mass of the atom at the baseline $\vartheta_\phi (k - 1)$ is the probability that the limiting Poisson process $\scN_{k}^{+}$ has no points.

% % *******************************************
% \begin{remark}[$2$-critical regime]
% % *******************************************
% Although $\gamma_k^{+}$ generally has no simple closed-form expression, it is explicit when $k = 2$. In this critical regime, $0 < \phi < 1$, and the matrix $\bfM_{2} (\bfa)$ from \eqref{def:star-Gram} has unit diagonal and off-diagonal entry
% \begin{equation*}
%     \varphi (a_1) \varphi (a_2)
%     = \frac{a_1 a_2}{\sqrt{(1 + a_1^2) (1 + a_2^2)}}.
% \end{equation*}
% Thus $\gamma_2^{+} (a_1, a_2) = 1 + \varphi (a_1) \varphi (a_2)$. Substituting this expression into \eqref{def:critical-spectral-intensity} gives
% \begin{equation*}
%     \Lambda_{2, \phi, \omega}^{+} (B)
%     = \frac{\phi^2}{2} \int_0^\infty \! \int_0^\infty
%     \bbone \{\varphi (a_1) \varphi (a_2) > \sqrt{\phi}, \, \theta_\phi (\gamma_2^{+} (a_1, a_2)) \in B\}
%     \, \varrho_{2, \omega} (\rmd a_1) 
%     \varrho_{2, \omega} (\rmd a_2).
% \end{equation*}
% Here $\varrho_{2, \omega} (\rmd a) = 3 \omega a^{-4} \, \rmd a$. 
% \end{remark}

% *******************************************
\subsection{Smallest nonzero eigenvalue}
\label{subsec:lower-main-results}
% *******************************************

The behavior of the smallest nonzero eigenvalue $\lambda_{p_n \wedge n} (\bfR_n)$ differs between the tall regime $\phi > 1$ and the wide regime $0 < \phi < 1$. For the almost sure convergence result below, we construct the matrices $\bfX_n$ from a single infinite array of independent copies of $\xi$, as in \cite{tikhomirovLimitSmallestSingular2015}.

% *******************************************
\begin{theorem}[Tall regime]
\label{thm:tall}
% *******************************************
Suppose \eqref{eqn:basic-assumption} holds with $\phi > 1$. Let $(x_{i \mu})_{i, \mu \geq 1}$ be an infinite array of independent copies of $\xi$, and, for each $n$, let $\bfX_n$ be the submatrix formed by its first $p_n$ rows and $n$ columns. Then, as $n \to \infty$,
\begin{equation}
    \lambda_n (\bfR_n) \to \lambda_- = (\sqrt{\phi} - 1)^2,
    \qquad \text{a.s.}
    \label{eqn:tall-corr}
\end{equation}
\end{theorem}

Thus, in the tall regime, finite variance alone ensures almost sure convergence to the lower MP edge $\lambda_-$, in agreement with Tikhomirov's result \cite{tikhomirovLimitSmallestSingular2015} for the sample covariance matrix. 

In the wide regime $0 < \phi < 1$, we have $k_* (\phi) = 2$ and $\alpha_2 = 3$. The three regimes below correspond to the tail of $\xi$ being subcritical, supercritical, or critical at order $2$, respectively.

% *******************************************
\begin{theorem}[Subcritical wide regime]
\label{thm:wide-subcritical}
% *******************************************
Suppose \eqref{eqn:basic-assumption} holds with $0 < \phi < 1$ and
\begin{equation}
    t^3 \bar{F}_{\xi} (t) \to 0,
    \qquad t \to \infty.
    \label{cond:wide-tail-sub}
\end{equation}
Then, as $n \to \infty$,
\begin{equation}
    \lambda_{p_n} (\bfR_n) \pconv \lambda_- = (1 - \sqrt{\phi})^2.
    \label{eqn:wide-subcritical-limit}
\end{equation}
\end{theorem}

% *******************************************
\begin{theorem}[Supercritical wide regime]
\label{thm:wide-supercritical}
% *******************************************
Suppose \eqref{eqn:basic-assumption} holds with $0 < \phi < 1$ and
\begin{equation}
    t^3 \bar{F}_{\xi} (t) \to \infty,
    \qquad t \to \infty.
    \label{cond:wide-tail-super}
\end{equation}
Then, as $n \to \infty$,
\begin{equation}
    \lambda_{p_n} (\bfR_n) \pconv 0.
    \label{eqn:wide-corr}
\end{equation}
\end{theorem}

In particular, \eqref{eqn:basic-assumption} alone does not guarantee convergence to $\lambda_-$ in the wide regime. Unlike the upper edge, the lower edge has no hierarchy of positive deterministic outlier levels. The supercritical condition \eqref{cond:wide-tail-super} already forces the smallest eigenvalue to converge to zero.

Figure \ref{fig:lower-phase-diagram} depicts the phase diagram for the smallest nonzero eigenvalue under the regularly varying tail assumption \eqref{cond:regular-variation} with $\alpha > 2$. The hatched region represents the supercritical wide regime, while the unhatched regions indicate convergence to the lower MP edge $\lambda_-$. The critical wide regime introduced below lies on the dashed segment $\alpha = 3$, $0 < \phi < 1$. On the boundary $\phi = 1$, almost sure convergence of the ESD to the MP law already implies that the smallest nonzero eigenvalue converges to zero almost surely.

\begin{figure}[tb]
    \centering
    \includegraphics[width=0.9\textwidth]{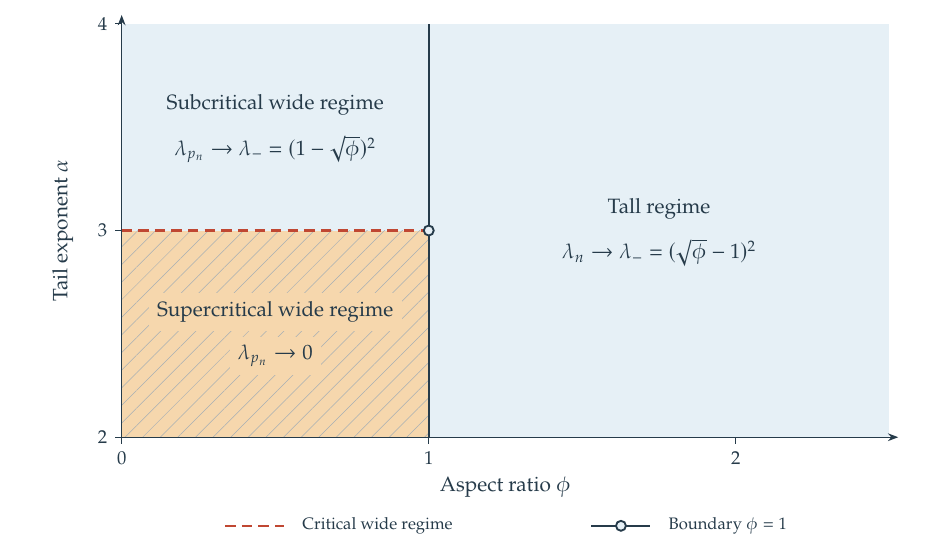}
    \caption{Phase diagram for the smallest nonzero eigenvalue}
    \label{fig:lower-phase-diagram}
\end{figure}

We now consider the critical wide regime
\begin{equation}
    t^3 \bar{F}_{\xi} (t) \to \omega \in (0, \infty),
    \qquad t \to \infty.
    \label{cond:wide-tail-critical}
\end{equation}
This is condition \eqref{cond:critical-collision} with $k = 2$. Define the point process on $(0, \lambda_-)$ by
\begin{equation}
    \scN_{2, n}^{-}
    := \sum\nolimits_{i = 1}^{p_n}
    \bbone \{0 < \lambda_i (\bfR_n) < \lambda_-\}
    \, \delta_{\lambda_i (\bfR_n)}.
    \label{def:lower-spectral-process}
\end{equation}

Recall the outlier map $\theta_\phi$ introduced in \eqref{def:upper-outlier-map}. When $0 < \phi < 1$, it is strictly increasing on $(0, 1 - \sqrt{\phi}]$, with $\theta_\phi (x) \to 0$ as $x \downarrow 0$ and $\theta_\phi (1 - \sqrt{\phi}) = \lambda_-$. In the spiked population model, a spike $x \in (0, 1 - \sqrt{\phi})$ produces a lower sample outlier with limiting location $\theta_\phi (x)$, whereas spikes in $[1 - \sqrt{\phi}, 1)$ do not produce outliers below the lower MP edge \cite{baikEigenvaluesLargeSample2006}.

The matrix $\bfM_2 (\bfa)$ defined in \eqref{def:star-Gram} is a symmetric $2 \times 2$ matrix with unit diagonal and off-diagonal entry $\varphi (a_1) \varphi (a_2)$. We denote its smaller eigenvalue by
\begin{equation}
    \gamma_2^{-} (a_1, a_2)
    := \lambda_2 (\bfM_2 (a_1, a_2))
    = 1 - \varphi (a_1) \varphi (a_2),
    \qquad a_1, a_2 > 0.
    \label{def:lower-star-eigenvalue}
\end{equation}
To produce a lower spectral outlier, we require $\gamma_2^{-} (a_1, a_2) < 1 - \sqrt{\phi}$. This is analogous to the condition $\gamma_k^{+} (\bfa) > \beta_\phi (k - 1)$ in \eqref{def:critical-spectral-intensity} for upper outliers above the spectral baseline. Also recall the tail measure \eqref{def:critical-tail-measure}.

% *******************************************
\begin{theorem}[Poisson limit for lower spectral spikes]
\label{thm:lower-poisson}
% *******************************************
Suppose that \eqref{eqn:basic-assumption} holds with $0 < \phi < 1$ and that \eqref{cond:wide-tail-critical} holds. Then, in the vague topology on $(0, \lambda_-)$,
\begin{equation}
    \scN_{2, n}^{-} \Rightarrow \scN_2^{-},
    \qquad
    \scN_2^{-} \sim \PRM (\Lambda_{2, \phi, \omega}^{-}).
    \label{eqn:lower-spectral-poisson}
\end{equation}
For every Borel set $B \subset (0, \lambda_-)$, the limiting intensity is
\begin{equation}
    \Lambda_{2, \phi, \omega}^{-} (B)
    := \frac{\phi^2}{2}
    \int_{\bbr_+^2}
    \bbone \curls[\big]{
    \gamma_2^{-} (\bfa) < 1 - \sqrt{\phi}, \,
    \theta_\phi (\gamma_2^{-} (\bfa)) \in B}
    \, \varrho_{2, \omega}^{\otimes 2} (\rmd \bfa).
    \label{def:lower-critical-spectral-intensity}
\end{equation}
The intensity is atomless and has finite positive total mass,
\begin{equation}
    \sup\nolimits_{\lambda \in (0, \lambda_-)} 
    \Lambda_{2, \phi, \omega}^{-} (\{\lambda\}) = 0,
    \qquad
    \Lambda_{2, \phi, \omega}^{-} ((0, \lambda_-)) \in (0, \infty).
    \label{eqn:lower-intensity-properties}
\end{equation}
Moreover, for every fixed $\lambda \in (0, \lambda_-)$,
\begin{equation}
    \abs[\big]{\{i \in \dbraks{p_n} : \lambda_i (\bfR_n) \leq \lambda\}}
    \Rightarrow \operatorname{Poisson}
    \pars[\big]{\Lambda_{2, \phi, \omega}^{-} ((0, \lambda])}.
    \label{eqn:lower-fixed-gap-poisson}
\end{equation}
\end{theorem}

% *******************************************
\begin{theorem}[Critical wide regime]
\label{thm:critical-minimum}
% *******************************************
Suppose that \eqref{eqn:basic-assumption} holds with $0 < \phi < 1$ and that \eqref{cond:wide-tail-critical} holds. Then
\begin{equation}
    \lambda_{p_n} (\bfR_n) \Rightarrow H_{2, \phi, \omega}^{-},
    \label{eqn:weak-convergence-critical-minimum}
\end{equation}
where the distribution function $H_{2, \phi, \omega}^{-}$ is given by
\begin{equation}
    H_{2, \phi, \omega}^{-} (\lambda)
    := \begin{cases}
        0, & \quad \lambda \leq 0, \\
        1 - \exp \curls[\big]{-\Lambda_{2, \phi, \omega}^{-} ((0, \lambda))},
        & \quad 0 < \lambda < \lambda_-, \\
        1, & \quad \lambda \geq \lambda_-.
    \end{cases}
    \label{eqn:critical-minimum-cdf}
\end{equation}
The distribution function $H_{2, \phi, \omega}^{-}$ is continuous except at the lower MP edge $\lambda_-$, where the limiting law has an atom of mass $1 - \lim_{\lambda \uparrow \lambda_-} H_{2, \phi, \omega}^{-} (\lambda) \in (0, 1)$.
\end{theorem}

The limiting law $H_{2, \phi, \omega}^{-}$ in Theorem \ref{thm:critical-minimum} assigns probability one to the interval $(0, \lambda_-]$. It can be characterized as the distribution of the smallest point of the limiting process $\scN_2^{-}$, with the value $\lambda_-$ assigned when the process is empty. As at the upper edge, the mass of the atom at $\lambda_-$ therefore equals the probability that $\scN_2^{-}$ has no points.

% *******************************************
\subsection{Pearson correlation matrix}
% *******************************************

All the preceding results can be extended to the Pearson correlation matrix $\bfR_n^{\circ}$ defined in \eqref{def:pearson-correlation}.

% *******************************************
\begin{theorem}[Pearson correlation matrix]
\label{cor:pearson-extension}
% *******************************************
Under their respective assumptions, all theorems in Sections \ref{subsec:main-results} and \ref{subsec:lower-main-results} remain valid with $\bfR_n^{\circ}$ in place of $\bfR_n$, except that $\lambda_n (\bfR_n)$ in Theorem \ref{thm:tall} is replaced by $\lambda_{n - 1} (\bfR_n^{\circ})$.
\end{theorem}

Note that in the tall regime $\phi > 1$, the centered matrices satisfy $\rank (\bfY_n^\circ) = \rank (\bfR_n^{\circ}) = n - 1$ eventually almost surely, and thus the smallest nonzero eigenvalue is then $\lambda_{n - 1} (\bfR_n^{\circ})$. Since Pearson correlations are invariant under shifts and positive rescaling, the mean-zero and unit-variance conditions in \eqref{eqn:basic-assumption} serve only as normalizations. The same conclusions therefore hold for any distribution of $\xi$ with finite positive variance, provided that the tail assumptions are imposed on the standardized variable $(\xi - \bbe \xi) / \sqrt{\operatorname{Var} (\xi)}$.

% *******************************************
\subsection{Proof strategy and organization}
\label{subsec:proof-organization}
% *******************************************

As mentioned, the main difficulty in analyzing $\bfR_n$ comes from row normalization, which couples the large and small entries within each row. In Section \ref{sec:truncation}, we develop a common spectral reduction for the largest eigenvalue when $\phi > 0$ and the smallest eigenvalue when $0 < \phi < 1$. We use truncation and resampling to separate typical rows from atypical rows containing large entries. The typical block supplies the MP bulk, while stabilization of the atypical row norms allows the resolvent analysis to isolate the contribution of the large entries. A Schur complement then relates eigenvalue counts on compact subsets of the two spectral gaps to those of an effective Gram matrix $\bfL_{\caB}$, which depends only on the large entries and thus is easier to handle. The outlier map $\theta_\phi$ relates its effective eigenvalues to the corresponding spectral outlier locations.

The effective Gram matrix $\bfL_{\caB}$ turns out to have a simple block diagonal structure, with blocks determined by the connected components of the large-entry graph $\caG$. In Section \ref{sec:collisions}, we analyze these blocks to establish the noncritical limits. The key observation is that, when every connected component of $\caG$ has at most $k$ edges, an eigenvalue of $\bfL_{\caB}$ above $k - 1$ can only arise from a star centered at a column vertex with $k$ row leaves. The corresponding block has largest eigenvalue $\gamma_k^{+} (\bfa)$, as defined in \eqref{def:star-Gram}. By spectral reduction from Section \ref{sec:truncation}, the eigenvalues of $\bfL_{\caB}$ above $1 + \sqrt{\phi}$ produce upper outliers, while those below $1 - \sqrt{\phi}$ produce lower outliers when $0 < \phi < 1$. Section \ref{sec:collisions} combines these ideas to establish the noncritical results in Theorems \ref{thm:upper-edge-convergence}, \ref{thm:collision-layer}, \ref{thm:all-orders-divergence}, \ref{thm:wide-subcritical}, and \ref{thm:wide-supercritical}. A separate proof of Theorem \ref{thm:tall} by adapting the method of Tikhomirov \cite{tikhomirovLimitSmallestSingular2015} to self-normalized matrices is defered to Appendix \ref{sec:tall}.

In the critical regimes, Section \ref{sec:critical-collisions} first establishes a Poisson limit for the point process of rescaled amplitude vectors associated with $k$-row collisions. The map $\gamma_k^{+}$ gives the largest eigenvalues of the corresponding blocks, while $\gamma_2^{-}$ gives their smallest eigenvalues when $k = 2$. The spectral reduction and the map $\theta_\phi$ then yield the upper and lower outlier processes in Theorems \ref{thm:upper-poisson} and \ref{thm:lower-poisson}. The extreme eigenvalue laws in Theorems \ref{thm:critical-maximum} and \ref{thm:critical-minimum} then follow by evaluating the probabilities that the corresponding limiting processes contain no points above or below a given level, respectively.

Finally, Appendix \ref{sec:pearson} extends the main results to Pearson correlation matrices.

% *******************************************
\subsection{Notation and conventions}
\label{subsec:notation}
% *******************************************

For $d \in \bbn$, write $\dbraks{d} := \{1, \ldots, d\}$. We use $\abs{A}$ to denote the cardinality of a finite set $A$. We generally use Roman letters $i, j$ for row indices in $\dbraks{p_n}$ and Greek letters $\mu, \nu$ for column indices in $\dbraks{n}$.

We use $\norm{\cdot}$ for the Euclidean norm of a vector or the induced operator norm of a matrix, and $\angles{\bfx, \bfy} := \bfx^\top \bfy$ for the standard inner product of real vectors. The symbol $\bfI$ denotes the identity matrix of the appropriate dimension, and $\bfe_i$ denotes the standard basis vector associated with index $i$ in the relevant coordinate space. We write $\mathbf{1}_d$ for the vector of $d$ ones. For $\caI \subset \dbraks{d}$, the coordinate projection matrix $\bfP_{\caI} \in \bbr^{\abs{\caI} \times d}$ restricts a vector to its coordinates in $\caI$. For a square matrix $\bfA = (A_{ij})_{i,j=1}^d$, write $\mathcal{D}[\bfA]$ and $\mathcal{D}_\perp[\bfA]$ for its diagonal and off-diagonal parts, respectively, so that
\begin{equation*}
    \bfA = \mathcal{D} [\bfA] + \mathcal{D}_\perp [\bfA],
    \qquad
    \mathcal{D} [\bfA] = \diag (A_{11}, \ldots, A_{dd}).
\end{equation*}
For real symmetric matrices, $\bfA \leq \bfB$ means that $\bfB - \bfA$ is positive semidefinite.

Throughout, we regard $n$ as the asymptotic parameter. Quantities may depend on $n$, and we suppress this dependence for convenience. Unless otherwise specified, limits are taken as $n \to \infty$. Constants $c, C > 0$ are independent of $n$ but may depend on the entry distribution and fixed parameters. Their values may change between occurrences unless explicitly fixed. For deterministic sequences $a_n$ and $b_n > 0$, the relations $a_n = O (b_n)$ and $a_n = o (b_n)$ mean that $\abs{a_n} / b_n$ is bounded and tends to zero, respectively. For positive sequences, $a_n \sim b_n$ means $a_n / b_n \to 1$, while $a_n \asymp b_n$ means $c \leq a_n / b_n \leq C$ for constants $0 < c < C$ independent of $n$. For random quantities, $X_n = O_{\bbp} (b_n)$ and $X_n = o_{\bbp} (b_n)$ mean that $\abs{X_n} / b_n$ is bounded in probability and converges to zero in probability, respectively.

% *******************************************
\section{Spectral reduction to the effective Gram matrix}
\label{sec:truncation}
% *******************************************

Throughout this section, we fix an integer $k \geq 2$ and assume, in addition to \eqref{eqn:basic-assumption}, the tail bound
\begin{equation}
    \bar{F}_{\xi} (t) \leq C_{\ref{eqn:analytic-tail-envelope}} t^{-\alpha},
    \qquad \forall \, t \geq 1,
    \qquad \alpha \equiv \alpha_k = 2 + {2} / {k} \in (2, 3].
    \label{eqn:analytic-tail-envelope}
\end{equation}
Here $C_{\ref{eqn:analytic-tail-envelope}} > 0$ is a constant that may depend on $k$. We define the truncation threshold by
\begin{equation}
    \fku_n = n^{1 / 2 - \tau},
    \qquad
    0 < \tau < \frac{1}{\alpha k (k + 1)}.
    \label{eqn:unified-truncation}
\end{equation}
The parameter $\tau$ may depend on $k$, but both are fixed independently of $n$. We suppress the dependence of auxiliary quantities and constants on these parameters in the notation. The threshold $\fku_n$ lies slightly below the natural scale $\sqrt{n}$ at which individual entries make a nonnegligible contribution to the row norms $\norm{\bfx_i}$. The estimates below rely on \eqref{eqn:basic-assumption} and the tail bound \eqref{eqn:analytic-tail-envelope}, with the truncation threshold chosen as in \eqref{eqn:unified-truncation}.

The main results of this section, Proposition \ref{prop:spectral-transfer} and Corollary \ref{cor:extremal-transfer}, relate the extreme eigenvalues of the sample correlation matrix $\bfR_n$ to those of the effective Gram matrix $\bfL_{\caB}$. In Section \ref{subsec:resampling-graph}, we construct the resampling representation \eqref{eqn:unified-resampling} for the data matrix $\bfX_n$. This construction decouples the small and large entries while keeping the column set occupied by large entries random after certain conditioning. We also introduce the large-entry graph $\caG$ and control its sparsity. Section \ref{subsec:truncated-moments} applies the local law for sparse covariance matrices \cite{hwangLocalLawTracy2019} to obtain spectral edge and resolvent estimates for the typical block, while Section \ref{subsec:atypical-decomposition} stabilizes the atypical row norms and separates the large entries from the truncated contributions. Section \ref{subsec:atypical-centered-resolvent} then controls the interaction of these components with the centered resolvent of the typical block. We combine these estimates in Section \ref{subsec:effective-gram} to identify the effective Gram matrix $\bfL_{\caB}$ and prove the needed resolvent approximation in Proposition \ref{prop:unified-resolvent}. Finally, Section \ref{subsec:spectral-transfer} uses a Schur complement to translate this approximation into comparisons between eigenvalue counts of $\bfR_n$ and $\bfL_{\caB}$.

We begin by the following notation used to streamline the probabilistic statements.

% *******************************************
\begin{definition}[stochastic domination]
% *******************************************
Let 
\begin{equation*}
    X = (X_n (t))_{n \geq 1, \, t \in \bbt_n},
    \qquad
    Y = (Y_n (t))_{n \geq 1, \, t \in \bbt_n}
\end{equation*}
be families of nonnegative random variables. The index sets $\bbt_n$ are deterministic and may depend on $n$. We write $X \prec Y$, uniformly in $t \in \bbt_n$, if for every $\eps > 0$ and $D > 0$ there exists $N_0 (\eps, D)$ such that
\begin{equation*}
    \sup\nolimits_{t \in \bbt_n} \bbp \{X_n (t) > n^\eps Y_n (t)\}
    \leq n^{-D},
    \qquad n \geq N_0 (\eps, D).
\end{equation*}
We also write $X_n = O_\prec (Y_n)$ if $\abs{X_n} \prec Y_n$. We also say a sequence of events $(\Omega_n)_{n \geq 1}$ holds with \emph{overwhelming probability} if, for every (large) $D > 0$, there exists $N_0 (D)$ such that 
\begin{equation*}
    \bbp (\Omega_n^c) \leq n^{-D},
    \qquad
    n \geq N_0 (D).
\end{equation*}
For conditional estimates, we use the same definitions with $\bbp$ replaced by the relevant conditional probability. These bounds hold uniformly over all conditioning realizations satisfying the restrictions specified in the statement.
\end{definition}

% *******************************************
\subsection{Resampling representation and the large-entry graph}
\label{subsec:resampling-graph}
% *******************************************

The probability that an entry exceeds the truncation threshold $\fku_n$ in absolute value is
\begin{equation}
    \fkq_n
    := \bar{F}_{\xi} (\fku_n)
    = \bbp \{ \abs{\xi} > \fku_n \}.
    \label{eqn:unified-large-entry-probability}
\end{equation}
We introduce the \emph{large-entry graph} $\caG = (\dbraks{p_n}, \dbraks{n}, \caE)$ as the bipartite graph whose two vertex sets consist of the row and column indices. An edge joins row $i$ to column $\mu$ precisely when the corresponding entry exceeds the threshold, $\abs{x_{i \mu}} > \fku_n$. In other words, the edge set is
\begin{equation}
    \caE := \curls[\big]{
    (i, \mu) \in \dbraks{p_n} \times \dbraks{n} :
    \abs{x_{i \mu}} > \fku_n}.
    \label{def:unified-large-entry-edges}
\end{equation}
For notational convenience, we suppress the dependence on $n$ in the graph and its associated index sets. The index sets of \emph{atypical} and \emph{typical} rows are defined, respectively, by
\begin{equation}
    \caB := \{i \in \dbraks{p_n} :
    \{i\} \times \dbraks{n} \cap \caE \neq \varnothing \},
    \qquad
    \caT := \dbraks{p_n} \setminus \caB.
\end{equation}
Therefore, an atypical row $i \in \caB$ contains at least one entry exceeding $\fku_n$ in absolute value, whereas a typical row $i \in \caT$ contains none. Define the neighbourhood $\caA_i$ of row $i$ in $\caG$ and the set $\caJ$ of occupied columns by
\begin{equation}
    \caA_i := \{\mu \in \dbraks{n} : (i, \mu) \in \caE\},
    \qquad
    \caJ := \bigcup\nolimits_{i \in \caB} \caA_i.
\label{eqn:unified-row-sets}
\end{equation}
Note that $\caB$ and $\caJ$ are the projections of the edge set $\caE$ onto the row and column vertex sets, respectively. For an index set $\caI \subset \dbraks{p_n}$, let $\bfP_{\caI} \in \bbr^{\abs{\caI} \times p_n}$ denote the corresponding coordinate projection matrix. Define
\begin{equation*}
    \bfX_{\caT} := \bfP_{\caT} \bfX_n, 
    \qquad
    \bfX_{\caB} := \bfP_{\caB} \bfX_n, 
    \qquad
    \bfY_{\caT} := \bfP_{\caT} \bfY_n, 
    \qquad
    \bfY_{\caB} := \bfP_{\caB} \bfY_n,
\end{equation*}
so that $\bfX_{\caT}$ and $\bfX_{\caB}$ consist of the typical and atypical rows of $\bfX_n$, respectively, while $\bfY_{\caT}$ and $\bfY_{\caB}$ are obtained from these submatrices by normalizing each row.

Our resolvent analysis requires conditioning on the large-entry pattern while leaving the actual occupied column set $\caJ$ random; see proof of Proposition \ref{prop:unified-centered-compressions} below. We achieve this by constructing a template graph $\bar{\caG}$ and relabelling its columns using an independent uniform permutation. This resampling representation preserves the law of $\bfX_n$ and makes the conditional structure explicit. Let $\zeta_n$ and $\psi_n$ have the laws of $\xi$ conditional on $\abs{\xi} \leq \fku_n$ and $\abs{\xi} > \fku_n$, respectively. Specifically, for every Borel set $A \subset \bbr$,
\begin{subequations} \label{eqn:alpha3-resampled-laws}
\begin{align}
    \bbp \{\zeta_n \in A\}
    = \frac{\bbp \{\xi \in A, \, \abs{\xi} \leq \fku_n\}}
    {\bbp \{\abs{\xi} \leq \fku_n\}},
    \label{eqn:alpha3-resampled-zeta} \\
    \bbp \{\psi_n \in A\}
    = \frac{\bbp \{\xi \in A, \, \abs{\xi} > \fku_n\}}
    {\bbp \{\abs{\xi} > \fku_n\}}.
    \label{eqn:alpha3-resampled-psi}
\end{align}
\end{subequations}
On an enlarged probability space, let
\begin{equation*}
    \bar{\chi}_{i \nu} \overset{\mathrm{i.i.d.}}{\sim} \mathrm{Bern} (\fkq_n),
    \qquad
    \bar{\psi}_{i \nu} \overset{\mathrm{i.i.d.}}{\sim} \caL (\psi_n),
    \qquad
    \zeta_{i \mu} \overset{\mathrm{i.i.d.}}{\sim} \caL (\zeta_n),
    \qquad
    \pi_n \sim \operatorname{Unif} (\caS_n)
\end{equation*}
be mutually independent, where $i \in \dbraks{p_n}$ and $\mu, \nu \in \dbraks{n}$. Here $\caL (\psi_n)$ and $\caL (\zeta_n)$ are the laws in \eqref{eqn:alpha3-resampled-laws}, and $\caS_n$ is the permutation group of $\dbraks{n}$. The resampling representation is
\begin{equation}
    x_{i \mu} = (1 - \chi_{i \mu}) \zeta_{i \mu}
       + \chi_{i \mu} \psi_{i \mu}, \qquad
    (\chi_{i \mu}, \psi_{i \mu})
    = (\bar{\chi}_{i \nu}, \bar{\psi}_{i \nu}),
    \qquad
    \nu = \pi_n^{-1} (\mu).
    \label{eqn:unified-resampling}
\end{equation}
For every fixed realization of $\pi_n$, the permuted pairs $(\chi_{i \mu}, \psi_{i \mu})$ remain i.i.d. and are independent of the truncated variables $\zeta_{i \mu}$. The entries in \eqref{eqn:unified-resampling} are therefore independent with law $\caL (\xi)$, and therefore this representation recovers the law of $\bfX_n$. Throughout this section we work on such representation. In particular,
\begin{equation*}
    \caE = \{(i, \mu) \in \dbraks{p_n} \times \dbraks{n} : \chi_{i \mu} = 1\}.
\end{equation*}

This type of Bernoulli resampling is standard in the modern analysis of heavy-tailed random matrices; see, e.g., \cite{aggarwalGOEStatisticsLevy2021,baoPhaseTransitionSmallest2025,hanDeformedFrechetLaw2025,liNecessarySufficientCondition2024}. Our construction \eqref{eqn:unified-resampling} additionally incorporates the independent permutation $\pi_n$ to keep the occupied column set $\caJ$ random under the conditioning $\bbp^{\bar{\caH}}$ introduced below. The unpermuted indicators $\bar{\chi}_{i \nu}$ determine the \emph{template large-entry graph} $\bar{\caG} = (\dbraks{p_n}, \dbraks{n}, \bar{\caE})$, whose edge set is
\begin{equation}
    \bar{\caE} := \{(i, \nu) \in \dbraks{p_n} \times \dbraks{n} :
    \bar{\chi}_{i \nu} = 1\}.
    \label{def:unified-template-edges}
\end{equation}
The actual graph $\caG$ is obtained by relabelling each template column $\nu$ of $\bar{\caG}$ as $\pi_n (\nu)$, while keeping all row labels fixed. In particular, $\abs{\caE} = \abs{\bar{\caE}}$. Using the template labels, define
\begin{equation*}
    \bar{\caA}_i := \{\nu \in \dbraks{n} : (i, \nu) \in \bar{\caE}\},
    \qquad
    \bar{\caJ} := \bigcup\nolimits_{i \in \caB} \bar{\caA}_i.
\end{equation*}
Then $\caA_i = \pi_n (\bar{\caA}_i)$ and $\caJ = \pi_n (\bar{\caJ})$, while the occupied row set $\caB$ is unchanged.

For the conditioning arguments, we introduce the $\sigma$-fields
\begin{equation*}
    \bar{\caH}_n := \sigma \pars[\big]{
    \{\bar{\chi}_{i \nu}\}_{i \in \dbraks{p_n}, \nu \in \dbraks{n}}},
    \qquad
    \caH_n := \sigma (\bar{\caH}_n, \pi_n).
\end{equation*}
To simplify the notation, we use the abbreviations
\begin{equation*}
    \bbe^{\bar{\caH}} [Z] \equiv \bbe [Z \mid \bar{\caH}_n],
    \qquad
    \bbp^{\bar{\caH}} (\Omega) \equiv \bbp (\Omega \, | \, \bar{\caH}_n), 
    \qquad
    \bbe^{{\caH}} [Z] \equiv \bbe [Z \mid {\caH}_n],
    \qquad
    \bbp^{\caH} (\Omega) \equiv \bbp (\Omega \, | \, {\caH}_n).    
\end{equation*}
Conditioning on these two $\sigma$-fields fixes the large-entry pattern in the template and actual column labels, respectively, while leaving the large-entry values independent with common law $\caL (\psi_n)$. In particular, the index sets $\caB$, $\caT$, and $\bar{\caJ}$ and the counts $\abs{\caE}$, $\abs{\caB}$ and $\abs{\caJ}=\abs{\bar{\caJ}}$ are $\bar{\caH}_n$-measurable. Conditional on $\bar{\caH}_n$, the set $\caJ = \pi_n (\bar{\caJ})$ is uniformly distributed over the subsets of $\dbraks{n}$ of cardinality $\abs{\bar{\caJ}}$ and is independent of both the large-entry values $(\bar{\psi}_{i \nu})$ in template labels and the truncated variables $({\zeta_{i \mu}})$. Such independence is important for our application of the random coordinate compression estimate in Lemma \ref{lemma:fixed-cardinality-random-compression}.

The following lemma, inspired by the sparsity analysis of extreme entries in Auffinger, Ben Arous, and P{\'e}ch{\'e} \cite{auffingerPoissonConvergenceLargest2009}, gives elementary bounds for the large-entry graph. Part \ref{item:unified-graph-edge-count} controls the number of large entries, while part \ref{item:unified-graph-components} shows that, with probability tending to one, every connected component has at most $k$ edges.

% *******************************************
\begin{lemma}[Large-entry graph]
\label{lem:unified-graph}
% *******************************************
The following hold for all sufficiently large $n$.
\begin{enumerate}[label = (\roman*)]
    \item \label{item:unified-graph-edge-count} The following event is $\bar{\caH}_n$-measurable and holds with overwhelming probability,
    \begin{equation}
        \Omega_n := \{\abs{\bar{\caE}} \leq n^{1 - \sigma / 2}\},
        \qquad
        \sigma \equiv \sigma_k = 1 / k - \alpha \tau.
        \label{def:unified-edge-event}
    \end{equation}
    \item \label{item:unified-graph-components} The event
    \begin{equation*}
        \tilde{\Omega}_n \equiv \tilde{\Omega}_{k, n} := \Omega_n \cap
        \{\text{every connected component of $\bar{\caG}$
        has at most $k$ edges}\}
    \end{equation*}
    is $\bar{\caH}_n$-measurable and satisfies
    \begin{equation}
        \bbp (\tilde{\Omega}_n) \geq 1 - O (n^{-(k + 1) \sigma + 1})
        = 1 - o(1).
        \label{eqn:unified-component-event-tail}
    \end{equation}
\end{enumerate}
\end{lemma}

% *******************************************
\begin{proof}[Proof of Lemma \ref{lem:unified-graph}]
% *******************************************
By the choice of $\tau$ in \eqref{eqn:unified-truncation} and the definition of $\sigma$ in \eqref{def:unified-edge-event},
\begin{equation*}
    1 / (k + 1) < \sigma_k < 1 / k \leq 1 / 2.
\end{equation*}
The $\bar{\caH}_n$-measurability of $\Omega_n$ is immediate. By the tail condition \eqref{eqn:analytic-tail-envelope},
\begin{equation}
    \fkq_n
    \leq C_{\ref{eqn:analytic-tail-envelope}} \fku_n^{-\alpha}
    = C_{\ref{eqn:analytic-tail-envelope}} n^{-1 - 1 / k + \alpha \tau}
    = C_{\ref{eqn:analytic-tail-envelope}} n^{-1 - \sigma}.
    \label{eqn:unified-edge-probability-bound}
\end{equation}
By independence of $\bar{\chi}_{i \nu}$, the number of large entries has a binomial distribution,
\begin{equation*}
    \abs{\bar{\caE}}
    = \sum\nolimits_{i = 1}^{p_n}
    \sum\nolimits_{\nu = 1}^n
    \bar{\chi}_{i \nu}
    \sim \operatorname{Binom} (p_n n, \fkq_n),
    \qquad
    \bbe \abs{\bar{\caE}}
    = p_n n \fkq_n
    \leq C n^{1 - \sigma},
\end{equation*}
where the expectation bound uses $p_n \asymp n$. With $t_n = n^{1 - \sigma / 2}$, a standard application of Chernoff's bound (see, e.g., \cite[Theorem 2.3.1]{vershyninHighdimensionalProbabilityIntroduction2018}) gives, for all sufficiently large $n$,
\begin{equation*}
    \bbp \{\abs{\bar{\caE}} > t_n\}
    \leq \pars{{e \bbe \abs{\bar{\caE}}} / {t_n}}^{t_n}
    \leq \exp (-c t_n)
    = \exp \pars{- c n^{1 - \sigma / 2}}.
\end{equation*}
Since $\sigma < 1 / 2$, this yields the claimed probability bound in part \ref{item:unified-graph-edge-count}.

We next bound the number of edges in the connected components of $\bar{\caG}$. Any connected component of $\bar{\caG}$ with at least $k + 1$ edges must contain a connected subgraph with exactly $k + 1$ edges. Let $\mathbb{G}_{k + 1}$ be the finite collection of isomorphism types of connected bipartite graphs with exactly $k + 1$ edges, where isomorphisms preserve the row and column vertex classes. For $\caC \in \mathbb{G}_{k + 1}$, write $\fkr (\caC)$ and $\fkc (\caC)$ for its numbers of row and column vertices. Then there are at most $(p_n)^{\fkr (\caC)} n^{\fkc (\caC)}$ embeddings of $\caC$ into the vertex set of $\bar{\caG}$, and all prescribed edges of any embedding are present with probability $\fkq_n^{k + 1}$ by independence of the edge indicators $\bar{\chi}_{i \nu}$. Since $\fkr (\caC) + \fkc (\caC) \leq k + 2$ by connectedness of $\caC$, a union bound gives
\begin{equation*}
\begin{split}
    & ~ \bbp \curls{\bar{\caG}
    \text{ has a connected component with at least $k + 1$ edges}} \\
    & \qquad \leq
    \sum\nolimits_{\caC \in \mathbb{G}_{k + 1}}
    p_n^{\fkr (\caC)} n^{\fkc (\caC)} \fkq_n^{k + 1}
    \leq C n^{k + 2} \fkq_n^{k + 1}
    \leq C n^{k + 2 - (k + 1) (1 + \sigma)}
    = C n^{- (k + 1) \sigma + 1}.
\end{split}
\end{equation*}
Here we used \eqref{eqn:unified-edge-probability-bound}. This probability tail is $o (1)$ since $\sigma > 1/(k+1)$. Combining this with part \ref{item:unified-graph-edge-count} proves \eqref{eqn:unified-component-event-tail}. The $\bar{\caH}_n$-measurability of the event $\tilde{\Omega}_n$ is again immediate.
\end{proof}

\begin{subequations} 
We record the following consequences of the scalar and matrix Bernstein inequalities in terms of stochastic domination as they are recurrently invoked. Let $\{x_a\}_{a = 1}^N$ be independent, centered real random variables. Suppose, for some deterministic $M_n, V_n \geq 0$,
\begin{equation*}
    \max\nolimits_{a \in \dbraks{N}} \abs{x_a} \leq M_n,
    \qquad
    \sum\nolimits_{a = 1}^N \bbe \abs{x_a}^2 \leq V_n,
\end{equation*}
Then, the scalar Bernstein inequality (see, e.g., \cite[Theorem 2.8.4]{vershyninHighdimensionalProbabilityIntroduction2018}) gives
\begin{equation}
    \abs[\Big]{\sum\nolimits_{a = 1}^N x_a}
    \prec M_n + \sqrt{V_n}.
    \label{eqn:scalar-Bernstein-domination}
\end{equation}
Similarly, let $\{\fkX_a\}_{a = 1}^N$ be independent, centered random matrices of size $d_1 \times d_2$, where $d_1 + d_2 \leq n^C$ for some constant $C > 0$. Suppose that
\begin{equation*}
    \max\nolimits_{a \in \dbraks{N}}
    \norm{\fkX_a} \leq M_n,
    \qquad
    \max \curls[\Big]{
    \norm[\Big]{\sum\nolimits_{a = 1}^N \bbe [\fkX_a \fkX_a^\top]},
    \norm[\Big]{\sum\nolimits_{a = 1}^N \bbe [\fkX_a^\top \fkX_a]}}
    \leq V_n.
\end{equation*}
Then, Tropp's matrix Bernstein inequality \cite[Theorem 1.6]{troppUserFriendlyTailBounds2012} gives
\begin{equation}
    \norm[\Big]{\sum\nolimits_{a = 1}^N \fkX_a}
    \prec M_n + \sqrt{V_n}.
    \label{eqn:matrix-Bernstein-domination}
\end{equation}
\end{subequations}

% *******************************************
\subsection{Typical block estimates}
\label{subsec:truncated-moments}
% *******************************************

The following moment estimates are needed to control the contributions of the truncated entries in both typical and atypical rows. 

% *******************************************
\begin{lemma}
\label{lem:unified-truncated-moments}
% *******************************************
There exists a constant $C > 0$ such that, for all sufficiently large $n$,
\begin{equation}
    \abs{\bbe \zeta_n} \leq C \fku_n^{1 - \alpha}, \qquad
    \abs[\big]{\bbe \abs{\zeta_n}^2 - 1} \leq C \fku_n^{2 - \alpha}, 
    \qquad
    \bbe \abs{\zeta_n}^4 \leq C \fku_n^{4 - \alpha}.
    \label{eqn:unified-low-moments}
\end{equation}
\end{lemma}

% *******************************************
\begin{proof}[Proof of Lemma \ref{lem:unified-truncated-moments}]
% *******************************************
By \eqref{eqn:unified-edge-probability-bound}, we have $\fkq_n = o (1)$. Thus $1 - \fkq_n \geq 1 / 2$ for all sufficiently large $n$. Since $\alpha > 2$, tail integration and \eqref{eqn:analytic-tail-envelope} gives, for $\ell = 1, 2$,
\begin{equation*}
    \bbe \braks[\big]{\abs{\xi}^{\ell} \bbone \{\abs{\xi} > \fku_n\}}
    = \fku_n^{\ell} \bar{F}_{\xi} (\fku_n)
    + \ell \int_{\fku_n}^{\infty} t^{\ell - 1} \bar{F}_{\xi} (t) \, \rmd t
    \leq C \fku_n^{\ell - \alpha}.
\end{equation*}
Combining these bounds with $\bbe \xi = 0$, $\bbe \abs{\xi}^2 = 1$, and the definition of $\zeta_n$ in \eqref{eqn:alpha3-resampled-zeta}, we obtain
\begin{align*}
    \abs{\bbe \zeta_n}
    & \leq \frac{\bbe [ \abs{\xi} \bbone \{\abs{\xi} > \fku_n\}]}
    {1 - \fkq_n}
    \leq C \fku_n^{1 - \alpha}, \\
    \abs[\big]{\bbe \abs{\zeta_n}^2 - 1}
    & \leq \frac{\bbe [\abs{\xi}^2 \bbone {\{\abs{\xi} > \fku_n \}}]}{1-\fkq_n}
    + \frac{\fkq_n}{1-\fkq_n}
    \leq C \fku_n^{2 - \alpha}.
\end{align*}
For the fourth moment, another application of tail integration yields
\begin{equation*}
    \bbe \abs{\zeta_n}^4
    = \frac{\bbe [\abs{\xi}^4 \bbone \{\abs{\xi} \leq \fku_n\}]}
    {1 - \fkq_n}
    \leq C
    \int_0^{\fku_n} t^3 \bar{F}_{\xi} (t) \, \rmd t
    \leq C \pars[\bigg]{1 + \int_1^{\fku_n} t^{3 - \alpha} \, \rmd t}
    \leq C \fku_n^{4 - \alpha},
\end{equation*}
where the last inequality uses $\alpha \leq 3$. This proves \eqref{eqn:unified-low-moments}.
\end{proof}

The next result controls the sums of squared truncated entries in all rows.

% *******************************************
\begin{lemma}
\label{lem:unified-square-sums}
% *******************************************
Uniformly over all realizations of $\bar{\caH}_n$ in $\Omega_n$, the following holds under $\bbp^{\bar{\caH}}$,
\begin{equation}
    \max_{i \in \dbraks{p_n}} \,
    \abs*{\sum\nolimits_{\mu \in \dbraks{n} \backslash \caA_i}
    \pars[\big]{\abs{\zeta_{i \mu}}^2 - \bbe \abs{\zeta_n}^2}}
    \prec \fku_n^2 = n^{1 - 2 \tau}.
    \label{eqn:unified-square-sums}
\end{equation}
The above estimate also holds under $\bbp^{\caH}$, uniformly over all realizations of $\caH_n$ in $\Omega_n$.
\end{lemma}

% *******************************************
\begin{proof}[Proof of Lemma \ref{lem:unified-square-sums}]
% *******************************************
We first prove \eqref{eqn:unified-square-sums} under $\bbp^{\caH}$. Fix a realization of $\caH_n$ in $\Omega_n$. Conditional on $\caH_n$, the variables $\{\abs{\zeta_{i \mu}}^2 - \bbe \abs{\zeta_n}^2\}_{\mu \in \dbraks{n} \backslash \caA_i}$ are independent and centered for each $i \in \dbraks{p_n}$. By \eqref{eqn:alpha3-resampled-zeta} and \eqref{eqn:unified-low-moments},
\begin{equation*}
    \abs[\big]{\abs{\zeta_{i \mu}}^2 - \bbe \abs{\zeta_n}^2}
    \leq C \fku_n^2,
    \qquad
    \operatorname{Var} \pars[\big]{\abs{\zeta_{i \mu}}^2 - \bbe \abs{\zeta_n}^2}
    \leq \bbe \abs{\zeta_n}^4 \leq C \fku_n^{4 - \alpha}.
\end{equation*}
Applying the Bernstein inequality \eqref{eqn:scalar-Bernstein-domination} with $M_n = C \fku_n^2$ and $V_n = C n \fku_n^{4 - \alpha}$ gives
\begin{equation*}
    \abs*{\sum\nolimits_{\mu \in \dbraks{n} \backslash \caA_i}
    \pars[\big]{\abs{\zeta_{i \mu}}^2 - \bbe \abs{\zeta_n}^2}}
    \prec \fku_n^2 + \sqrt{n} \fku_n^{2 - \alpha / 2}
    \prec \fku_n^2,
\end{equation*}
where the last bound follows from $\alpha \tau < 1 / k$. A union bound over $i \in \dbraks{p_n}$ then proves \eqref{eqn:unified-square-sums} under $\bbp^{\caH}$. Finally, since $\caH_n = \sigma (\bar{\caH}_n, \pi_n)$ and $\Omega_n$ is $\bar{\caH}_n$-measurable, averaging the uniform conditional estimate over the random permutation $\pi_n$ gives the same bound under $\bbp^{\bar{\caH}}$.
\end{proof}

Recall the MP distribution function $\tilde{F}_\phi$ from \eqref{def:MP-law}. Let $F_\phi$ denote the distribution function of the companion MP law with aspect ratio $\phi$, and let $\fkm_\phi$ be its Stieltjes transform,
\begin{equation}
    F_\phi (t) = \phi \tilde{F}_\phi (t) + (1 - \phi) \bbone \{t \geq 0\},
    \qquad
    \fkm_\phi (z)
    = \int_{\bbr} \frac{1}{x - z} \, \rmd F_\phi (x).
    \label{eqn:companion-MP-transform}
\end{equation}
Write $\varphi_n := \abs{\caT} / n$ for the aspect ratio of the typical block and $\bfR_{\caT} = \bfY_{\caT} \bfY_{\caT}^\top$ for its sample correlation matrix. Define the companion matrix and its resolvent by
\begin{equation}
    \bfQ_{\caT} := \bfY_{\caT}^\top \bfY_{\caT},
    \qquad
    \bfG_{\caT} (z) := (\bfQ_{\caT} - z \bfI)^{-1} = (G_{\mu \nu} (z))_{\mu, \nu = 1}^n.
    \label{eqn:unified-centered-resolvent}
\end{equation}
The next result provides the spectral and resolvent estimates for the typical block used below.

% *******************************************
\begin{proposition}[Typical block estimates]
\label{prop:unified-core}
% *******************************************
Fix a compact interval $\bbk_+ \subset (\lambda_+, \infty)$. For the lower edge estimates, assume $\phi \in (0, 1)$ and also fix a compact interval $\bbk_- \subset (0, \lambda_-)$. Uniformly over all realizations of $\bar{\caH}_n$ in $\Omega_n$, the following estimates hold under $\bbp^{\bar{\caH}}$.
\begin{enumerate}[label = (\roman*)]
    \item There exist constants $c_\pm \equiv c_\pm (\bbk_{\pm}) > 0$ such that, with overwhelming probability,
    \begin{equation}
        \lambda_1 (\bfQ_{\caT}) 
        \leq \min \bbk_+ - c_+,
        \qquad
        \lambda_{\abs{\caT}} (\bfQ_{\caT}) 
        \geq \max \bbk_- + c_-.
        \label{eqn:unified-core-trace}
    \end{equation}
    The corresponding resolvent bounds also hold with overwhelming probability,
    \begin{equation}
        \sup\nolimits_{z \in \bbk_\pm} \norm{\bfG_{\caT} (z)} \leq 1 / c_{\pm}.
        \label{eqn:unified-core-lower-gap}
    \end{equation}
    \item On either interval, we have the entrywise law
    \begin{equation}
        \sup\nolimits_{z \in \bbk_{\pm}}
        \max\nolimits_{\mu, \nu \in \dbraks{n}} 
        \abs[\big]{G_{\mu \nu} (z) - \delta_{\mu \nu} \fkm_{\varphi_n}\!(z)}
        \prec n^{-\tau}.
        \label{eqn:unified-core-entrywise}
    \end{equation}
\end{enumerate}
\end{proposition}

By Lemma \ref{lem:unified-graph} \ref{item:unified-graph-edge-count}, we have $\varphi_n \to \phi$ on $\Omega_n$, which holds with overwhelming probability. The proof of Proposition \ref{prop:unified-core} also gives the unconditional limits
\begin{equation}
    \lambda_1 (\bfR_{\caT}) 
    = \lambda_1 (\bfQ_{\caT})
    \pconv \lambda_+ (\phi),
    \qquad
    \lambda_{\abs{\caT}} (\bfR_{\caT})
    = \lambda_{\abs{\caT}} (\bfQ_{\caT})
    \pconv \lambda_- (\phi)
    \label{eqn:eig-typical-to-edge}
\end{equation}
where the second convergence requires $\phi \in (0, 1)$.

% *******************************************
\begin{proof}[Proof of Proposition \ref{prop:unified-core}]
% *******************************************
Fix a realization of $\bar{\caH}_n$ in $\Omega_n$. All stochastic estimates below hold under $\bbp^{\bar{\caH}}$, uniformly over such realizations. By Lemma \ref{lem:unified-graph}, $\abs{\caB} / n \leq n^{-\sigma / 2}$ on $\Omega_n$, and thus $\varphi_n \to \phi$ uniformly over these realizations. As mentioned, for the lower edge estimates, we assume $\phi \in (0, 1)$. In this case, we have $\varphi_n \in [\phi / 2, (\phi + 1) / 2] \subset (0, 1)$ for all sufficiently large $n$, uniformly over these realizations.

\shortpara{Standardization}: Define the standardized proxy by
\begin{equation}
    \tilde{\bfY}_{\caT}
    := (\tilde{y}_{i \mu})_{i \in \caT, \, \mu \in \dbraks{n}},
    \qquad
    \tilde{y}_{i \mu}
    := \frac{\zeta_{i \mu} - \bbe \zeta_n}
    {\sqrt{n \! \Var (\zeta_n)}}.
    \label{def:iid-proxy}
\end{equation}
The associated companion matrix and resolvent are
\begin{equation*}
    \tilde{\bfQ}_{\caT} := \tilde{\bfY}_{\caT}^\top \tilde{\bfY}_{\caT},
    \qquad
    \tilde{\bfG}_{\caT} (z) := (\tilde{\bfQ}_{\caT} - z \bfI)^{-1}
    = (\tilde{G}_{\mu \nu} (z))_{\mu, \nu = 1}^n.
\end{equation*}
By Lemma \ref{lem:unified-truncated-moments}, $\Var (\zeta_n) = 1 + O (\fku_n^{2 - \alpha}) \to 1$. Hence, conditional on $\bar{\caH}_n$, the entries of $\tilde{\bfY}_{\caT}$ are i.i.d. with mean zero, variance $1 / n$, and absolute value at most $C n^{-\tau}$. In particular,
\begin{equation}
    \bbe \abs{\tilde{y}_{i \mu}}^\ell
    \leq C^{\ell - 2} n^{-(\ell - 2) \tau}
    \bbe \abs{\tilde{y}_{i \mu}}^2
    = C^{\ell - 2} n^{-1 - (\ell - 2) \tau},
    \qquad \ell \geq 3.
    \label{eqn:unified-sparse-moments}
\end{equation}
The moment condition in \cite[Assumption 2.6]{hwangLocalLawTracy2019} therefore holds with sparsity parameter $n^\tau$, and \cite[Theorem 2.9]{hwangLocalLawTracy2019} gives the upper edge rigidity estimate
\begin{equation*}
    \abs[\big]{\lambda_{1} (\tilde{\bfQ}_{\caT})
    - L_+ (\varphi_n, \tilde{\kappa}_4) }
    \prec n^{-4\tau} + n^{-2/3}.
\end{equation*}
Here $L_+ (\varphi_n, \tilde{\kappa}_4)$ is the upper endpoint of the deterministic law in \cite[Theorem 2.7]{hwangLocalLawTracy2019}, and $\tilde{\kappa}_4$ is the common fourth cumulant of the entries $\tilde{y}_{i \mu}$. As explained in the proof of \cite[Lemma 2.5]{baoPhaseTransitionSmallest2025}, the argument for \cite[Theorem 2.9]{hwangLocalLawTracy2019} extends to the lower edge once a crude positive lower bound for $\lambda_{\abs{\caT}} (\tilde{\bfQ}_{\caT})$ is available. This bound can be deduced from \cite[Theorem 1]{tikhomirovSmallestSingularValue2016}. By \eqref{eqn:unified-low-moments}, $\sqrt{n} \tilde{y}_{i \mu}$ converges in distribution to the nondegenerate random variable $\xi$, so the required anticoncentration condition holds uniformly in $n$. Consequently, writing $L_- (\varphi_n, \tilde{\kappa}_4)$ for the corresponding lower endpoint, we obtain
\begin{equation*}
    \abs[\big]{\lambda_{\abs{\caT}} (\tilde{\bfQ}_{\caT})
    - L_- (\varphi_n, \tilde{\kappa}_4) }
    \prec n^{-4\tau} + n^{-2/3}.
\end{equation*}
The endpoint expansions in \cite[Equations (2.29) and (2.30)]{hwangLocalLawTracy2019}, together with the cumulant bound $\abs{\tilde{\kappa}_4} \leq C \bbe \abs{\tilde{y}_{i \mu}}^4 \leq C n^{-1 - 2 \tau}$ from \eqref{eqn:unified-sparse-moments}, lead to
\begin{equation*}
    L_{\pm} (\varphi_n, \tilde{\kappa}_4)
    = \lambda_{\pm} (\varphi_n)
    \pm \sqrt{\varphi_n} \lambda_{\pm} (\varphi_n)
    (n \tilde{\kappa}_4)
    + \bigO{n^{-4\tau}}
    = \lambda_{\pm} (\varphi_n)
    + \bigO{n^{-2\tau}}.
\end{equation*}
Since $\tau < 1 / 18$, combining the two rigidity estimates with the endpoint expansions gives
\begin{equation}
    \abs{\lambda_1 (\tilde{\bfQ}_{\caT}) - \lambda_+ (\varphi_n)}
    \prec n^{-2 \tau},
    \qquad
    \abs{\lambda_{\abs{\caT}} (\tilde{\bfQ}_{\caT}) - \lambda_- (\varphi_n)}
    \prec n^{-2 \tau}
    \label{eqn:alpha3-proxy-edge-rigidity}
\end{equation}
In particular, $\norm{\tilde{\bfY}_{\caT}} = O_\prec (1)$.

\shortpara{Comparison}: To transfer these estimates to $\bfQ_{\caT}$, we introduce the auxiliary block
\begin{equation*}
    \bar{\bfY}_{\caT}
    := (\bar{y}_{i \mu})_{i \in \caT, \mu \in \dbraks{n}},
    \qquad
    \bar{y}_{i \mu} := \frac{\zeta_{i \mu}}
    {\sqrt{n \bbe \abs{\zeta_n}^2}}.
\end{equation*}
The definitions of $\tilde{\bfY}_{\caT}$ and $\bar{\bfY}_{\caT}$ give
\begin{equation*}
    \bar{\bfY}_{\caT}
    = \sqrt{\frac{\Var (\zeta_n)}{\bbe \abs{\zeta_n}^2}}
    \tilde{\bfY}_{\caT}
    + \frac{\bbe \zeta_n}{\sqrt{n \bbe \abs{\zeta_n}^2}}
    \mathbf{1}_{\abs{\caT}} \mathbf{1}_n^\top.
\end{equation*}
Since $\abs{\caT} \leq p_n = O (n)$ and $\norm{\tilde{\bfY}_{\caT}} = O_\prec (1)$, Lemma \ref{lem:unified-truncated-moments} yields
\begin{equation*}
    \norm{\bar{\bfY}_{\caT} - \tilde{\bfY}_{\caT}}
    \prec \abs{\bbe \zeta_n}^2 + \sqrt{n} \abs{\bbe \zeta_n}
    \leq C \sqrt{n} \fku_n^{1 - \alpha}
    = C n^{-1 / k + (\alpha - 1) \tau}
    \leq C n^{-2 \tau},
\end{equation*}
where the last inequality uses $(\alpha + 1) \tau < 1 / k$. Note that for $i \in \caT$, we have $\caA_i = \varnothing$. Therefore, Lemma \ref{lem:unified-square-sums} and $\bbe \abs{\zeta_n}^2 \to 1$ give
\begin{equation*}
    \begin{aligned}
        \norm{\bfY_{\caT} - \bar{\bfY}_{\caT}}
        & \leq \norm{\bar{\bfY}_{\caT}}
        \max_{i \in \caT}
        \abs[\bigg]{\frac{\sqrt{n \bbe \abs{\zeta_n}^2}}{\norm{\bfx_i}} - 1}
        \prec n^{-2 \tau}.
    \end{aligned}
\end{equation*}
Combining the two preceding bounds with $\norm{\tilde{\bfY}_{\caT}} = O_\prec (1)$ gives
\begin{equation}
    \norm{\bfY_{\caT} - \tilde{\bfY}_{\caT}}
    + \norm{\bfQ_{\caT} - \tilde{\bfQ}_{\caT}}
    \prec n^{-2 \tau}.
    \label{eqn:unified-core-comparison}
\end{equation}
Together with \eqref{eqn:alpha3-proxy-edge-rigidity} and $\varphi_n \to \phi$, this proves \eqref{eqn:unified-core-trace} with
\begin{equation*}
    c_+ = (\min \bbk_+ - \lambda_+) / 2,
    \qquad
    c_- = (\lambda_- - \max \bbk_-) / 2 \wedge (\min \bbk_-)
\end{equation*}
For $\phi \in (0, 1)$, the second bound in \eqref{eqn:unified-core-trace} also forces $\rank (\bfQ_{\caT}) = \abs{\caT}$ on the corresponding event. In this case, the remaining $n - \abs{\caT}$ eigenvalues of $\bfQ_{\caT}$ are zero, and hence we can combine \eqref{eqn:unified-core-trace} with the spectral theorem to obtain \eqref{eqn:unified-core-lower-gap}.

\shortpara{Entrywise law}: It remains to prove \eqref{eqn:unified-core-entrywise}. Set $E_+ = 2 (\max \bbk_+)$ and $E_- = (\min \bbk_-) / 2$, and define the spectral domain
\begin{equation*}
    \bbd_n
    \equiv \bbd_n (E_-, E_+, \tau)
    := \curls{z = E + \rmi \eta:
    E_- \leq E \leq E_+, \,
    n^{-1 + \tau} \leq \eta \leq 1}.
\end{equation*}
Applying \cite[Proposition 2.13]{hwangLocalLawTracy2019} to the standardized proxy $\tilde{\bfY}_{\caT}$ yields
\begin{equation}
    \max\nolimits_{\mu, \nu \in \dbraks{n}}
    \abs[\big]{\tilde{G}_{\mu \nu} (z)
    - \delta_{\mu \nu} \fkm_{\varphi_n}\!(z)}
    \prec 
    n^{-\tau} + \sqrt{\dfrac{\Im \fkm_{\varphi_n}\!(z)}{n \eta}}
    + \dfrac{1}{n \eta}
    \label{eqn:alpha3-proxy-entrywise-local-law}
\end{equation}
uniformly over $z \in \bbd_n$. We now extend \eqref{eqn:alpha3-proxy-entrywise-local-law} to the real intervals $\bbk_\pm$. For $E \in \bbk_\pm$, set $\eta_n = n^{-1 + \tau}$, so that $E + \rmi \eta_n \in \bbd_n$. Since both intervals $\bbk_\pm$ stay a fixed positive distance from the support of the companion MP law with aspect ratio $\varphi_n$, we have
\begin{equation}
    \Im \fkm_{\varphi_n}\!(E + \rmi \eta_n)
    = \int_{-\infty}^\infty \frac{\eta_n}{(\lambda - E)^2 + \eta_n^2}
    \rmd F_{\varphi_n} (\lambda)
    = O (\eta_n)
    \label{eqn:imaginary-fkm-bound}
\end{equation}
uniformly for $E \in \bbk_\pm$. Substituting this bound into \eqref{eqn:alpha3-proxy-entrywise-local-law} gives
\begin{equation}
    \sup\nolimits_{E \in \bbk_\pm}
    \max\nolimits_{\mu, \nu \in \dbraks{n}}
    \abs[\big]{\tilde{G}_{\mu \nu} (E + \rmi \eta_n)
    - \delta_{\mu \nu} \fkm_{\varphi_n} (E + \rmi \eta_n)}
    \prec n^{-\tau}.
    \label{eqn:alpha3-proxy-entrywise-regularized}
\end{equation}
Here \eqref{eqn:alpha3-proxy-entrywise-local-law} yields the bound in \eqref{eqn:alpha3-proxy-entrywise-regularized} for each fixed $E \in \bbk_\pm$, with stochastic domination uniform in $E$. The supremum over $E \in \bbk_\pm$ follows from a standard $\varepsilon$-net argument: combining a union bound on a sufficiently fine deterministic net of $\bbk_\pm$ with Lipschitz continuity of the resolvent and its deterministic equivalent in $E$; see, e.g., \cite[Remark 2.7]{benaych-georgesLecturesLocalSemicircle2018}.

On the other hand, the rigidity estimate \eqref{eqn:alpha3-proxy-edge-rigidity} and the spectral theorem imply that
\begin{equation*}
    \sup\nolimits_{E \in \bbk_\pm} \!
    \curls[\big]{\norm[\big]{\tilde{\bfG}_{\caT} (E)
    - \Re \tilde{\bfG}_{\caT} (E + \rmi \eta_n)}
    + \abs[\big]{\fkm_{\varphi_n} \!(E)
    - \Re \fkm_{\varphi_n} \!(E + \rmi \eta_n)}}
    = \bigO{\eta_n^2}
\end{equation*}
holds with overwhelming probability. Combining this with \eqref{eqn:alpha3-proxy-entrywise-regularized} yields
\begin{equation*}
    \sup\nolimits_{E \in \bbk_\pm}
    \max\nolimits_{\mu, \nu \in \dbraks{n}} \,
    \abs[\big]{\tilde{G}_{\mu \nu} (E)
    - \delta_{\mu \nu} \fkm_{\varphi_n}\!(E)}
    \prec n^{-\tau}.
\end{equation*}
Finally, we use the resolvent identity
\begin{equation*}
    \bfG_{\caT} (z) - \tilde{\bfG}_{\caT} (z)
    = \bfG_{\caT} (z)
    (\tilde{\bfQ}_{\caT} - \bfQ_{\caT})
    \tilde{\bfG}_{\caT} (z),
\end{equation*}
together with \eqref{eqn:unified-core-comparison} to transfer this estimate to $\bfG_{\caT}$. This proves \eqref{eqn:unified-core-entrywise}.
\end{proof}

% *******************************************
\subsection{Decomposition of the atypical block}
\label{subsec:atypical-decomposition}
% *******************************************

We enlarge the conditioning $\sigma$-fields to include the large entries in the whole matrix as well as the truncated variables in the typical block,
\begin{equation*}
    \bar{\caF}_n = \sigma \pars[\big]{\bar{\caH}_n,
    \{\bar{\psi}_{i \nu}\}_{i \in \dbraks{p_n}, \nu \in \dbraks{n}},
    \{\zeta_{i \mu}\}_{i \in \caT, \mu \in \dbraks{n}}},
    \qquad
    \caF_n = \sigma (\bar{\caF}_n, \pi_n).
\end{equation*}
Conditional on $\bar{\caF}_n$, the remaining randomness in $\bfX_n$ under the resampling representation \eqref{eqn:unified-resampling} comes from the truncated variables $\{\zeta_{i \mu}\}_{i \in \caB, \mu \in \dbraks{n}}$ in atypical rows and the random permutation $\pi_n$ that maps template column labels to actual column labels. In particular, the index set $\caJ$ of occupied columns remains uniformly distributed over the subsets of $\dbraks{n}$ of cardinality $\abs{\bar{\caJ}}$. Conditioning on ${\caF}_n$ also fixes $\pi_n$. 

The large entries in atypical rows contribute at leading order to the row norms, and thus we need to retain them when stabilizing the normalizers. Specifically, for $i \in \caB$, define the \emph{stabilized normalizers} by
\begin{equation}
    \hbar_i^2 
    = n + \sum\nolimits_{\mu \in \caA_i} \abs{\psi_{i \mu}}^2
    = n + \sum\nolimits_{\nu \in \bar{\caA}_i} \abs{\bar{\psi}_{i \nu}}^2.
    \label{eqn:unified-frozen-denominator}
\end{equation}
These normalizers are $\bar{\caF}_n$-measurable. Define the diagonal comparison matrix
\begin{equation}
    \boldsymbol{\Theta}_{\caB}
    = \diag \pars[\big]{ (\Theta_i)_{i \in \caB}},
    \qquad
    \Theta_i := {\hbar_i} / {\norm{\bfx_i}},
    \label{eqn:unified-frozen-comparison}
\end{equation}
so that $\mathfrak{Y}_{\caB} := \boldsymbol{\Theta}_{\caB}^{-1} \bfY_{\caB}$ is the stabilized block. We decompose it as
\begin{equation}
    \mathfrak{Y}_{\caB}
    = \boldsymbol{\Theta}_{\caB}^{-1} \bfY_{\caB}
    = \fkS_{\caB} + \fkW_{\caB} + \fkD_{\caB},
    \label{eqn:unified-atypical-decomposition}
\end{equation}
where the components $\fkS_{\caB} = (\fks_{i \mu})$, $\fkW_{\caB} = (\fkw_{i \mu})$ and $\fkD_{\caB} = (\fkd_{i \mu})$ have entries
\begin{equation}
    \fks_{i \mu}
    = \frac{\chi_{i \mu} \psi_{i \mu}}{\hbar_i},
    \qquad
    \mathfrak{w}_{i \mu}
    = \frac{(1 - \chi_{i \mu}) (\zeta_{i \mu} - \bbe \zeta_n)}{\hbar_i},
    \qquad
    \mathfrak{d}_{i \mu}
    = \frac{(1 - \chi_{i \mu}) \bbe \zeta_n}{\hbar_i},
    \label{eqn:unified-S-W-D}
\end{equation}
for $i \in \caB$ and $\mu \in \dbraks{n}$. Here $\fkS_{\caB}$ contains the normalized large entries, $\fkW_{\caB}$ contains the centered truncated entries, and $\fkD_{\caB}$ accounts for the mean of the truncated variables. Conditional on $\caF_n$, both $\fkS_{\caB}$ and $\fkD_{\caB}$ are fixed, while $\fkW_{\caB}$ has independent centered entries. The next lemma gives elementary bounds on these components. Recall $\tilde{\Omega}_n$ from Lemma \ref{lem:unified-graph} \ref{item:unified-graph-components}.

% *******************************************
\begin{lemma}
\label{lem:unified-freezing}
% *******************************************
The following statements hold for all sufficiently large $n$.
\begin{subequations}
\begin{enumerate}[label = (\roman*)]
    \item For every realization of $\bar{\caF}_n$ in $\tilde{\Omega}_n$, the following bounds hold deterministically,
    \begin{equation}
        \norm{\fkS_{\caB}} \leq k,
        \qquad
        \norm{\fkS_{\caB}}_{\mathrm{F}}^2
        \leq n^{1 - \sigma / 2},
        \qquad
        \norm{\fkD_{\caB}} 
        \leq C n^{-1 / k - \sigma / 4 + (\alpha - 1) \tau}
        \leq C n^{-\tau}.
        \label{eqn:unified-freezing-bounds}
    \end{equation}
    \item \label{item:fkW-entries-moments} Conditional on $\caF_n$, the entries of $\fkW_{\caB}$ are independent and centered, with
    \begin{equation}
        \abs{\mathfrak{w}_{i \mu}} \leq C n^{-\tau},
        \qquad
        \bbe^{\caF} \abs{\mathfrak{w}_{i \mu}}^2 \leq C n^{-1},
        \qquad
        \Var^{\caF} \! \pars[\big]{\abs{\mathfrak{w}_{i \mu}}^2}
        \leq C n^{-\alpha / 2-(4 - \alpha) \tau}.
        \label{eqn:unified-diffuse-entries}
    \end{equation}
    \item Write $\fkW_{\caB} = (\bfkw_1, \ldots, \bfkw_n)$ for its columns. Uniformly over all realizations of $\caF_n$ in $\tilde{\Omega}_n$, the following bounds hold simultaneously under $\bbp^{\caF}$,
    \begin{equation}
        \max\nolimits_{\mu \in \dbraks{n}} \norm{\bfkw_\mu} \prec n^{-\tau}, \qquad
        \norm{\fkW_{\caB}} \prec 1, \qquad
        \norm{\fkW_{\caB}}_{\mathrm{F}}^2 \prec n^{1 - \sigma / 2}.
        \label{eqn:unified-column-bound}
    \end{equation}
    \item Uniformly over all realizations of $\caF_n$ in $\tilde{\Omega}_n$, the following bound holds under $\bbp^{\caF}$,
    \begin{equation}
        \norm{\boldsymbol{\Theta}_{\caB} - \bfI}
        \prec n^{-2 \tau}.
        \label{eqn:unified-normalizer-comparison}
    \end{equation}
\end{enumerate}
\end{subequations}
\end{lemma}

% *******************************************
\begin{proof}[Proof of Lemma \ref{lem:unified-freezing}]
% *******************************************
We first prove \eqref{eqn:unified-freezing-bounds}. We introduce $\bar{\fkS}_{\caB} = (\bar{\mathfrak{s}}_{i \nu})$ and $\bar{\fkD}_{\caB} = (\bar{\mathfrak{d}}_{i \nu})$ as the $\bar{\caF}_n$-measurable counterparts of $\fkS_{\caB}$ and $\fkD_{\caB}$, where
\begin{equation*}
    \bar{\mathfrak{s}}_{i \nu} = \bar{\chi}_{i \nu} \bar{\psi}_{i \nu} / {\hbar_i},
    \qquad
    \bar{\mathfrak{d}}_{i \nu} = (1 - \bar{\chi}_{i \nu}) \bbe \zeta_n / {\hbar_i}.
\end{equation*}
Then, the matrices $\fkS_{\caB}$ and $\fkD_{\caB}$ can be obtained from these counterparts by permuting the columns using $\pi_n$. The norms in \eqref{eqn:unified-freezing-bounds} are unchanged by this permutation and are therefore $\bar{\caF}_n$-measurable. Fix a realization of $\bar{\caF}_n$ in $\tilde{\Omega}_n$. For each $i \in \caB$, we have
\begin{equation*}
    \sum\nolimits_{\nu \in \bar{\caA}_i} \abs{\bar{\fks}_{i \nu}}^2
    = {(\hbar_i^2 - n)} / {\hbar_i^2} \leq 1.
\end{equation*}
Thus $\abs{\bar{\fks}_{i \nu}} \leq 1$ and $\norm{\fkS_{\caB}}_{\mathrm{F}}^2 \leq \abs{\caB} \leq \abs{\caE}$. Thus the second estimate in \eqref{eqn:unified-freezing-bounds} follows from Lemma \ref{lem:unified-graph} \ref{item:unified-graph-edge-count}. On the other hand, by Lemma \ref{lem:unified-graph} \ref{item:unified-graph-components}, each row and column of $\bar{\fkS}_{\caB}$ contains at most $k$ nonzero entries, and thus $\norm{\bar{\fkS}_{\caB}}_1 \vee \norm{\bar{\fkS}_{\caB}}_\infty \leq k$. Consequently, we have
\begin{equation*}
    \norm{\fkS_{\caB}} \leq \sqrt{\norm{\bar{\fkS}_{\caB}}_1 \norm{\bar{\fkS}_{\caB}}_\infty} \leq k.
\end{equation*}
For the mean term, \eqref{eqn:unified-low-moments} and $\hbar_i \geq \sqrt{n}$ give
\begin{equation*}
    \norm{\fkD_{\caB}} \leq \norm{\fkD_{\caB}}_{\mathrm{F}}
    \leq C \abs{\caB}^{1 / 2} \abs{\bbe \zeta_n}
    \leq C n^{-1 / k - \sigma / 4 + (\alpha - 1) \tau}
    = C n^{- {5\sigma}/{4} - \tau}
    \leq C n^{-\tau}.
\end{equation*}
Here the equality follows from $\sigma = 1 / k - \alpha \tau$ in \eqref{def:unified-edge-event}. 

Next, the first two bounds in \eqref{eqn:unified-diffuse-entries} follow directly from the definition of $\zeta_{i \mu}$ in \eqref{eqn:alpha3-resampled-zeta} and $\hbar_i^2 \geq n$. The fourth moment estimate in \eqref{eqn:unified-low-moments} gives the last bound,
\begin{equation*}
    \Var^{\caF} \pars[\big]{\abs{\mathfrak{w}_{i \mu}}^2}
    \leq \bbe^{\caF} \abs{\mathfrak{w}_{i \mu}}^4
    \leq {C n^{-2} \bbe \abs{\zeta_n - \bbe \zeta_n}^4}
    \leq C n^{-2} \fku_n^{4 - \alpha}.
\end{equation*}

We now prove \eqref{eqn:unified-column-bound}. Fix a realization of $\caF_n$ in $\tilde{\Omega}_n$. For each $\mu \in \dbraks{n}$, write
\begin{equation}
    \norm{\bfkw_\mu}^2
    = \sum\nolimits_{i \in \caB}
    \pars[\big]{\abs{\mathfrak{w}_{i \mu}}^2
    - \bbe^{\caF} \abs{\mathfrak{w}_{i \mu}}^2}
    + \sum\nolimits_{i \in \caB}
    \bbe^{\caF} \abs{\mathfrak{w}_{i \mu}}^2.
    \label{eqn:column-norm-decomposition}
\end{equation}
The summands in the first sum are independent and centered conditional on $\caF_n$. By \eqref{eqn:unified-diffuse-entries},
\begin{equation*}
    \abs{\mathfrak{w}_{i \mu}}^2 \leq C n^{-2 \tau},
    \qquad
    \sum\nolimits_{i \in \caB}
    \Var^{\caF} \pars{\abs{\mathfrak{w}_{i \mu}}^2}
    \leq C n^{-\alpha / 2-(4 - \alpha) \tau} \abs{\caB}
    = o (n^{-4 \tau}).
\end{equation*}
The last step uses $n^{-\alpha / 2} \abs{\caB} \leq \abs{\caB} / n \leq n^{-\sigma / 2}$ and $\sigma / 2 > \alpha \tau$, which follows from \eqref{eqn:unified-truncation}. Applying the Bernstein inequality \eqref{eqn:scalar-Bernstein-domination} to \eqref{eqn:column-norm-decomposition} with $M_n = C n^{-2 \tau}$ and $V_n = C n^{-4 \tau}$ gives
\begin{equation*}
    \norm{\bfkw_\mu}^2
    \prec n^{-2 \tau} + \sum\nolimits_{i \in \caB}
    \bbe^{\caF} \abs{\mathfrak{w}_{i \mu}}^2
    \leq n^{-2 \tau} + C \abs{\caB} / n
    \leq C n^{-2 \tau},
\end{equation*}
where the second moments are controlled by \eqref{eqn:unified-diffuse-entries}. This proves the first estimate in \eqref{eqn:unified-column-bound}. The same argument controls the Frobenius norm of $\fkW_{\caB}$,
\begin{align*}
    \norm{\fkW_{\caB}}_{\mathrm{F}}^2
    & = \sum\nolimits_{i \in \caB}
    \sum\nolimits_{\mu = 1}^n
    \pars[\big]{\abs{\mathfrak{w}_{i \mu}}^2
    - \bbe^{\caF} \abs{\mathfrak{w}_{i \mu}}^2}
    + \sum\nolimits_{i \in \caB}
    \sum\nolimits_{\mu = 1}^n
    \bbe^{\caF} \abs{\mathfrak{w}_{i \mu}}^2 \\
    & \prec \pars[\big]{ n^{-2 \tau}
    + n^{1 / 2 - \alpha / 4-(4 - \alpha) \tau / 2} \abs{\caB} }
    + \abs{\caB}
    \leq C \abs{\caB}
    \leq C n^{1 - \sigma / 2},
\end{align*}
which proves the last estimate in \eqref{eqn:unified-column-bound}. To control $\norm{\fkW_{\caB}}$, write
\begin{equation*}
    \fkW_{\caB}
    = \sum\nolimits_{i \in \caB, \mu \in \dbraks{n}}
    \fkX_{i \mu},
    \qquad
    \fkX_{i \mu}
    := \mathfrak{w}_{i \mu} \bfe_{i} \bfe_{\mu}^\top.
\end{equation*}
Conditional on $\caF_n$, the summands are independent and centered. Moreover, by \eqref{eqn:unified-diffuse-entries}, the following holds uniformly over all realizations of $\caF_n$ in $\tilde{\Omega}_n$,
\begin{align*}
    \norm{\fkX_{i \mu}}
    & = \abs{\mathfrak{w}_{i \mu}}
    \leq C n^{-\tau}, \\
    \norm*{\sum\nolimits_{i \in \caB}
    \sum\nolimits_{\mu = 1}^n
    \bbe^{\caF} [\fkX_{i \mu} \fkX_{i \mu}^{\top}] }
    & = \max_{i \in \caB}
    \sum\nolimits_{\mu = 1}^n
    \bbe^{\caF} \abs{\mathfrak{w}_{i \mu}}^2
    \leq C, \\
    \norm*{\sum\nolimits_{i \in \caB}
    \sum\nolimits_{\mu = 1}^n
    \bbe^{\caF} [\fkX_{i \mu}^{\top} \fkX_{i \mu}] }
    & = \max_{\mu \in \dbraks{n}}
    \sum\nolimits_{i \in \caB}
    \bbe^{\caF} \abs{\mathfrak{w}_{i \mu}}^2
    \leq C \abs{\caB} / n
    \leq C n^{-\sigma / 2}.
\end{align*}
Therefore, we can apply the matrix Bernstein inequality \eqref{eqn:matrix-Bernstein-domination} with $M_n = C n^{-\tau}$ and $V_n = C$ to obtain the second estimate in \eqref{eqn:unified-column-bound}.

Finally, we prove \eqref{eqn:unified-normalizer-comparison}. Conditional on $\caF_n$, the truncated variables $\{ \zeta_{i \mu} \}_{i \in \caB, \mu \notin \caA_i}$ in atypical rows are independent with common law $\caL (\zeta_n)$. Thus, the conclusion of Lemma \ref{lem:unified-square-sums} remains valid conditional on $\caF_n$ when the maximum is restricted to $i \in \caB$,
\begin{equation*}
    \max_{i \in \caB} \,
    \abs*{\sum\nolimits_{\mu \notin \caA_i}
    \abs{\zeta_{i \mu}}^2
    - (n - \abs{\caA_i}) \bbe \abs{\zeta_n}^2}
    \prec n^{1-2 \tau}.
\end{equation*}
On the event $\tilde{\Omega}_n$, we have $\abs{\caA_i} \leq k$. Together with \eqref{eqn:unified-low-moments}, this gives
\begin{equation*}
    \abs[\big]{(n - \abs{\caA_i}) \bbe \abs{\zeta_n}^2 - n}
    \leq C (k + n \fku_n^{2 - \alpha}) = o (n^{1-2 \tau}).
\end{equation*}
Since $\hbar_i^2 \geq n$, these estimates yield3 $\norm{\bfx_i}^2 / \hbar_i^2 = 1 + O_{\prec} (n^{-2 \tau})$ uniformly over $i \in \caB$, which proves \eqref{eqn:unified-normalizer-comparison}.
\end{proof}

% *******************************************
\subsection{Centered resolvent estimates}
\label{subsec:atypical-centered-resolvent}
% *******************************************

We proceed to control the interaction of the sparse and centered truncated components in \eqref{eqn:unified-atypical-decomposition} with the centered resolvent of the typical block,
\begin{equation*}
    \bfUps_{\caT} (z)
    := \bfG_{\caT} (z) - \fkm_{\varphi_n}\!(z) \bfI
    = (\Upsilon_{\mu \nu} (z))_{\mu, \nu = 1}^n.
\end{equation*}
Recall that $\mathcal{D}_\perp [\bfA]$ denotes the off-diagonal part of a square matrix $\bfA$.

% *******************************************
\begin{proposition}
\label{prop:unified-centered-compressions}
% *******************************************
Fix a compact interval $\bbk_+ \subset (\lambda_+, \infty)$. For the lower edge estimates, assume $\phi \in (0, 1)$ and also fix a compact interval $\bbk_- \subset (0, \lambda_-)$. Uniformly over all realizations of $\bar{\caH}_n$ in $\tilde{\Omega}_n$, the following estimates hold under $\bbp^{\bar{\caH}}$,
\begin{subequations}
\begin{align}
    \sup\nolimits_{z \in \bbk_\pm} \!
    \curls*{\norm[\big]{\fkS_{\caB} \bfUps_{\caT} (z) \fkS_{\caB}^{\top}}
    + \norm[\big]{\fkS_{\caB} \bfUps_{\caT} (z) \fkW_{\caB}^{\top}}
    + \norm[\big]{\fkW_{\caB} \bfUps_{\caT} (z) \fkW_{\caB}^{\top}}}
    & \prec n^{-\tau},
    \label{eqn:unified-centered-compressions} \\
    \norm[\big]{\fkS_{\caB} \fkW_{\caB}^{\top}}
    + \norm[\big]{\mathcal{D}_\perp[\fkW_{\caB} \fkW_{\caB}^{\top}]}
    & \prec n^{-\tau}.
    \label{eqn:unified-identity-kernel}
\end{align}
\end{subequations}
\end{proposition}

Our proof of Proposition \ref{prop:unified-centered-compressions} relies on Tropp's random principal submatrix estimate \cite[Theorem 1.1]{troppNormsRandomSubmatrices2008}. Tropp's theorem concerns random subsets obtained by including each index independently with a common probability, whereas we require a version for uniformly chosen subsets of fixed cardinality. The following lemma gives the needed variant, which follows from Tropp's result by a simple coupling argument. We defer its proof to Appendix \ref{sec:tech-lemmas}.

% *******************************************
\begin{lemma}
\label{lemma:fixed-cardinality-random-compression}
% *******************************************
Let $\bfH \in \bbr^{n \times n}$ be deterministic and symmetric. Let $d \in \dbraks{n}$, and let $\caJ \subset \dbraks{n}$ be uniformly distributed over all subsets of cardinality $d$. Denote by $\bfP_\caJ : \bbr^n \to \bbr^d$ the corresponding coordinate projection. Then,
\begin{equation}
    \norm{\bfP_\caJ \bfH \bfP_\caJ^{\top}}
    \prec \norm{\bfH}_{\max}
    + \sqrt{d / n} \norm{\bfH}.
    \label{eq:fixed-cardinality-compression-moment}
\end{equation}
where $\norm{\bfH}_{\max}$ denotes the largest absolute value among the entries of $\bfH$.
\end{lemma}

% *******************************************
\begin{proof}[Proof of Proposition \ref{prop:unified-centered-compressions}]
% *******************************************
Fix a realization of $\bar{\caH}_n$ in $\tilde{\Omega}_n$. All stochastic estimates below are under $\bbp^{\bar{\caH}}$ and are uniform over such realizations. Proposition \ref{prop:unified-core} and $\partial_z \bfUps_{\caT} = \bfG_{\caT}^2 - (\partial_z \fkm_{\varphi_n}) \bfI$ gives
\begin{equation}
    \sup\nolimits_{z \in \bbk_\pm}
    \curls[\big]{\norm{\bfUps_{\caT} (z)}
    + \norm{\partial_z \bfUps_{\caT} (z)}} \prec 1,
    \qquad
    \sup\nolimits_{z \in \bbk_\pm}
    \norm{\bfUps_{\caT} (z)}_{\max} \prec n^{-\tau},
    \label{eqn:unified-centered-lipschitz}
\end{equation} 
It suffices to prove \eqref{eqn:unified-centered-compressions} for each fixed $z \in \bbk_\pm$, with the supremum omitted and stochastic domination uniform in $z$. Indeed, by \eqref{eqn:unified-freezing-bounds}, \eqref{eqn:unified-column-bound}, and \eqref{eqn:unified-centered-lipschitz}, each of the three norms is Lipschitz continuous in $z$ with Lipschitz constant $O (n^{\varepsilon})$ for some constant $\varepsilon > 0$ with overwhelming probability. A standard $\varepsilon$-net argument then yields the desired supremum bound. We henceforth suppress the spectral parameter $z$.

\shortpara{The $\fkS$-$\fkS$ estimate}. Since the columns of $\fkS_{\caB}$ outside $\caJ$ vanish and $\norm{\fkS_{\caB}} \leq k$ by \eqref{eqn:unified-freezing-bounds},
\begin{equation*}
    \norm[\big]{\fkS_{\caB} \bfUps_{\caT} \fkS_{\caB}^\top}
    = \norm[\big]{\fkS_{\caB} \bfP_{\caJ}^\top
    \bfP_{\caJ} \bfUps_{\caT} \bfP_{\caJ}^\top
    \bfP_{\caJ} \fkS_{\caB}^\top}
    \leq k^2 \norm[\big]{\bfP_{\caJ} \bfUps_{\caT} \bfP_{\caJ}^\top}.
\end{equation*}
Under $\bbp^{\bar{\caF}}$, the index set $\caJ$ is uniformly distributed among the subsets of $\dbraks{n}$ with cardinality $\abs{\caJ} = \abs{\bar{\caJ}}$\footnote{This is the main reason for introducing the permutation $\pi_n$ in the resampling representation \eqref{eqn:unified-resampling}.}. Applying Lemma \ref{lemma:fixed-cardinality-random-compression}, we obtain
\begin{equation*}
    \norm[\big]{\bfP_{\caJ} \bfUps_{\caT} \bfP_{\caJ}^\top}
    \prec \norm{\bfUps_{\caT}}_{\max}
    + ({\abs{\bar{\caJ}}} / {n})^{1/2}
    \norm{\bfUps_{\caT}}
    \prec n^{-\tau},
\end{equation*}
where we also used \eqref{eqn:unified-centered-lipschitz}, $\abs{\bar{\caJ}} \leq \abs{\bar{\caE}} \leq n^{1 - \sigma / 2}$ from Lemma \ref{lem:unified-graph} \ref{item:unified-graph-edge-count} and $\tau < \sigma / 4$.

\shortpara{The $\fkS$-$\fkW$ estimates}. We condition further on $\caF_n$. Let $\bfA \in \bbr^{n \times n}$ be $\caF_n$-measurable with $\norm{\bfA} = \bigO{1}$. Consider the rank-one decomposition
\begin{equation*}
    \fkS_{\caB} \bfA \fkW_{\caB}^{\top}
    = \sum\nolimits_{i \in \caB} \sum\nolimits_{\mu = 1}^n
    \fkX_{i \mu},
    \qquad
    \fkX_{i \mu} = \fkw_{i \mu}
    \fkS_{\caB} \bfA \bfe_{\mu} \bfe_{i}^\top.
\end{equation*}
Thus the summands $\fkX_{i \mu}$ are independent centered matrices conditional on $\caF_n$. On $\tilde{\Omega}_n$, the estimates \eqref{eqn:unified-freezing-bounds} and \eqref{eqn:unified-diffuse-entries} give
\begin{align*}
    \norm{\fkX_{i \mu}}
    & \leq C \abs{\fkw_{i \mu}}
    \leq C n^{-\tau}, \\
    \norm*{\sum\nolimits_{i \in \caB}
    \sum\nolimits_{\mu = 1}^n
    \bbe^{\caF} [\fkX_{i \mu} \fkX_{i \mu}^{\top}] }
    & \leq \frac{C \abs{\caB}}{n} \norm{\fkS_{\caB} \bfA}^2
    \leq C n^{-\sigma / 2}, \\
    \norm*{\sum\nolimits_{i \in \caB}
    \sum\nolimits_{\mu = 1}^n
    \bbe^{\caF} [\fkX_{i \mu}^{\top} \fkX_{i \mu}] }
    & \leq \frac{C}{n} \norm{\fkS_{\caB} \bfA}_{\mathrm{F}}^2
    \leq \frac{C}{n} \norm{\fkS_{\caB}}_{\mathrm{F}}^2
    \leq C n^{-\sigma / 2}.
\end{align*}
Hence, the matrix Bernstein inequality \eqref{eqn:matrix-Bernstein-domination} with $M_n = C n^{-\tau}$ and $V_n = C n^{-\sigma / 2}$ yields
\begin{equation*}
    \norm{\fkS_{\caB} \bfA \fkW_{\caB}^\top}
    \prec n^{-\tau} + n^{-\sigma / 4}
    \prec n^{-\tau}.
\end{equation*}
Taking $\bfA = \bfUps_{\caT}$ or $\bfA = \bfI$ gives the two $\fkS$-$\fkW$ estimates.

\shortpara{The $\fkW$-$\fkW$ estimates}. Again condition further on $\caF_n$ and let $\bfA \in \bbr^{n \times n}$ be $\caF_n$-measurable with $\norm{\bfA} = \bigO{1}$. We introduce the conditional covariance $\bfSigma_\mu := \bbe^{\caF} [\bfkw_\mu \bfkw_\mu^\top]$ and write
\begin{equation}
    \fkW_{\caB} \bfA \fkW_{\caB}^\top
    - \bbe^{\caF} [\fkW_{\caB} \bfA \fkW_{\caB}^\top]
    = \sum\nolimits_{\mu = 1}^n A_{\mu \mu}
    (\bfkw_\mu \bfkw_\mu^\top - \bfSigma_\mu)
    + \sum\nolimits_{\mu \neq \nu}
    A_{\mu \nu} \bfkw_\mu \bfkw_\nu^\top.
    \label{eqn:fkW-fkW-diag-offdiag}
\end{equation}
By Lemma \ref{lem:unified-freezing} \ref{item:fkW-entries-moments}, each covariance $\bfSigma_\mu$ is diagonal and satisfies $\norm{\bfSigma_\mu} \leq C / n$.

For the diagonal sum in \eqref{eqn:fkW-fkW-diag-offdiag}, fix an arbitrarily small $\delta > 0$ and introduce
\begin{equation*}
    \Xi_\mu := \curls{\norm{\bfkw_\mu} \leq n^{-\tau + \delta}},
    \qquad
    \Xi := \bigcap\nolimits_{\mu = 1}^n \Xi_\mu,
    \qquad
    \hat{\bfSigma}_\mu
    := \bbe^{\caF} [\bfkw_\mu \bfkw_\mu^\top \mid \Xi_\mu].
\end{equation*}
By \eqref{eqn:unified-column-bound}, we have $\bbp^{\caF} (\Xi^c) \leq n^{-D}$ for every fixed $D > 0$. As each event $\Xi_\mu$ depends only on the corresponding column, the columns $\bfkw_\mu$ remain independent under $\bbp^{\caF}\{\, \cdot \mid \Xi\}$. Therefore, by the definition of $\hat{\bfSigma}_\mu$, 
\begin{equation*}
    \norm{\hat{\bfSigma}_\mu - \bfSigma_\mu}
    = \norm[\bigg]{\frac{\bbp^{\caF} (\Xi_\mu^c)}{\bbp^{\caF} (\Xi_\mu)}
    \bfSigma_\mu
    - \frac{1}{\bbp^{\caF} (\Xi_\mu)}
    \bbe^{\caF}[\bfkw_\mu \bfkw_\mu^{\top}
    \bbone(\Xi_\mu^c)]}
    = O (n^{-D + 1}),
\end{equation*}
where we used the first bound in \eqref{eqn:unified-diffuse-entries} to control $\norm{\bfkw_\mu}$ on $\Xi_\mu^c$. Consequently,
\begin{equation}
    \max_{\mu \in \dbraks{n}} \norm{\hat{\bfSigma}_\mu} \leq C / n ,
    \qquad
    \norm*{\sum\nolimits_{\mu = 1}^n \hat{\bfSigma}_\mu} \leq C,
    \qquad
    \sum\nolimits_{\mu = 1}^n
    \norm{\hat{\bfSigma}_\mu - \bfSigma_\mu} \leq C n^{-D + 2}.
    \label{eqn:hat-Sigma-properties}
\end{equation}
Under $\bbp^{\caF} \{\, \cdot \mid \Xi\}$, consider the following sum of independent centered self-adjoint matrices,
\begin{equation*}
    \sum\nolimits_{\mu = 1}^n \fkX_\mu,
    \qquad
    \fkX_\mu := A_{\mu \mu}
    (\bfkw_\mu \bfkw_\mu^{\top} - \hat{\bfSigma}_\mu).
\end{equation*}
Using \eqref{eqn:hat-Sigma-properties} and $\norm{\bfkw_\mu} \leq n^{-\tau + \delta}$ on $\Xi$, we obtain $\norm{\fkX_\mu} \leq C n^{-2 \tau + 2 \delta}$ and
\begin{align*}
    \norm*{\sum\nolimits_{\mu = 1}^n
    \bbe^{\caF} [\fkX_\mu^2 \mid \Xi]}
    & \leq C \norm*{\sum\nolimits_{\mu = 1}^n
    \bbe^{\caF} \braks{(\bfkw_\mu \bfkw_\mu^{\top})^2 \mid \Xi} } \\
    & \leq C n^{-2 \tau + 2 \delta}
    \norm*{\sum\nolimits_{\mu = 1}^n \hat{\bfSigma}_\mu}
    \leq C n^{-2 \tau + 2 \delta}.
\end{align*}
Applying the matrix Bernstein inequality \eqref{eqn:matrix-Bernstein-domination} with $M_n = C n^{-2 \tau + 2 \delta}$ and $V_n = C n^{-2 \tau + 2 \delta}$ yields $\norm{\sum_{\mu = 1}^n \fkX_\mu} \prec n^{-\tau + \delta}$ under $\bbp^{\caF} \{\, \cdot \mid \Xi\}$. Since $\delta > 0$ is arbitrarily small and $\Xi$ holds with overwhelming probability, removing the conditioning on $\Xi$ preserves the bound. Together with the last estimate in \eqref{eqn:hat-Sigma-properties}, this gives, under $\bbp^{\caF}$,
\begin{equation}
    \norm*{\sum\nolimits_{\mu = 1}^n A_{\mu \mu}
    (\bfkw_\mu \bfkw_\mu^\top - \bfSigma_\mu)}
    \leq \norm*{\sum\nolimits_{\mu = 1}^n \fkX_\mu}
    + \norm*{\sum\nolimits_{\mu = 1}^n
    A_{\mu \mu} ({\hat{\bfSigma}_\mu - \bfSigma_\mu})} 
    \prec n^{-\tau}.
    \label{eqn:diagonal-fkW-fkW-bound}    
\end{equation}

Let $\tilde{\fkW}_{\caB} = (\tilde{\bfkw}_1,\ldots,\tilde{\bfkw}_n)$ be an independent copy of $\fkW_{\caB}$ conditional on $\caF_n$. By the standard decoupling inequality for the tails of $U$-statistics (see, e.g., \cite[Theorem 3.4.1]{penaDecouplingDependenceIndependence2012}), we have
\begin{equation}
    \bbp^{\caF}\curls*{
    \norm*{\sum\nolimits_{\mu \neq \nu}
    A_{\mu \nu}\bfkw_\mu\bfkw_\nu^{\top}} > Ct}
    \leq C\bbp^{\caF}\curls*{
    \norm*{\sum\nolimits_{\mu \neq \nu}
    A_{\mu \nu}\bfkw_\mu \tilde{\bfkw}_\nu^{\top}} > t},
    \label{eqn:tail-decoupling}
\end{equation}
for every $t > 0$. It therefore suffices to control the decoupled sum
\begin{equation*}
    \sum\nolimits_{\mu \neq \nu}
    A_{\mu \nu} \bfkw_\mu \tilde{\bfkw}_\nu^{\top}
    = \sum\nolimits_{i \in \caB}
    \sum\nolimits_{\nu = 1}^n \tilde{\fkX}_{i \nu},
    \qquad
    \tilde{\fkX}_{i \nu}
    := \tilde{\fkw}_{i \nu} \bfku_\nu \bfe_i^{\top},
    \qquad
    \bfku_\nu := \fkW_{\caB} \mathcal{D}_\perp[\bfA] \bfe_\nu.
\end{equation*}
After conditioning also on $\fkW_{\caB}$, the matrices $\tilde{\fkX}_{i \nu}$ are independent and centered. By \eqref{eqn:unified-column-bound} and the definition of $\bfku_\nu$, the following estimates hold with overwhelming probability,
\begin{equation*}
    \norm{\bfku_\nu}
    \leq C \norm{\fkW_{\caB}}
    \leq C n^{\delta},
    \qquad
    \norm*{\sum\nolimits_{\nu = 1}^n \bfku_\nu \bfku_\nu^\top}
    \leq C \norm{\fkW_{\caB}}^2
    \leq C n^{2 \delta},
    \qquad
    \sum\nolimits_{\nu = 1}^n \norm{\bfku_\nu}^2
    \leq C \norm{\fkW_{\caB}}_{\mathrm{F}}^2
    \leq C n^{1 - \sigma / 2 + \delta}.    
\end{equation*}
Therefore, it suffices to consider realizations of $\fkW_{\caB}$ for which these estimates hold. For such realizations,
\begin{align*}
    \norm{\tilde{\fkX}_{i \nu}}
    & \leq C n^{\delta} \abs{\tilde{\fkw}_{i \nu}}
    \leq C n^{-\tau + \delta}, \\
    \norm*{\sum\nolimits_{i \in \caB}
    \sum\nolimits_{\nu = 1}^n
    \bbe^{\caF} [\tilde{\fkX}_{i \nu}
    \tilde{\fkX}_{i \nu}^\top \mid \fkW_{\caB}]}
    & \leq \frac{C \abs{\caB}}{n}
    \norm*{\sum\nolimits_{\nu = 1}^n \bfku_\nu \bfku_\nu^\top}
    \leq C n^{-\sigma/2 + 2 \delta}, \\
    \norm*{\sum\nolimits_{i \in \caB}
    \sum\nolimits_{\nu = 1}^n
    \bbe^{\caF} [\tilde{\fkX}_{i \nu}^\top
    \tilde{\fkX}_{i \nu} \mid \fkW_{\caB}]}
    & \leq \frac{C}{n} \sum\nolimits_{\nu = 1}^n \norm{\bfku_\nu}^2
    \leq C n^{- \sigma / 2 + \delta}.
\end{align*}
Now, the matrix Bernstein inequality \eqref{eqn:matrix-Bernstein-domination} with $M_n = C n^{-\tau + \delta}$ and $V_n = C n^{-\sigma / 2 + 2 \delta}$ controls the norm of the decoupled sum by $O_{\prec} (n^{-\tau + \delta})$. Here $\delta > 0$ is arbitrarily small. The tail decoupling inequality \eqref{eqn:tail-decoupling} then transfers this bound to the original off-diagonal sum,
\begin{equation}
    \norm*{\sum\nolimits_{\mu \neq \nu}
    A_{\mu \nu} \bfkw_\mu {\bfkw}_\nu^{\top}}
    \prec n^{-\tau}.
    \label{eqn:offdiagonal-fkW-fkW-bound}
\end{equation}

Combining the two estimates \eqref{eqn:diagonal-fkW-fkW-bound} and \eqref{eqn:offdiagonal-fkW-fkW-bound} with \eqref{eqn:fkW-fkW-diag-offdiag}, we conclude that
\begin{equation}
    \norm[\big]{\fkW_{\caB} \bfA \fkW_{\caB}^\top
    - \bbe^{\caF} [\fkW_{\caB} \bfA \fkW_{\caB}^\top]}
    \prec n^{-\tau}.
    \label{eqn:fkW-fkW-centered}
\end{equation}
For $\bfA = \bfI$, the conditional mean $\bbe^{\caF} [\fkW_{\caB} \fkW_{\caB}^\top]$ is diagonal. Hence, using $\norm{\mathcal{D}_\perp[\bfH]} \leq 2 \norm{\bfH}$ and \eqref{eqn:fkW-fkW-centered} proves the second estimate in \eqref{eqn:unified-identity-kernel}. For $\bfA = \bfUps_{\caT}$, the conditional mean satisfies
\begin{equation*}
    \norm[\big]{\bbe^{\caF} [\fkW_{\caB} \bfUps_{\caT} \fkW_{\caB}^\top]}
    = \norm*{\sum\nolimits_{\mu = 1}^n \Upsilon_{\mu \mu} \bfSigma_\mu}
    \leq C \norm{\bfUps_{\caT}}_{\max} \prec n^{-\tau},
\end{equation*}
where we used $\norm{\bfSigma_\mu} \leq C/n$ and \eqref{eqn:unified-centered-lipschitz}. This proves the last estimate in \eqref{eqn:unified-centered-compressions}.
\end{proof}

% *******************************************
\subsection{The effective Gram matrix}
\label{subsec:effective-gram}
% *******************************************

The preceding estimates allow us to isolate the contribution of large entries in the atypical block $\bfY_{\caB}$. Define the \emph{effective Gram matrix}
\begin{equation}
    \bfL_{\caB}
    = \bfI + \mathcal{D}_\perp[\fkS_{\caB} \fkS_{\caB}^{\top}]
    = \fkS_{\caB} \fkS_{\caB}^{\top}
    + \diag \pars[\big]{(n / \hbar_i^2)_{i \in \caB}}.
    \label{eqn:unified-component-matrix}
\end{equation}
By definition, $\bfL_{\caB}$ is positive definite with unit diagonal.

% *******************************************
\begin{proposition}
\label{prop:unified-resolvent}
% *******************************************
Fix an integer $k \geq 2$ and assume \eqref{eqn:basic-assumption} and \eqref{eqn:analytic-tail-envelope}. Let $\bbk_+ \subset (\lambda_+, \infty)$ be a compact interval. When $\phi \in (0, 1)$, also let $\bbk_- \subset (0, \lambda_-)$ be a compact interval. Then,
\begin{equation}
    \sup\nolimits_{z \in \bbk_\pm}
    \norm*{\bfY_{\caB} (\bfQ_{\caT} - z \bfI)^{-1} \bfY_{\caB}^{\top}
    - \fkm_\phi (z) \bfL_{\caB}} \pconv 0.
    \label{eqn:unified-resolvent-approximation}
\end{equation}
\end{proposition}

% *******************************************
\begin{proof}[Proof of Proposition \ref{prop:unified-resolvent}]
% *******************************************
It suffices to prove \eqref{eqn:unified-resolvent-approximation} under $\bbp^{\bar{\caH}}$, uniformly over all realizations of $\bar{\caH}_n$ in $\tilde{\Omega}_n$. The unconditional convergence then follows by averaging the conditional probabilities and using $\bbp (\tilde{\Omega}_n) \to 1$, as established in Lemma \ref{lem:unified-graph} \ref{item:unified-graph-components}. We begin by decomposing the matrix in \eqref{eqn:unified-resolvent-approximation} as
\begin{equation}
    \bfY_{\caB} (\bfG_{\caT} - \fkm_{\varphi_n} \bfI) \bfY_{\caB}^\top
    + \fkm_{\varphi_n} \!
    (\bfY_{\caB} \bfY_{\caB}^\top - \bfL_{\caB})
    + (\fkm_{\varphi_n} - \fkm_\phi) \bfL_{\caB}.
    \label{eqn:resolvent-estimate-components}
\end{equation}

\begin{subequations}
For the first term in \eqref{eqn:resolvent-estimate-components}, recall that $\bfUps_{\caT} = \bfG_{\caT} - \fkm_{\varphi_n} \bfI$ and that $\norm{\bfUps_{\caT}} \prec 1$ by \eqref{eqn:unified-centered-lipschitz}. Since $\bfUps_{\caT} (z)$ is symmetric for real $z$, expanding $\mathfrak{Y}_{\caB} = \fkS_{\caB} + \fkW_{\caB} + \fkD_{\caB}$ gives
\begin{equation*}
    \begin{aligned}
        \norm[\big]{\mathfrak{Y}_{\caB} \bfUps_{\caT} \mathfrak{Y}_{\caB}^{\top}} 
        \leq & ~ 2 \pars[\big]{
        \norm{\fkS_{\caB}} + \norm{\fkW_{\caB}}
        + \norm{\fkD_{\caB}} }
        \norm{\fkD_{\caB}} \norm{\bfUps_{\caT}} \\
        & + \norm[\big]{\fkS_{\caB} \bfUps_{\caT} \fkS_{\caB}^{\top}}
        + 2 \norm[\big]{\fkS_{\caB} \bfUps_{\caT} \fkW_{\caB}^{\top}} 
        + \norm[\big]{\fkW_{\caB} \bfUps_{\caT} \fkW_{\caB}^{\top}} 
        \prec n^{-\tau}.
    \end{aligned}
\end{equation*}
The three terms in the second line are controlled by \eqref{eqn:unified-centered-compressions}, and the first line is bounded using \eqref{eqn:unified-freezing-bounds} and \eqref{eqn:unified-column-bound}. Moreover, $\bfY_{\caB} = \boldsymbol{\Theta}_{\caB} \mathfrak{Y}_{\caB}$, where $\boldsymbol{\Theta}_{\caB}$ is the diagonal comparison matrix defined in \eqref{eqn:unified-frozen-comparison}. Hence, combining \eqref{eqn:unified-normalizer-comparison} with the preceding estimate yields
\begin{equation}
    \sup\nolimits_{z \in \bbk_\pm}
    \norm[\big]{\bfY_{\caB}
    [\bfG_{\caT} (z) - \fkm_{\varphi_n}\!(z) \bfI]
    \bfY_{\caB}^{\top}}
    \prec n^{-\tau}.
    \label{eqn:unified-resolvent-detailed-centered}
\end{equation}

We next estimate the second term in \eqref{eqn:resolvent-estimate-components}. Note that both $\bfY_{\caB} \bfY_{\caB}^{\top}$ and $\bfL_{\caB}$ have unit diagonal. Consequently, \eqref{eqn:unified-component-matrix} gives
\begin{equation*}
    \bfY_{\caB} \bfY_{\caB}^{\top} - \bfL_{\caB}
    = \mathcal{D}_\perp
    [\fkY_{\caB} \fkY_{\caB}^{\top} - \fkS_{\caB} \fkS_{\caB}^{\top}]
    + \mathcal{D}_\perp
    [\bfY_{\caB} \bfY_{\caB}^{\top} - \fkY_{\caB} \fkY_{\caB}^{\top}].
\end{equation*}
As before, the second term can be controlled by $O_{\prec} (n^{-2 \tau})$ using \eqref{eqn:unified-normalizer-comparison}. To estimate the first term, we expand $\mathfrak{Y}_{\caB} = \fkS_{\caB} + \fkW_{\caB} + \fkD_{\caB}$ to obtain
\begin{align*}
    \norm[\big]{\mathcal{D}_\perp
    [\mathfrak{Y}_{\caB} \mathfrak{Y}_{\caB}^{\top}
    - \fkS_{\caB} \fkS_{\caB}^{\top}]} 
    \leq & ~ 4
    (\norm{\fkS_{\caB}} + \norm{\fkW_{\caB}} + \norm{\fkD_{\caB}})
    \norm{\fkD_{\caB}} \\
    & + 4 \norm[\big]{\fkS_{\caB} \fkW_{\caB}^{\top}}
    + \norm[\big]{\mathcal{D}_\perp [\fkW_{\caB} \fkW_{\caB}^{\top}]} 
    \prec n^{-\tau},
\end{align*}
where we used $\norm{\mathcal{D}_\perp [\bfA]} \leq 2 \norm{\bfA}$ and \eqref{eqn:unified-identity-kernel}. Therefore,
\begin{equation}
    \norm[\big]{\bfY_{\caB} \bfY_{\caB}^{\top} - \bfL_{\caB}}
    \prec n^{-\tau}.
    \label{eqn:unified-resolvent-detailed-Gram}
\end{equation}

For the last term of \eqref{eqn:resolvent-estimate-components}, recall that $\varphi_n = (p_n - \abs{\caB}) / n$. Lemma \ref{lem:unified-graph} \ref{item:unified-graph-edge-count} ensures that $\abs{\varphi_n - \phi} \to 0$ uniformly over the conditioning realizations under consideration. By continuity of the companion MP Stieltjes transform in the aspect ratio,
\begin{equation}
    \sup\nolimits_{z \in \bbk_\pm}
    \abs{\fkm_{\varphi_n}\!(z) - \fkm_\phi (z)} \to 0,
    \qquad
    \sup\nolimits_{z \in \bbk_\pm} 
    \abs{\fkm_{\varphi_n}\!(z)} = O (1).
    \label{eqn:unified-resolvent-detailed-MP}
\end{equation}
Finally, the definition of $\bfL_{\caB}$ in \eqref{eqn:unified-component-matrix}, together with $\hbar_i^2 \geq n$ and \eqref{eqn:unified-freezing-bounds}, gives
\begin{equation}
    \norm{\bfL_{\caB}}
    \leq \norm{\fkS_{\caB}}^2
    + \max\nolimits_{i \in \caB} (n / \hbar_i^2)
    \leq k^2 + 1.
    \label{eqn:Lb-trivial-bound}
\end{equation}
Combining \eqref{eqn:unified-resolvent-detailed-centered}--\eqref{eqn:Lb-trivial-bound} with the decomposition \eqref{eqn:resolvent-estimate-components} proves the required conditional convergence and completes the proof.
\end{subequations}
\end{proof}

% *******************************************
\subsection{Spectral reduction}
\label{subsec:spectral-transfer}
% *******************************************

We use Proposition \ref{prop:unified-resolvent} to relate the extreme eigenvalues of $\bfR_n$ to those of the effective Gram matrix $\bfL_{\caB}$ introduced in \eqref{eqn:unified-component-matrix}. Recall the Stieltjes transform $\fkm_\phi$ from \eqref{eqn:companion-MP-transform}. It satisfies the self-consistent equation
\begin{equation}
    z = -\frac{1}{\fkm_\phi (z)}
    + \frac{\phi}{1 + \fkm_\phi (z)}.
 \label{eqn:companion-MP-equation}
\end{equation}

The Stieltjes transform $\fkm_\phi$ is negative and strictly increasing on $(\lambda_+, \infty)$. It has the same property on $(0, \lambda_-)$ when $\phi \in (0, 1)$. Recall the outlier map $\theta_\phi$ from \eqref{def:upper-outlier-map}. Its inverse is well defined on $(\lambda_+, \infty)$ for every $\phi > 0$ and on $(0, \lambda_-)$ when $0 < \phi < 1$. We denote
\begin{equation}
    \fkx_\phi (z) := \theta_\phi^{-1} (z) = -1 / \fkm_\phi (z).
    \label{def:inverse-map}
\end{equation}
Indeed, substituting $\fkm_\phi = -1 / \fkx_\phi$ into \eqref{eqn:companion-MP-equation} gives $z = \theta_\phi (\fkx_\phi (z))$. Moreover, differentiating $\fkx_\phi = -1 / \fkm_\phi$ shows that $\fkx_\phi$ is strictly increasing on each gap,
\begin{equation}
    \partial_z \fkx_\phi (z)
    = \frac{\partial_z \fkm_\phi (z)}{\fkm_\phi (z)^2} > 0.
    \label{eqn:fkx-increasing-homeomorphism}
\end{equation}
The endpoint limits are given below, with the second line valid when $0 < \phi < 1$:
\begin{alignat*}{2}
    \lim\nolimits_{z \downarrow \lambda_+}
    \fkx_\phi (z) & = 1 + \sqrt{\phi},
    & \qquad
    \lim\nolimits_{z \uparrow \infty}
    \fkx_\phi (z) & = \infty, \\
    \lim\nolimits_{z \uparrow \lambda_-}
    \fkx_\phi (z) & = 1 - \sqrt{\phi},
    & \qquad
    \lim\nolimits_{z \downarrow 0}
    \fkx_\phi (z) & = 0.
\end{alignat*}
Consequently, $\fkx_\phi$ defines an increasing homeomorphism from $(\lambda_+, \infty)$ onto $(1 + \sqrt{\phi}, \infty)$ and, when $0 < \phi < 1$, from $(0, \lambda_-)$ onto $(0, 1 - \sqrt{\phi})$.

The next result uses $\fkx_\phi$ to compare the upper and lower eigenvalue counts of $\bfR_n$ with those of $\bfL_{\caB}$. For a symmetric $\bfA \in \bbr^{d \times d}$ and $a \in \bbr$, define the strict upper and lower eigenvalue counts by
\begin{equation}
    \fkn_+^{\bfA} (a)
    := \abs[\big]{\{i \in \dbraks{d} : \lambda_i (\bfA) > a\}}, \
    \qquad
    \fkn_-^{\bfA} (a)
    := \abs[\big]{\{i \in \dbraks{d} : \lambda_i (\bfA) < a\}}.
    \label{def:eigenvalue-counts}
\end{equation}

% *******************************************
\begin{proposition}
\label{prop:spectral-transfer}
% *******************************************
Fix an integer $k \geq 2$ and assume \eqref{eqn:basic-assumption} and \eqref{eqn:analytic-tail-envelope}. Let $\bbk_+ \subset (\lambda_+, \infty)$ and $\bbk_- \subset (0, \lambda_-)$ be compact intervals. For the lower count bounds, additionally assume $\phi \in (0, 1)$. Then there is a deterministic sequence $\eps_n \downarrow 0$ such that, with probability tending to one, the following bounds hold simultaneously,
\begin{subequations}
\begin{align}
    \fkn_+^{\bfL_{\caB}} (\fkx_\phi (z) + \eps_n)
    \leq \fkn_+^{\bfR_n} (z)
    \leq \fkn_+^{\bfL_{\caB}} (\fkx_\phi (z) - \eps_n),
    \qquad \forall z \in \bbk_+,
    \label{eqn:spectral-count-sandwich} \\    
    \fkn_-^{\bfL_{\caB}} (\fkx_\phi (z) - \eps_n)
    \leq \fkn_-^{\bfR_n} (z)
    \leq \fkn_-^{\bfL_{\caB}} (\fkx_\phi (z) + \eps_n),
    \qquad \forall z \in \bbk_-.
    \label{eqn:lower-spectral-count-sandwich}
\end{align}
\end{subequations}
\end{proposition}

% *******************************************
\begin{proof}[Proof of Proposition \ref{prop:spectral-transfer}]
% *******************************************
We focus on the proof of \eqref{eqn:spectral-count-sandwich}, as the proof of \eqref{eqn:lower-spectral-count-sandwich} is completely analogous. By the first estimate in \eqref{eqn:unified-core-trace} and Proposition \ref{prop:unified-resolvent}, we may choose a deterministic sequence $\eps_n \downarrow 0$ such that, with probability tending to one,
\begin{equation}
    \min \bbk_+ - \lambda_1 (\bfQ_{\caT}) 
    \geq c_+, 
    \qquad
    \sup\nolimits_{z \in \bbk_+}
    \norm*{\bfY_{\caB} (\bfQ_{\caT} - z \bfI)^{-1} \bfY_{\caB}^{\top}
    - \fkm_\phi (z) \bfL_{\caB}} 
    < c_0 \eps_n / 2.
    \label{def:event-reduction}
\end{equation}
Here $c_0 \equiv c_0 (\bbk_+) := \inf\nolimits_{z \in \bbk_+} \, \abs{\fkm_\phi (z)} > 0$. Throughout this proof, we work on the event \eqref{def:event-reduction}.

Up to a simultaneous permutation of the rows and columns, we can write
\begin{equation*}
    \bfR_n - z \bfI = \begin{pmatrix}
        \bfY_{\caT} \bfY_{\caT}^\top - z \bfI 
        & \bfY_{\caT} \bfY_{\caB}^\top \\
        \bfY_{\caB} \bfY_{\caT}^\top 
        & \bfY_{\caB} \bfY_{\caB}^\top - z \bfI
    \end{pmatrix}.
\end{equation*}
The first bound in \eqref{def:event-reduction} ensures that the upper-left block is negative definite for every $z \in \bbk_+$. By the resolvent identity, its Schur complement is
\begin{equation*}
    (\bfY_{\caB} \bfY_{\caB}^{\top} - z \bfI)
    - \bfY_{\caB} \bfY_{\caT}^{\top}
    (\bfR_\caT - z \bfI)^{-1}
    \bfY_{\caT} \bfY_{\caB}^{\top}
    = - z \bfF_n (z),
\end{equation*}
where we define
\begin{equation*}
    \bfF_n (z) := \bfI
    + \bfY_{\caB} (\bfQ_{\caT} - z \bfI)^{-1} \bfY_{\caB}^\top
    \in \bbr^{\abs{\caB} \times \abs{\caB}}.
\end{equation*}
Since $z > 0$, Sylvester's law of inertia gives the counting identity
\begin{equation}
    \fkn_+^{\bfR_n} (z) 
    = \abs[\big]{\{1 \leq i \leq \abs{\caB}: 
    \lambda_i (\bfF_n (z)) < 0\}}.
    \label{eqn:exact-upper-Schur-count}
\end{equation}

We next use the second estimate in \eqref{def:event-reduction} to compare $\bfF_n (z)$ with $\bfI + \fkm_\phi (z) \bfL_{\caB}$. Since $\fkm_\phi (z) \leq - c_0$ and $\fkm_\phi (z) \fkx_\phi (z) = -1$ for every $z \in \bbk_+$, Weyl's inequality yields the following implications for each $1 \leq i \leq \abs{\caB}$,
\begin{alignat*}{2}
    \lambda_i (\bfL_{\caB}) > \fkx_\phi (z) + \eps_n
    &\quad \Longrightarrow \quad
    \lambda_{\abs{\caB}+1-i} \pars[\big]{
    \bfI + \fkm_\phi (z) \bfL_{\caB} } < - c_0 \eps_n 
    &&\quad \Longrightarrow \quad
    \lambda_{\abs{\caB}+1-i} (\bfF_n (z)) < - c_0 \eps_n / 2 < 0, \\
    \lambda_i (\bfL_{\caB}) \leq \fkx_\phi (z) - \eps_n
    &\quad \Longrightarrow \quad
    \lambda_{\abs{\caB}+1-i} \pars[\big]{
    \bfI + \fkm_\phi (z) \bfL_{\caB} } \geq c_0 \eps_n
    &&\quad \Longrightarrow \quad
    \lambda_{\abs{\caB}+1-i} (\bfF_n (z)) \geq c_0 \eps_n / 2 > 0.
\end{alignat*}
Both implications hold uniformly for $z \in \bbk_+$ on the event \eqref{def:event-reduction}. Hence,
\begin{equation*}
    \fkn_+^{\bfL_{\caB}} (\fkx_\phi (z) + \eps_n)
    \leq \abs[\big]{\{1 \leq i \leq \abs{\caB}: 
    \lambda_i (\bfF_n (z)) < 0\}}
    \leq \fkn_+^{\bfL_{\caB}} (\fkx_\phi (z) - \eps_n).
\end{equation*}
Combining this bound with \eqref{eqn:exact-upper-Schur-count} proves \eqref{eqn:spectral-count-sandwich}.
\end{proof}

For the extreme eigenvalues, the reduction takes a particularly simple form.

% *******************************************
\begin{corollary}
\label{cor:extremal-transfer}
% *******************************************
Fix an integer $k \geq 2$ and assume \eqref{eqn:basic-assumption} and \eqref{eqn:analytic-tail-envelope}.
\begin{subequations}
\begin{enumerate}[label = (\roman*)]
    \item \label{item:extremal-transfer-edge} If $\lambda_1 (\bfL_{\caB}) \leq 1 + \sqrt{\phi} + o_{\bbp} (1)$, then
    \begin{equation}
        \lambda_1 (\bfR_n) \pconv \lambda_+.
        \label{eqn:extremal-transfer-edge}
    \end{equation}
    \item \label{item:extremal-transfer-lower-edge} If $\phi \in (0, 1)$ and $\lambda_{\abs{\caB}} (\bfL_{\caB}) \geq 1 - \sqrt{\phi} + o_{\bbp} (1)$, then
    \begin{equation}
        \lambda_{p_n} (\bfR_n) \pconv \lambda_-.
        \label{eqn:extremal-transfer-lower-edge}
    \end{equation}
    \item \label{item:extremal-transfer-outlier} If $\lambda_1 (\bfL_{\caB}) = x + o_{\bbp} (1)$ for a finite constant $x > 1 + \sqrt{\phi}$, then
    \begin{equation}
        \lambda_1 (\bfR_n) \pconv \theta_\phi (x).
        \label{eqn:extremal-transfer-outlier}
    \end{equation}
\end{enumerate}    
\end{subequations}
\end{corollary}

% *******************************************
\begin{proof}[Proof of Corollary \ref{cor:extremal-transfer}]
% *******************************************
For part \ref{item:extremal-transfer-edge}, recall $\abs{\lambda_1 (\bfR_\caT) - \lambda_+}= o_{\bbp} (1)$ from \eqref{eqn:eig-typical-to-edge}. Cauchy's interlacing inequality therefore gives the lower bound $\lambda_1 (\bfR_n) \geq \lambda_+ - o_{\bbp} (1)$. For the upper bound, we use the condition $\lambda_1 (\bfL_{\caB}) \leq 1 + \sqrt{\phi} + o_{\bbp} (1)$. For every fixed $\delta > 0$, strict monotonicity of $\fkx_\phi$ in \eqref{eqn:fkx-increasing-homeomorphism} implies $\fkx_\phi (\lambda_+ + \delta) \geq 1 + \sqrt{\phi} + c(\delta)$ for some constant $c(\delta) > 0$. The upper bound in \eqref{eqn:spectral-count-sandwich} then gives, with probability tending to one,
\begin{equation*}
    \fkn_+^{\bfR_n} (\lambda_+ + \delta) 
    \leq \fkn_+^{\bfL_{\caB}} \pars[\big]{1 + \sqrt{\phi} + c(\delta) - \eps_n}
    \leq \fkn_+^{\bfL_{\caB}} \pars[\big]{1 + \sqrt{\phi} + c(\delta)/2}
    = 0
\end{equation*}
Since $\delta > 0$ is arbitrary, this implies the desired upper bound $\lambda_1 (\bfR_n) \leq \lambda_+ + o_{\bbp} (1)$. Combining this with the lower bound proves \eqref{eqn:extremal-transfer-edge}. The proof of part \ref{item:extremal-transfer-lower-edge} is analogous, using \eqref{eqn:lower-spectral-count-sandwich} in place of \eqref{eqn:spectral-count-sandwich}. 

For part \ref{item:extremal-transfer-outlier}, suppose that $\lambda_1 (\bfL_{\caB}) \to x > 1 + \sqrt{\phi}$ in probability. Fix $\delta > 0$ small enough that $\theta_\phi (x) - \delta > \lambda_+$. By \eqref{eqn:fkx-increasing-homeomorphism}, there exists some constant $c(\delta) > 0$ such that
\begin{equation*}
    \fkx_\phi (\theta_\phi (x) - \delta)
    + c(\delta)
    \leq x
    \leq \fkx_\phi (\theta_\phi (x) + \delta)
    - c(\delta).
\end{equation*}
Consequently, by \eqref{eqn:spectral-count-sandwich} and $\lambda_1 (\bfL_{\caB}) \to x$ in probability, with probability tending to one,
\begin{align*}
    \fkn_+^{\bfR_n} (\theta_\phi (x) + \delta)
    & \leq \fkn_+^{\bfL_{\caB}} (x + c(\delta) - \eps_n)
    \leq \fkn_+^{\bfL_{\caB}} (x + c(\delta)/2)
    = 0, \\
    \fkn_+^{\bfR_n} (\theta_\phi (x) - \delta)
    & \geq \fkn_+^{\bfL_{\caB}} (x - c(\delta) + \eps_n)
    \geq \fkn_+^{\bfL_{\caB}} (x - c(\delta)/2)
    \geq 1.
\end{align*}
Since $\delta > 0$ is arbitrary, combining these two bounds proves \eqref{eqn:extremal-transfer-outlier}.
\end{proof}

In later sections, applying Proposition \ref{prop:spectral-transfer} or Corollary \ref{cor:extremal-transfer} \emph{at order $\ell$} means replacing $k$ by a fixed integer $\ell \geq 2$ in the hypotheses and construction. Specifically, we require the tail bound \eqref{eqn:analytic-tail-envelope} with $\alpha = \alpha_\ell = 2 + 2 / \ell$ and choose a fixed truncation exponent $0 < \tau_\ell < 1 / [\alpha_\ell \ell (\ell + 1)]$. All auxiliary objects, including the graph $\caG$ and the effective Gram matrix $\bfL_{\caB}$, are constructed using the same threshold $\fku_n = n^{1/2 - \tau_\ell}$. For notational simplicity, we frequently suppress the dependence of these objects on $\ell$ and $\tau_\ell$.

% *******************************************
\section{Noncritical asymptotics}
\label{sec:collisions}
% *******************************************

This section establishes the noncritical asymptotic results in Section \ref{sec:introduction} except Theorem \ref{thm:tall}. Section \ref{subsec:collision-components} describes the block diagonal structure of the effective Gram matrix $\bfL_{\caB}$ and analyzes the spectrum of its blocks associated with the star components of the large-entry graph $\caG$. Sections then \ref{subsec:upper-edge-proof} and \ref{subsec:collision-strong-stars} establish upper and lower bounds for $\lambda_1 (\bfL_{\caB})$, respectively. Under the subcritical tail condition at order $k$, Lemma \ref{lem:collision-cap} gives $\lambda_1 (\bfL_{\caB}) \leq k - 1 + o_{\bbp} (1)$. While Lemma \ref{lem:strong-stars} provides the lower bound $k - o_{\bbp} (1)$ for both $\lambda_1 (\bfL_{\caB})$ and $\lambda_1 (\bfR_n)$ under the supercritical tail condition at order $k$. 

Together with Corollary \ref{cor:extremal-transfer}, these bounds yield the upper MP edge limit in Theorem \ref{thm:upper-edge-convergence}, the deterministic upper outlier limits in Theorem \ref{thm:collision-layer}, and the lower MP edge limit in Theorem \ref{thm:wide-subcritical} under their respective tail conditions. The construction in the proof of Lemma \ref{lem:strong-stars} also gives direct proofs of $\lambda_1 (\bfR_n) \to \infty$ in Theorem \ref{thm:all-orders-divergence} and $\lambda_{p_n} (\bfR_n) \to 0$ in Theorem \ref{thm:wide-supercritical}, without reducing to the effective Gram matrix $\bfL_{\caB}$.

% *******************************************
\subsection{Component blocks}
\label{subsec:collision-components}
% *******************************************

The spectral reduction results in Section \ref{subsec:spectral-transfer} allow us to study spectral outliers of $\bfR_n$ through the effective Gram matrix $\bfL_{\caB}$ defined in \eqref{eqn:unified-component-matrix}. This matrix depends only on the large entries and, as illustrated below, has a block diagonal structure determined by the connected components of the large-entry graph $\caG$.

Regard each connected component of $\caG$ as a bipartite subgraph, including its edges. Let
\begin{equation}
    \bbg := \{ \caC :
    \caC \text{ is a connected component of }
    \caG \text{ containing at least one edge} \}.
    \label{def:collision-component-collection}
\end{equation}
For $\caC \in \bbg$, let $\fke (\caC)$ denote its number of edges, and let $\fkr (\caC)$ and $\fkc (\caC)$ denote its numbers of row and column vertices, respectively. Each atypical row $i \in \caB$ belongs to exactly one component in $\bbg$. The corresponding row index sets therefore partition $\caB$,
\begin{equation}
    \caB = \bigsqcup \nolimits_{\caC \in \bbg}
    \caI (\caC),
    \qquad
    \caI (\caC) := \{ i \in \caB : 
    i \text{ is a row vertex of } \caC \}.
    \label{eqn:component-row-partition}
\end{equation}
Note that $\fkr (\caC) = \abs{\caI (\caC)}$. Every row vertex in $\caC$ is incident to at least one edge, giving $\fkr (\caC) \leq \fke (\caC)$. Let $\fkS_{\caI (\caC)} := (\fks_{i \mu})_{i \in \caI (\caC), \mu \in \dbraks{n}}$ be the restriction of the sparse matrix $\fkS_{\caB}$ from \eqref{eqn:unified-S-W-D} to the rows in $\caI (\caC)$. Following \eqref{eqn:unified-component-matrix}, define the corresponding block by
\begin{equation}
    \bfL_{\caI (\caC)} := \bfI
    + \mathcal{D}_\perp \braks{\fkS_{\caI (\caC)} \fkS_{\caI (\caC)}^{\top}}
    = \fkS_{\caI (\caC)} \fkS_{\caI (\caC)}^{\top}
    + \diag \pars[\big]{(n / \hbar_i^2)_{i \in \caI (\caC)}}.
    \label{def:component-block-matrix}
\end{equation}
This is the principal submatrix of $\bfL_{\caB}$ indexed by $\caI (\caC)$. For distinct rows $i, j \in \caB$, the off-diagonal entry $L_{i j}$ is the inner product of their sparse rows in $\fkS_{\caB}$ and can be nonzero only if $\caA_i \cap \caA_j \neq \varnothing$. Since rows in different connected components cannot share a large-entry column, the off-diagonal blocks between distinct component index sets vanish. Thus, after reordering the atypical rows, we have
\begin{equation}
    \bfL_{\caB} = \bigoplus \nolimits_{\caC \in \bbg}
    \bfL_{\caI (\caC)}.
    \label{eqn:unified-component-cap}
\end{equation}
Consequently,
\begin{equation}
    \lambda_1 (\bfL_{\caB}) 
    = \max\nolimits_{\caC \in \bbg}
    \lambda_1 (\bfL_{\caI (\caC)}),
    \qquad
    \lambda_{\abs{\caB}} (\bfL_{\caB}) 
    = \min\nolimits_{\caC \in \bbg}
    \lambda_{\fkr (\caC)} (\bfL_{\caI (\caC)}).
\end{equation}
Each block $\bfL_{\caI (\caC)}$ is positive definite with unit diagonal, and thus $\operatorname{Tr} \bfL_{\caI (\caC)} = \fkr (\caC)$. This also gives a bound on the norm of each block. Recall the event $\tilde{\Omega}_{k, n}$ from Lemma \ref{lem:unified-graph} \ref{item:unified-graph-components}. On this event, every connected component of $\caG$ has at most $k$ edges, and therefore
\begin{equation}
    \bfL_{\caI (\caC)} \geq 0, 
    \qquad
    \operatorname{Tr} \bfL_{\caI (\caC)}
    = \fkr (\caC),
    \qquad
    \norm{\bfL_{\caI (\caC)}}
    \leq \fkr (\caC) \leq \fke (\caC)
    \leq k
    \label{eqn:basic-prop-diag-block}
\end{equation}

If $\caC$ contains only one row vertex $i$, then $\caI (\caC) = \{i\}$ and $\bfL_{\caI (\caC)}$ is the $1 \times 1$ matrix with entry $1$, even when row $i$ contains several large entries. If $\caI (\caC) = \{i, j\}$, then
\begin{equation}
    \bfL_{\caI (\caC)}
    = \begin{pmatrix} 1 & L_{i j} \\ L_{i j} & 1 \end{pmatrix},
    \qquad
    L_{i j} = \sum \nolimits_{\mu = 1}^n
    \fks_{i \mu} \fks_{j \mu}
    = \frac{1}{\hbar_i \hbar_j} 
    \sum \nolimits_{\mu \in \caA_i \cap \caA_j}
    \psi_{i \mu} \psi_{j \mu}
    \label{eqn:2by2-block}
\end{equation}
The eigenvalues of this block are $1 + \abs{L_{i j}}$ and $1 - \abs{L_{i j}}$. Under the tail condition \eqref{eqn:analytic-tail-envelope} with $k=2$, i.e., $t^3 \bar{F}_{\xi} (t) \leq C$, these are the only nontrivial blocks needed to analyze the extreme eigenvalues of $\bfR_n$. In fact, in this case, every connected component of $\caG$ has at most two edges with probability tending to one by Lemma \ref{lem:unified-graph} \ref{item:unified-graph-components}, and the expression of $L_{ij}$ could be further simplified as $\caA_i \cap \caA_j$ then contains exactly one column index.

However, for general $k$, a component may connect several rows through more than one column. For example, when $k \geq 4$, a component may have the pattern
\begin{equation*}
    \caA_{i_1} = \{\mu\},
    \qquad 
    \caA_{i_2} = \{\mu, \nu\},
    \qquad
    \caA_{i_3} = \{\nu\}
\end{equation*}
Here rows $i_1$ and $i_2$ are connected through column $\mu$, while rows $i_2$ and $i_3$ are connected through column $\nu$. All three rows belong to the same block, although rows $i_1$ and $i_3$ share no large-entry column. Despite this variety of component structures, the following lemma shows that we can focus on stars centered at a column vertex with $k$ row leaves.

% *******************************************
\begin{lemma}
\label{lem:collision-components}
% *******************************************
Define the collection of stars
\begin{equation}
    \bbg_{k\star} := \{\caC \in \bbg : \caC \text{ is a star centered at a column vertex, with } k \text{ row leaves}\}.
    \label{def:collection-stars}
\end{equation}
Suppose $\tilde{\Omega}_{k, n}$ holds and let $\caC \in \bbg$. Then 
\begin{equation}
    \lambda_1 (\bfL_{\caI (\caC)}) > k - 1
    \qquad \Longrightarrow \qquad
    \caC \in \bbg_{k\star}.
\end{equation}
\end{lemma}

% *******************************************
\begin{proof}[Proof of Lemma \ref{lem:collision-components}]
% *******************************************
Recall that on $\tilde{\Omega}_{k, n}$, every connected component of $\caG$ has at most $k$ edges. Thus, by \eqref{eqn:basic-prop-diag-block}, the condition $\lambda_1 (\bfL_{\caI (\caC)}) > k - 1$ forces $\fkr (\caC) = \fke (\caC) = k$. Since $\caC$ is connected, $\fkr (\caC) + \fkc (\caC) \leq \fke (\caC) + 1$, so $\fkc (\caC) = 1$. Consequently, $\caC$ is a star centered at its unique column vertex, with $k$ row leaves.
\end{proof}

We can express the block of $\bfL_{\caB}$ associated with a star in terms of the matrix $\bfM_k$ introduced in \eqref{def:star-Gram}. Let $\caC \in \bbg_{k\star}$ have center $\mu \in \dbraks{n}$ and row leaves $1 \leq i_1 < \cdots < i_k \leq p_n$. Define its raw amplitude vector by
\begin{equation}
    \bfa (\caC) := (a_r (\caC))_{r = 1}^k,
    \qquad
    a_r (\caC) := \frac{1}{\sqrt{n}} \abs{\psi_{i_r \mu}}.
    \label{def:raw-star-amplitudes}
\end{equation}
Since each leaf $i_r$ has $\caA_{i_r} = \{\mu\}$, the normalizer in \eqref{eqn:unified-frozen-denominator} satisfies $\hbar_{i_r}^2 = n + \abs{\psi_{i_r \mu}}^2$. In particular, we have $\abs{\fks_{i_r \mu}} = \varphi (a_r (\caC))$ by definition of $\varphi$ in \eqref{def:star-amplitude-single}. Together with definition of $\bfM_k$ in \eqref{def:star-Gram} and the block $\bfL_{\caI (\caC)}$ in \eqref{def:component-block-matrix}, we therefore obtain
\begin{equation}
    \bfL_{\caI (\caC)}
    = \boldsymbol{\varsigma} (\caC)
    \, \bfM_k (\bfa (\caC))
    \boldsymbol{\varsigma} (\caC)^\top,
    \qquad
    \boldsymbol{\varsigma} (\caC)
    := \diag \pars[\big]{ (\operatorname{sgn} (\psi_{i_r \mu}))_{r = 1}^k },
    \label{eqn:star-sign-conjugation}
\end{equation}
where $\boldsymbol{\varsigma} (\caC)$ is a diagonal sign matrix. This sign conjugation has no effect on the spectrum, i.e., $\bfL_{\caI (\caC)}$ and $\bfM_k (\bfa (\caC))$ have the same eigenvalues, which depend only on the entry magnitudes $\bfa (\caC)$. The next lemma gives the basic properties of the maps $\bfa \mapsto \bfM_k (\bfa)$ and $\bfa \mapsto \gamma_k^{+} (\bfa) = \lambda_1 (\bfM_k (\bfa))$. We note that it also gives the needed properties of $\gamma_2^{-} (\bfa)$ as
\begin{equation}
    \gamma_2^{-}(a_1,a_2)
    = 1 - \varphi(a_1) \varphi(a_2)
    = 2 - \gamma_2^{+}(a_1,a_2).
\end{equation}

% *******************************************
\begin{lemma}
\label{lem:star-spectrum}
% *******************************************
Fix $k \geq 2$ and let $\bfa = (a_r)_{r = 1}^k \in [0, \infty)^k$. The following statements hold.
\begin{enumerate}[label = (\roman*)]
    \item \label{item:star-spectrum-bounds} The eigenvalues of $\bfM_k (\bfa)$ satisfy
    \begin{equation}
        \lambda_2 (\bfM_k (\bfa)) \leq 1 \leq \gamma_k^{+} (\bfa) \leq k,
        \qquad
        \gamma_k^{+} (a \mathbf{1}_k) = 1 + (k - 1) \frac{a^2}{1 + a^2},
        \label{eqn:equal-amplitude-star}
    \end{equation}
    where the equality holds for every $a \geq 0$. Moreover, $\gamma_k^{+} (\bfa) \to k$ as every $a_r \to \infty$.
    \item \label{item:star-spectrum-secular} If at least two coordinates of $\bfa$ are positive, then $\gamma_k^{+} (\bfa)$ is the unique root $x > 1$ of
    \begin{equation}
        \sum\nolimits_{r = 1}^k \frac{a_r^2}{x - 1 + x a_r^2} = 1.
        \label{eqn:unified-star-secular}
    \end{equation}
    \item \label{item:star-spectrum-monotonicity} The largest eigenvalue $\gamma_k^{+} (\bfa)$ is continuous on $[0, \infty)^k$. On $(0, \infty)^k$, it is strictly increasing in each coordinate and satisfies $1 < \gamma_k^{+} (\bfa) < k$.
    \item \label{item:star-spectrum-amplitudes} For every $\eta > 0$, there is a constant $c_{\ref{eqn:raw-star-macroscopic}} (k, \eta) > 0$ such that
    \begin{equation}
        \gamma_k^{+} (\bfa) \geq k - 1 + \eta
        \quad \Longrightarrow \quad
        \min\nolimits_{r \in \dbraks{k}} a_r
        \geq c_{\ref{eqn:raw-star-macroscopic}} (k, \eta).
        \label{eqn:raw-star-macroscopic}
    \end{equation}
\end{enumerate}
\end{lemma}

% *******************************************
\begin{proof}[Proof of Lemma \ref{lem:star-spectrum}]
% *******************************************
We work with normalized coordinates $\bfs = (s_r)_{r = 1}^k \in [0, 1]^k$ and
\begin{equation*}
    \tilde{\bfM} \equiv \tilde{\bfM}_k (\bfs) := \bfD + \bfs \bfs^\top,
    \qquad
    \bfD \equiv \bfD_k (\bfs) := \diag (1 - s_1^2, \ldots, 1 - s_k^2).
\end{equation*}
For $\bfs = \varphi (\bfa)$, we have $\bfM_k (\bfa) = \tilde{\bfM}_k (\bfs)$ and $\gamma_k^{+} (\bfa) = \lambda_1 (\tilde{\bfM}_k (\bfs))$.

For part \ref{item:star-spectrum-bounds}, the bounds $0 \leq \bfD \leq \bfI$ imply $\tilde{\bfM} \geq 0$. Together with $\operatorname{Tr} \tilde{\bfM} = k$, this gives
\begin{equation*}
    1 = {\operatorname{Tr} \tilde{\bfM}} / {k}
    \leq \lambda_1 (\tilde{\bfM}) \leq \operatorname{Tr} \tilde{\bfM} = k,
    \qquad
    \lambda_2 (\tilde{\bfM})
    \leq \lambda_1 (\bfD) + \lambda_2 (\bfs \bfs^\top)
    \leq 1,
\end{equation*}
where the bound on $\lambda_2 (\tilde{\bfM})$ follows from Weyl's inequality and $\lambda_2 (\bfs \bfs^\top) = 0$. This proves the eigenvalue bounds in \eqref{eqn:equal-amplitude-star}. For equal coordinates $\bfs = s \mathbf{1}_k$ with $s \in [0,1]$, we have
\begin{equation*}
    \tilde{\bfM}_k (s \mathbf{1}_k)
    = (1 - s^2) \bfI + s^2 \mathbf{1}_k \mathbf{1}_k^\top.
\end{equation*}
The largest eigenvalue is therefore $1 + (k - 1) s^2$, with eigenvector $\mathbf{1}_k$. Taking $s = \varphi (a)$ gives the second formula in \eqref{eqn:equal-amplitude-star}. Finally, if every $a_r \to \infty$, then $\bfs \to \mathbf{1}_k$ and $\tilde{\bfM}_k (\bfs) \to \mathbf{1}_k \mathbf{1}_k^\top$, which yields $\gamma_k^{+} (\bfa) \to k$.

For part \ref{item:star-spectrum-secular}, let $x \in (1, \infty)$ and put
\begin{equation*}
    H (x, \bfs)
    := \bfs^\top [x \bfI - \bfD_k (\bfs)]^{-1} \bfs
    = \sum\nolimits_{r = 1}^k \frac{s_r^2}{x - 1 + s_r^2}.
\end{equation*}
The matrix $x \bfI - \bfD$ is invertible, and the matrix determinant lemma gives
\begin{equation}
    \det (x \bfI - \tilde{\bfM})
    = \braks{1 - \bfs^\top (x \bfI - \bfD)^{-1} \bfs}
    \det (x \bfI - \bfD)
    = \braks{1 - H (x, \bfs)}
    \prod\nolimits_{r = 1}^k (x - 1 + s_r^2).
    \label{eqn:determinant-Mk}    
\end{equation}
Note that ${s_r^2} / ({x - 1 + s_r^2}) \to \bbone \{s_r > 0\}$ as $x \downarrow 1$. Hence, if at least two coordinates of $\bfs$ are positive, then $\lim_{x \downarrow 1} H (x, \bfs) \geq 2$. In addition, we have $\lim_{x \to \infty} H (x, \bfs) = 0$ and
\begin{equation}
    \partial_x H (x, \bfs) = -\sum\nolimits_{r = 1}^k
    \pars[\bigg]{\frac{s_r}{x - 1 + s_r^2}}^2 < 0.
    \label{eqn:monotonicity-H}
\end{equation}
Consequently, the equation $H (x, \bfs) = 1$ has a unique root in $(1, \infty)$, which equals $\lambda_1 (\tilde{\bfM})$ by \eqref{eqn:determinant-Mk} and part \ref{item:star-spectrum-bounds}. Substituting $s_r = \varphi (a_r)$ into this equation gives \eqref{eqn:unified-star-secular}. 

For part \ref{item:star-spectrum-monotonicity}, continuity of $\lambda_1 (\tilde{\bfM}_k (\bfs))$ and $\varphi$ gives continuity of $\gamma_k^{+}$. For $\bfs \in (0, 1)^k$, set $x = \lambda_1 (\tilde{\bfM}_k (\bfs))$. Part \ref{item:star-spectrum-secular} gives $x > 1$ and $H (x, \bfs) = 1$. Since $H$ is strictly decreasing in $x$ and $H (k, \bfs) < \sum_{r = 1}^k 1 / k = 1$, we also have $x < k$. The derivative $\partial_x H (x, \bfs)$ is negative by \eqref{eqn:monotonicity-H}, so the implicit function theorem yields
\begin{equation*}
    \partial_{s_r} \lambda_1 (\tilde{\bfM}_k (\bfs))
    = -\frac{\partial_{s_r} H (x, \bfs)}{\partial_x H (x, \bfs)}
    = \frac{2 (x - 1) s_r}{(x - 1 + s_r^2)^2}
    \braks[\bigg]{\sum\nolimits_{v = 1}^k
    \pars[\bigg]{\frac{s_v}{x - 1 + s_v^2}}^2}^{-1} > 0.
\end{equation*}
Since $\varphi$ is strictly increasing on $(0, \infty)$, this proves the corresponding assertions for $\gamma_k^{+}$.

For part \ref{item:star-spectrum-amplitudes}, suppose $x = \lambda_1 (\tilde{\bfM}_k (\bfs)) \geq k - 1 + \eta > 1$. Note that if at most one coordinate of $\bfs$ is positive, then $s_u s_v = 0$ for every $u \neq v$, and hence $\tilde{\bfM} = \bfI$. Therefore, at least two coordinates of $\bfs$ are positive, and we can use the equation $H (x, \bfs) = 1$ to get, for each $r \in \dbraks{k}$,
\begin{equation*}
    1 = \frac{s_r^2}{x - 1 + s_r^2}
    + \sum_{v \neq r} \frac{s_v^2}{x - 1 + s_v^2}
    \leq \frac{s_r^2}{x - 1 + s_r^2}
    + \frac{k - 1}{x},
\end{equation*}
where we also used $s_v^2 \leq 1$. Multiplying both sides by $x (x - 1 + s_r^2)$ and rearranging yields
\begin{equation*}
    (k - 1) s_r^2
    \geq (x - k + 1) (x - 1)
    \geq \eta (k - 2 + \eta)
    \geq \eta^2.
\end{equation*}
Since $a_r \geq \varphi (a_r) = s_r$, one may take $c_{\ref{eqn:raw-star-macroscopic}} (k, \eta) = \sqrt{\eta^2 / (k - 1)} > 0$.
\end{proof}

% *******************************************
\subsection{The subcritical regime}
\label{subsec:upper-edge-proof}
% *******************************************

Note that the event $\tilde{\Omega}_{k, n}$ still allows stars with $k$ row leaves, whose blocks may have eigenvalues above $k - 1$. The next lemma uses the subcritical tail conditions \eqref{cond:upper-tail-subcritical} or \eqref{cond:wide-tail-sub} to rule out eigenvalues exceeding this level by any fixed positive amount, with probability tending to one.

% *******************************************
\begin{lemma}
\label{lem:collision-cap}
% *******************************************
Suppose that \eqref{eqn:basic-assumption} holds, and fix an integer $k \geq 2$ such that $t^{\alpha_k} \bar{F}_{\xi} (t) \to 0$ as $t \to \infty$. Construct $\bfL_{\caB}$ using the threshold $\fku_n = n^{1/2-\tau_k}$ with $\tau_k$ satisfying \eqref{eqn:unified-truncation}. Then
\begin{equation}
    \lambda_1 (\bfL_{\caB}) \leq k - 1 + o_{\bbp} (1).
    \label{eqn:subcritical-local-cap}
\end{equation}
\end{lemma}

% *******************************************
\begin{proof}[Proof of Lemma \ref{lem:collision-cap}]
% *******************************************
Fix $\eta \in (0, 1)$ and condition on a realization of $\caH_n$ in $\tilde{\Omega}_{k, n}$. This fixes the large-entry graph $\caG$ while leaving the large-entry values $\psi_{i \mu}$ independent with common law $\caL (\psi_n)$ from \eqref{eqn:alpha3-resampled-psi}.

Fix arbitrary $\eta > 0$. By Lemma \ref{lem:collision-components}, any block with $\lambda_1 (\bfL_{\caI (\caC)}) \geq k - 1 + \eta$ must correspond to a star $\caC$ centered at a column vertex with $k$ row leaves. Lemma \ref{lem:star-spectrum} \ref{item:star-spectrum-amplitudes} and \eqref{eqn:star-sign-conjugation} then give $\min\nolimits_{r \in \dbraks{k}} a_r (\caC) \geq 2c$, where $c := c_{\ref{eqn:raw-star-macroscopic}} (k, \eta) / 2 > 0$. If the star has center $\mu \in \dbraks{n}$ and leaves $i_1, \ldots, i_k \in \dbraks{p_n}$, then \eqref{def:raw-star-amplitudes} implies
\begin{equation*}
    \abs{\psi_{i_r \mu}}
    = \sqrt{n} a_r (\caC)
    > c \sqrt{n},
    \qquad r \in \dbraks{k}.
\end{equation*}
Since $c \sqrt{n} > \fku_n$ for all sufficiently large $n$, the conditional large-entry law \eqref{eqn:alpha3-resampled-psi} gives
\begin{equation*}
    \bbp \{\abs{\psi_n} > c \sqrt{n}\}
    = \fkq_n^{-1} \bar{F}_{\xi}
    \pars{c \sqrt{n}}.
\end{equation*}
Define the $\caH_n$-measurable counting variable
\begin{equation}
    Z_{k, n}
    := \sum\nolimits_{\mu = 1}^n
    \sum\nolimits_{1 \leq i_1 < \cdots < i_k \leq p_n}
    \prod\nolimits_{r = 1}^k \chi_{i_r \mu}.
    \label{def:Zk-star-count}
\end{equation}
On $\tilde{\Omega}_{k, n}$, each nonzero summand in $Z_{k, n}$ identifies a star component with $k$ row leaves, since the component cannot contain any additional edges. Conditional independence of the $k$ large-entry variables in each star and a union bound then give, for all sufficiently large $n$,
\begin{equation}
    \bbp^{\caH} \{\lambda_1 (\bfL_{\caB}) > k - 1 + \eta\}
    \leq Z_{k, n}
    \braks[\big]{\fkq_n^{-1} \bar{F}_{\xi} \pars[\big]{c \sqrt{n}}}^k.
    \label{eqn:subcritical-conditional-cap}
\end{equation}
Since the indicators $\chi_{i \mu}$ are i.i.d. with law $\mathrm{Bern} (\fkq_n)$,
\begin{equation*}
    \bbe Z_{k, n} = n \binom{p_n}{k} \fkq_n^k.
\end{equation*}
Averaging the conditional bound \eqref{eqn:subcritical-conditional-cap} over $\caH_n$ thus cancels the factor $\fkq_n^{-k}$ and yields
\begin{align*}
    \bbp \{\lambda_1 (\bfL_{\caB}) > k - 1 + \eta\}
    & \leq \bbe \braks{Z_{k, n} \bbone (\tilde{\Omega}_{k, n})}
    \braks[\big]{\fkq_n^{-1} \bar{F}_{\xi} \pars[\big]{c \sqrt{n}}}^k
    + \bbp (\tilde{\Omega}_{k, n}^c) \\
    & \leq n \binom{p_n}{k} \bar{F}_{\xi} 
    \pars{c \sqrt{n}}^k
    + o (1)
    \leq C \braks[\big]{n^{\alpha_k/2} \bar{F}_{\xi} (c \sqrt{n})}^k
    + o (1)
    = o (1)    
\end{align*}
Here Lemma \ref{lem:unified-graph} \ref{item:unified-graph-components} bounds the complementary probability of $\tilde{\Omega}_{k, n}$, while $p_n \asymp n$ and $\alpha_k = 2 + 2/k$ give the bound on the first term. The tail assumption that $t^{\alpha_k} \bar{F}_{\xi} (t) \to 0$ as $t \to \infty$ makes that term tend to zero. Since $\eta \in (0, 1)$ was arbitrary, this proves \eqref{eqn:subcritical-local-cap}.
\end{proof}

Combining this upper bound with Corollary \ref{cor:extremal-transfer} immediately leads to the convergence of the extreme eigenvalues to the corresponding MP edges.

% *******************************************
\begin{proof}[Proof of Theorem \ref{thm:upper-edge-convergence}]
% *******************************************
The tail condition \eqref{cond:upper-tail-subcritical} implies \eqref{eqn:analytic-tail-envelope} at order $k_* (\phi)$. Construct $\bfL_{\caB}$ using the truncation threshold $\fku_n = n^{1/2-\tau_k}$, and apply Lemma \ref{lem:collision-cap} to obtain
\begin{equation*}
    \lambda_1 (\bfL_{\caB}) \leq k_* (\phi) - 1 + o_{\bbp} (1)
    \leq 1 + \sqrt{\phi} + o_{\bbp} (1).
\end{equation*}
The final inequality uses $k_* (\phi) - 1 = 1 + \lfloor \sqrt{\phi} \rfloor$. Corollary \ref{cor:extremal-transfer} \ref{item:extremal-transfer-edge} then implies \eqref{eqn:upper-edge-convergence}.
\end{proof}

% *******************************************
\begin{proof}[Proof of Theorem \ref{thm:wide-subcritical}]
% *******************************************
The tail condition \eqref{cond:wide-tail-sub} implies \eqref{eqn:analytic-tail-envelope} at order $k = 2$, with $\alpha_2 = 3$. Construct $\bfL_{\caB}$ using the truncation threshold $\fku_n = n^{1/2 - \tau_2}$. On $\tilde{\Omega}_{2, n}$, every component has at most two edges. Hence each block in \eqref{eqn:unified-component-cap} is either the scalar $1$ or a $2 \times 2$ block with eigenvalues $1 \pm \abs{L_{i j}}$. See \eqref{eqn:2by2-block}. Consequently, on this event,
\begin{equation*}
    \lambda_1 (\bfL_{\caB}) - 1
    = \norm{\bfL_{\caB} - \bfI}
    = 1 - \lambda_{\abs{\caB}} (\bfL_{\caB}).
\end{equation*}
Lemma \ref{lem:collision-cap}, applied at order $k = 2$, and $\bbp (\tilde{\Omega}_{2, n}) \to 1$ therefore lead to
\begin{equation}
    \lambda_{\abs{\caB}} (\bfL_{\caB}) 
    \geq 1 - \sqrt{\phi} + o_{\bbp} (1).
\end{equation}
Applying Corollary \ref{cor:extremal-transfer} \ref{item:extremal-transfer-lower-edge} at order $k = 2$ proves \eqref{eqn:wide-subcritical-limit}.
\end{proof}

% *******************************************
\subsection{The supercritical regime}
\label{subsec:collision-strong-stars}
% *******************************************

This section provides the proof of Theorems \ref{thm:collision-layer}, \ref{thm:all-orders-divergence} and \ref{thm:wide-supercritical}. We start with Theorem \ref{thm:collision-layer}. By Corollary \ref{cor:extremal-transfer}, it suffices to show that $\lambda_1 (\bfL_{\caB}) = k + o_{\bbp} (1)$. The subcritical tail condition $t^{\alpha_{k + 1}} \bar{F}_{\xi} (t) \to 0$ in \eqref{cond:deterministic-layer} allows us to apply Lemma \ref{lem:collision-cap} at order $k + 1$, yielding $\lambda_1 (\bfL_{\caB}) \leq k + o_{\bbp} (1)$. It remains to establish the matching lower bound under the supercritical condition $t^{\alpha_{k}} \bar{F}_{\xi} (t) \to \infty$ at order $k$.

The following lemma provides this lower bound. Its proof finds $k$ large entries in a common column, each dominating the remaining entries in its row.

% *******************************************
\begin{lemma}
\label{lem:strong-stars}
% *******************************************
Suppose that \eqref{eqn:basic-assumption} holds, and fix an integer $k \geq 2$ such that $t^{\alpha_k} \bar{F}_{\xi} (t) \to \infty$ as $t \to \infty$. Construct $\bfL_{\caB}$ using the truncation threshold $\fku_n = n^{1/2-\tau}$ with $\tau \in (0, 1/2)$. Then
\begin{equation}
    \lambda_1 (\bfL_{\caB}) \geq k - o_{\bbp} (1),
    \qquad
    \lambda_1 (\bfR_n) \geq k - o_{\bbp} (1).
    \label{eqn:supercritical-local-bound}
\end{equation}
\end{lemma}

We emphasize that the definition of $\bfL_{\caB}$ remains valid without the upper tail bound \eqref{eqn:analytic-tail-envelope}. Under its stated assumptions, Lemma \ref{lem:strong-stars} applies whenever $\bfL_{\caB}$ is constructed using $\fku_n = n^{1/2-\tau}$ with any fixed $\tau \in (0,1/2)$.

% *******************************************
\begin{proof}[Proof of Lemma \ref{lem:strong-stars}]
% *******************************************
Define
\begin{equation*}
    b_n = B_n^{1 / (2 \alpha_k)},
    \qquad
    B_n = \inf\nolimits_{t \geq \sqrt{n}} \, {t^{\alpha_k} \bar{F}_{\xi} (t)}.
\end{equation*}
The supercritical assumption $t^{\alpha_k} \bar{F}_{\xi} (t) \to \infty$ gives $B_n \to \infty$ and $b_n \to \infty$. Hence,
\begin{subequations}
\label{eqn:strong-star-tail-selection}
\begin{align}
    n^{1 + 1 / k} \bar{F}_{\xi} (\sqrt{n} b_n)
    & = b_n^{-\alpha_k} \braks[\big]{
    (\sqrt{n} b_n)^{\alpha_k} \bar{F}_{\xi} (\sqrt{n} b_n)}
    \geq b_n^{-\alpha_k} B_n \geq B_n^{1/2} \to \infty, 
    \label{eqn:tail-bn-upper} \\
    n \bar{F}_{\xi} (\sqrt{n} b_n)
    & = b_n^{-2} \braks[\big]{ n b_n^2 \bar{F}_{\xi} (\sqrt{n} b_n) }
    \leq b_n^{-2} \bbe \braks[\big]{
    \abs{\xi}^2 \bbone \{\abs{\xi} > \sqrt{n} b_n\}} \to 0.
    \label{eqn:tail-bn-lower} 
\end{align}
\end{subequations}
where \eqref{eqn:tail-bn-lower} uses the finite variance assumption from \eqref{eqn:basic-assumption}.

For $i \in \dbraks{p_n}$ and $\mu \in \dbraks{n}$, define
\begin{equation}
    \Xi_{i \mu} := \curls*{\abs{x_{i \mu}} > \sqrt{n} b_n,
    \quad
    \sum\nolimits_{\nu \neq \mu} \abs{x_{i \nu}}^2 \leq 2 n}.
    \label{eqn:good-strong-star}
\end{equation}
Since $b_n \to \infty$, for each fixed row $i$ the events $\{\Xi_{i\mu}\}_{\mu=1}^n$ are pairwise disjoint for all sufficiently large $n$. We call row $i$ \emph{single-dominant} if $\Xi_{i \mu}$ holds for some $\mu \in \dbraks{n}$. Independence of the entries within a row and the weak law of large numbers give
\begin{equation*}
    \eta_n
    := \bbp (\Xi_{i \mu})
    = \bbp \{\abs{x_{i \mu}} > \sqrt{n} b_n\}
    \bbp \curls*{\frac{1}{n} \sum\nolimits_{\nu \neq \mu} 
    \abs{x_{i \nu}}^2 \leq 2}
    = (1+o(1)) \bar{F}_{\xi} (\sqrt{n} b_n).
\end{equation*}

To count collisions among single-dominant rows, define
\begin{equation*}
    \hat{Z}_{k, n} := \sum\nolimits_{\abs{\caI} = k} \hat{U}_{\caI},
    \qquad 
    \hat{U}_{\caI} := \sum\nolimits_{\mu = 1}^n
    \prod\nolimits_{i \in \caI} \bbone (\Xi_{i \mu}),
\end{equation*}
where the first sum runs over subsets $\caI \subset \dbraks{p_n}$ with $\abs{\caI} = k$. The disjointness of $\{\Xi_{i\mu}\}_{\mu=1}^n$ gives $\hat{U}_{\caI} \in \{0,1\}$, and independence of the rows gives $\bbe \hat{U}_{\caI} = n \eta_n^k$. Thus, by \eqref{eqn:tail-bn-upper},
\begin{equation}
    \bbe \hat{Z}_{k, n} = \binom{p_n}{k} n \eta_n^k
    \sim \frac{1}{k!} n p_n^k \eta_n^k
    \sim \frac{\phi^k}{k!} \braks[\big]{n^{1 + 1 / k} \bar{F}_{\xi} (\sqrt{n} b_n)}^k
    \to \infty.
    \label{eqn:expectation-Zhat}
\end{equation}
To estimate the variance of $\hat{Z}_{k, n}$, consider subsets $\caI_1, \caI_2  \subset \dbraks{p_n}$ with $\abs{\caI_1 \cap \caI_2} = \ell$. If $\ell \geq 1$ and $\hat{U}_{\caI_1} = \hat{U}_{\caI_2} = 1$, the two collisions must occur in the same column, by disjointness of the events within any shared row. If $\ell = 0$, the two indicators are independent by independence of the rows. Consequently,
\begin{equation*}
    \bbe (\hat{U}_{\caI_1} \hat{U}_{\caI_2})
    = \begin{cases}
        (\bbe \hat{U}_{\caI_1}) (\bbe \hat{U}_{\caI_2})
        = n^2 \eta_n^{2 k}, & \quad \ell = 0, \\
        n \eta_n^{2 k - \ell}, & \quad 1 \leq \ell \leq k.
    \end{cases}
\end{equation*}
Using $n \eta_n \to 0$ from \eqref{eqn:tail-bn-lower} and $p_n \asymp n$, we obtain
\begin{align*}
    \Var (\hat{Z}_{k, n})
    & = \sum\nolimits_{\ell=0}^k 
    \sum\nolimits_{\abs{\caI_1} = \abs{\caI_2} = k,
    \, \abs{\caI_1 \cap \caI_2} = \ell}
    \braks[\big]{\bbe (\hat{U}_{\caI_1} \hat{U}_{\caI_2})
    - (\bbe \hat{U}_{\caI_1}) (\bbe \hat{U}_{\caI_2})} \\
    & \leq \sum\nolimits_{\ell = 1}^{k}
    \binom{p_n}{k} \binom{k}{\ell} \binom{p_n - k}{k - \ell}
    n \eta_n^{2 k - \ell}
    \leq \sum\nolimits_{\ell = 1}^{k} 
    n p_n^{2k-\ell} \eta_n^{2 k - \ell}
    \leq C n p_n^{k} \eta_n^{k}.
\end{align*}
Chebyshev's inequality and \eqref{eqn:expectation-Zhat} now give
\begin{equation*}
    \bbp \{\hat{Z}_{k, n} = 0\}
    \leq \bbp \curls[\big]{\abs{\hat{Z}_{k, n} - \bbe \hat{Z}_{k, n}} 
    \geq \bbe \hat{Z}_{k, n}}
    \leq \Var (\hat{Z}_{k, n}) / (\bbe \hat{Z}_{k, n})^2 
    \leq C / (\bbe \hat{Z}_{k, n}) \to 0.
\end{equation*}
Thus $\{\hat{Z}_{k, n} > 0\}$ holds with probability tending to one.

On this event, choose a set of rows $\caI \subset \dbraks{p_n}$ with $\abs{\caI} = k$ and a column $\mu \in \dbraks{n}$ such that $\Xi_{i \mu}$ holds for every $i \in \caI$. Since $\sqrt{n} b_n > \fku_n$ for all sufficiently large $n$, we have $\caI \subset \caB$ and $\mu \in \bigcap_{i \in \caI} \caA_i$. The resampling representation \eqref{eqn:unified-resampling}, the definitions \eqref{eqn:unified-frozen-denominator} and \eqref{eqn:unified-S-W-D}, and the row bound in \eqref{eqn:good-strong-star} yield
\begin{equation*}
    \hbar_i^2
    = n + \sum\nolimits_{\rho \in \caA_i} \abs{x_{i \rho}}^2
    \leq \abs{x_{i \mu}}^2 + 3 n,
    \qquad
    \abs{\fks_{i \mu}}^2
    \geq \frac{\abs{x_{i \mu}}^2}{\abs{x_{i \mu}}^2 + 3 n}
    \geq \frac{b_n^2}{b_n^2 + 3},    
    \qquad
    i \in \caI.
\end{equation*}
Write $\fkS_{\caB} = (\bfks_1, \ldots, \bfks_n)$ for the columns. Then the effective Gram matrix satisfies $\bfL_{\caB} \geq \bfks_\mu \bfks_\mu^\top$ by its definition in \eqref{eqn:unified-component-matrix}. Hence, with probability tending to one,
\begin{equation*}
    \lambda_1 (\bfL_{\caB})
    \geq \norm{\bfks_\mu}^2
    = \sum\nolimits_{i \in \caB} \abs{\fks_{i \mu}}^2
    \geq \sum\nolimits_{i \in \caI} \abs{\fks_{i \mu}}^2
    \geq \frac{k b_n^2}{b_n^2 + 3}
    = k - o (1),
\end{equation*}
where the last step uses $b_n \to \infty$. This proves the lower bound for $\bfL_\caB$ in \eqref{eqn:supercritical-local-bound}.

For the same rows $\caI$ and column $\mu$, the bound in \eqref{eqn:good-strong-star} also gives
\begin{equation*}
    \norm{\bfx_i}^2
    = \sum\nolimits_{\rho=1}^n \abs{x_{i \rho}}^2
    \leq \abs{x_{i \mu}}^2 + 2 n,
    \qquad
    \abs{y_{i \mu}}^2
    \geq \frac{\abs{x_{i \mu}}^2}{\abs{x_{i \mu}}^2 + 2 n}
    \geq \frac{b_n^2}{b_n^2 + 2},
    \qquad
    i \in \caI.
\end{equation*}
Write $\bfQ_n = \bfY_n^\top \bfY_n$ for the companion matrix. Then, with probability tending to one,
\begin{equation}
    \lambda_1 (\bfR_n)
    = \lambda_1 (\bfQ_n)
    \geq \angles{\bfe_\mu, \bfQ_n \bfe_\mu} 
    % & = \sum\nolimits_{i=1}^{p_n} \abs{y_{i \mu}}^2 
    \geq \sum\nolimits_{i \in \caI} \abs{y_{i \mu}}^2
    \geq \frac{k b_n^2}{b_n^2 + 2}
    = k - o (1).    
    \label{eqn:lower-Rn-k}    
\end{equation}
This proves the lower bound for $\bfR_n$ in \eqref{eqn:supercritical-local-bound}.
\end{proof}

% *******************************************
\begin{proof}[Proof of Theorem \ref{thm:collision-layer}]
% *******************************************
Fix $k \geq k_* (\phi)$ and construct $\bfL_{\caB}$ using the threshold $\fku_n = n^{1/2-\tau_{k + 1}}$, where $\tau_{k + 1}$ is fixed and satisfies \eqref{eqn:unified-truncation} at order $k + 1$. Under the subcritical condition in \eqref{cond:deterministic-layer}, Lemma \ref{lem:collision-cap}, applied at order $k + 1$, gives the upper bound $\lambda_1 (\bfL_{\caB}) \leq k + o_{\bbp} (1)$. The supercritical condition at order $k$ allows us to apply Lemma \ref{lem:strong-stars} to obtain the matching lower bound $\lambda_1 (\bfL_{\caB}) \geq k - o_{\bbp} (1)$. Thus $\lambda_1 (\bfL_{\caB}) = k + o_{\bbp} (1)$. Since $k \geq k_* (\phi) > 1 + \sqrt{\phi}$, Corollary \ref{cor:extremal-transfer} \ref{item:extremal-transfer-outlier}, applied at order $k + 1$, yields \eqref{eqn:deterministic-layer-limit}.
\end{proof}

Next, the divergence in Theorem \ref{thm:all-orders-divergence} follows immediately from the direct lower bound for $\lambda_1 (\bfR_n)$ in Lemma \ref{lem:strong-stars}, applied at arbitrarily large fixed orders.

% *******************************************
\begin{proof}[Proof of Theorem \ref{thm:all-orders-divergence}]
% *******************************************
Fix $M > 0$ and choose an integer $k > \max \{M, k_* (\phi)\}$. Under \eqref{cond:all-orders-supercritical}, Lemma \ref{lem:strong-stars}, applied at order $k$, gives $\lambda_1 (\bfR_n) \geq k - o_{\bbp} (1)$. Since $k > M$, it follows that $\bbp \{\lambda_1 (\bfR_n) > M\} \to 1$. As $M > 0$ was arbitrary, this proves \eqref{eqn:divergence-max-eig}.
\end{proof}

Under \eqref{cond:wide-tail-super}, the argument in the proof of Lemma \ref{lem:strong-stars} with $k = 2$ also yields $\lambda_{\abs{\caB}} (\bfL_{\caB}) = o_{\bbp} (1)$. Since $\theta_\phi (x) \to 0$ as $x \downarrow 0$, this suggests that $\lambda_{p_n} (\bfR_n) \to 0$ in probability. However, the spectral reduction in Proposition \ref{prop:spectral-transfer} applies only on compact subsets of $(0, \lambda_-)$, bounded away from zero. Here we give a direct proof, avoiding the need to extend Corollary \ref{cor:extremal-transfer} \ref{item:extremal-transfer-outlier} to the origin.

% *******************************************
\begin{proof}[Proof of Theorem \ref{thm:wide-supercritical}]
% *******************************************
Condition \eqref{cond:wide-tail-super} is precisely the supercritical assumption of Lemma \ref{lem:strong-stars} with $k = 2$. We retain the notation from its proof, which establishes $b_n \to \infty$ and $\bbp \{\hat{Z}_{2, n} > 0\} \to 1$. On the event $\{\hat{Z}_{2, n} > 0\}$, choose distinct rows $i, j \in \dbraks{p_n}$ and a column $\mu$ such that both $\Xi_{i \mu}$ and $\Xi_{j \mu}$ hold. Let $\bfR_{\caI}$ be the principal submatrix of $\bfR_n$ indexed by $\caI = \{i, j\}$. Applying the argument leading to \eqref{eqn:lower-Rn-k} to these two rows gives
\begin{equation*}
    \lambda_1 (\bfR_{\caI})
    \geq \abs{y_{i \mu}}^2 + \abs{y_{j \mu}}^2
    \geq \frac{2 b_n^2}{b_n^2 + 2}
    = 2 - o(1).
\end{equation*}
Since $\bfR_{\caI}$ has unit diagonal, $\operatorname{Tr} \bfR_{\caI} = 2$. Cauchy's interlacing inequality therefore yields
\begin{equation*}
    0 \leq \lambda_{p_n} (\bfR_n)
    \leq \lambda_2 (\bfR_{\caI})
    = 2 - \lambda_1 (\bfR_{\caI})
    = o(1)
\end{equation*}
on the same event. Since this event has probability tending to one, \eqref{eqn:wide-corr} follows.
\end{proof}

% *******************************************
\section{Critical laws}
\label{sec:critical-collisions}
% *******************************************

Throughout this section, we fix an integer $k \geq k_* (\phi)$ and assume \eqref{eqn:basic-assumption} and the critical tail condition \eqref{cond:critical-collision} at order $k$. For the lower edge results, we specialize to $k = 2$ and $0 < \phi < 1$, in which case \eqref{cond:critical-collision} reduces to \eqref{cond:wide-tail-critical}. This tail condition implies \eqref{eqn:analytic-tail-envelope}, and hence we construct the effective Gram matrix $\bfL_{\caB}$ as in Section \ref{sec:truncation}, using the threshold $\fku_n = n^{1/2 - \tau_k}$ for a fixed exponent $\tau_k$ satisfying \eqref{eqn:unified-truncation} at order $k$.

We prove the upper critical laws in Theorems \ref{thm:upper-poisson} and \ref{thm:critical-maximum} and their lower counterparts in Theorems \ref{thm:lower-poisson} and \ref{thm:critical-minimum} through a common analysis of the effective Gram matrix $\bfL_{\caB}$. Section \ref{subsec:collision-amplitude-process} establishes the Poisson limit of the amplitude process $\scA_{k, n}$ for stars with $k$ row leaves. Section \ref{subsec:collision-effective-process} combines the block diagonal structure of $\bfL_{\caB}$ with the map $\gamma_k^{+}$ to transfer this limit to the upper effective eigenvalue process $\scP_{k, n}^{+}$ above $\beta_\phi (k - 1) = (1 + \sqrt{\phi}) \vee (k - 1)$. Here $1 + \sqrt{\phi}$ is the BBP threshold. For $k \geq 3$, the term $k - 1$ reflects the lower bound in Lemma \ref{lem:strong-stars}, since criticality at order $k$ implies supercriticality at order $k - 1$. For the lower edge, the identity $\gamma_2^{-} = 2 - \gamma_2^{+}$ yields the process $\scP_{2, n}^{-}$ by reflection. Finally, Section \ref{subsec:critical-spectral-proofs} applies Proposition \ref{prop:spectral-transfer} and the outlier map $\theta_\phi$ to transfer these limits to the corresponding spectral processes $\scN_{k, n}^{+}$ and $\scN_{2, n}^{-}$ of $\bfR_n$. The limiting laws of the largest and smallest eigenvalues then follow from the probabilities that the corresponding limiting processes contain no points above or below a given level, respectively.

% *******************************************
\subsection{The critical amplitude process}
\label{subsec:collision-amplitude-process}
% *******************************************

Recall the collection of star components $\bbg_{k\star}$ from \eqref{def:collection-stars} and their associated amplitude vectors $\bfa (\caC)$ from \eqref{def:raw-star-amplitudes}. For stars with $k$ row leaves, define the amplitude process on $\bbr_+^k$ by
\begin{equation}
    \scA_{k, n}
    := \sum\nolimits_{\caC \in \bbg_{k\star}}
    \delta_{\bfa (\caC)},
    \label{def:critical-amplitude-process}
\end{equation}
Conditioning on $\caH_n$ fixes the graph $\caG$ and hence the collection of stars $\bbg_{k\star}$. The remaining randomness in $\scA_{k, n}$ comes from the large-entry values $\psi_{i \mu}$. Distinct stars in $\bbg_{k\star}$ have disjoint edge sets, and therefore their amplitude vectors $\bfa (\caC)$ are conditionally independent with common law $\rho_n^{\otimes k}$, where
\begin{equation}
    \rho_n := \caL (\abs{\psi_n} / \sqrt{n}).
    \label{def:conditional-amplitude-law}
\end{equation}
To determine the limit of $\scA_{k, n}$ under the critical tail condition \eqref{cond:critical-collision}, we first derive the asymptotics of the conditional amplitude law $\rho_n$ and the number of $k$-stars
\begin{equation}
    T_{k, n} := \abs{\bbg_{k\star}}
    = \scA_{k, n} (\bbr_+^k).
    \label{def:critical-star-count}
\end{equation}
Recall the tail measure $\varrho_{k, \omega}$ from \eqref{def:critical-tail-measure}. Under \eqref{cond:critical-collision}, the tail probability $\fkq_n$ in \eqref{eqn:unified-large-entry-probability} satisfies
\begin{equation}
    \fkq_n = \bar{F}_{\xi} (\fku_n)
    \sim \omega n^{-1 - 1 / k + \alpha_k \tau_k},
    \label{eqn:critical-truncation-probability}
\end{equation}
In particular, $p_n \fkq_n \to 0$ since $\alpha_k \tau_k < 1/k$.

% *******************************************
\begin{lemma}
\label{lem:critical-star-count}
% *******************************************
Under \eqref{cond:critical-collision}, the following limits hold.
\begin{enumerate}[label = (\roman*)]
    \item \label{item:critical-amplitude-vague} Let $\rho_n$ be defined as in \eqref{def:conditional-amplitude-law}. In the vague topologies on $\bbr_+$ and $\bbr_+^k$, respectively, 
    \begin{equation}
        n^{\alpha_k \tau_k} \rho_n
        \Rightarrow \omega^{-1} \varrho_{k, \omega},
        \qquad
        n^{k \alpha_k \tau_k} \rho_n^{\otimes k}
        \Rightarrow \omega^{-k} \varrho_{k, \omega}^{\otimes k}.
        \label{eqn:conditional-amplitude-vague}
    \end{equation}
    \item \label{item:critical-star-count} The count $T_{k, n}$ of $k$-stars, defined in \eqref{def:critical-star-count}, satisfies
    \begin{equation}
        n^{-k \alpha_k \tau_k} T_{k, n}
        \pconv {\phi^k \omega^k} / {k!}.
        \label{eqn:critical-star-count-limit}
    \end{equation}
\end{enumerate}
\end{lemma}

% *******************************************
\begin{proof}[Proof of Lemma \ref{lem:critical-star-count}]
% *******************************************
For part \ref{item:critical-amplitude-vague}, fix $0 < u < v < \infty$. The conditional law large-entry law \eqref{eqn:alpha3-resampled-psi} and \eqref{eqn:critical-truncation-probability} give
\begin{equation*}
    n^{\alpha_k \tau_k} \rho_n ((u, v])
    = \frac{n^{\alpha_k \tau_k}}{\fkq_n}
    \braks[\big]{\bar{F}_{\xi} (\sqrt{n} u) - \bar{F}_{\xi} (\sqrt{n} v)}
    \to u^{-\alpha_k} - v^{-\alpha_k},
\end{equation*}
where the first identity holds for all sufficiently large $n$. By the definition of $\varrho_{k, \omega}$ in \eqref{def:critical-tail-measure}, this proves the first limit in \eqref{eqn:conditional-amplitude-vague}. Since $k$ is fixed, taking products gives the second limit.

For part \ref{item:critical-star-count}, recall the $\caH_n$-measurable counting variable $Z_{k, n}$ from \eqref{def:Zk-star-count}. On $\tilde{\Omega}_{k, n}$, the tuples counted by $Z_{k, n}$ correspond exactly to the stars in $\bbg_{k\star}$, and hence $T_{k, n} = Z_{k, n}$. Since the indicators $\chi_{i \mu}$ are i.i.d. with law $\mathrm{Bern} (\fkq_n)$, we can use \eqref{eqn:critical-truncation-probability} and $p_n \sim \phi n$ to obtain
\begin{equation}
    \bbe Z_{k, n}
    = n \binom{p_n}{k} \fkq_n^k
    \sim \frac{\phi^k \omega^k}{k!} n^{k \alpha_k \tau_k}
    \to \infty.
    \label{eqn:expectation-count-Zk}
\end{equation}
To compute the variance of $Z_{k, n}$, write
\begin{equation*}
    Z_{k, n} = \sum\nolimits_{\abs{\caI} = k} 
    \sum\nolimits_{\mu=1}^n U_{\caI, \mu},
    \qquad
    U_{\caI, \mu} := \prod\nolimits_{i \in \caI} \chi_{i \mu},
\end{equation*}
where the summation over $\caI$ runs over subsets of $\dbraks{p_n}$. Two summands $U_{\caI_1, \mu_1}$ and $U_{\caI_2, \mu_2}$ are independent unless $\mu_1=\mu_2$ and $\caI_1 \cap \caI_2 \neq \varnothing$. Hence, for summands in the same column with $\abs{\caI_1 \cap \caI_2} = \ell \geq 1$, we have $\bbe (U_{\caI_1, \mu} U_{\caI_2, \mu}) = \fkq_n^{2k - \ell}$. Consequently, using $p_n \fkq_n \to 0$,
\begin{equation*}
    \Var (Z_{k, n})
    \leq n \sum\nolimits_{\ell = 1}^k
    \binom{p_n}{k} \binom{k}{\ell} \binom{p_n - k}{k - \ell}
    \fkq_n^{2k - \ell}
    \leq n \sum\nolimits_{\ell = 1}^k
    p_n^{2k - \ell} \fkq_n^{2k - \ell}
    \leq C n p_n^k \fkq_n^k.    
\end{equation*}
In particular, $\Var (Z_{k, n}) = O (\bbe Z_{k, n})$. Chebyshev's inequality and $\bbe Z_{k, n} \to \infty$ from \eqref{eqn:expectation-count-Zk} therefore lead to $Z_{k, n} / \bbe Z_{k, n} = 1 + o_{\bbp} (1)$. Combining this with \eqref{eqn:expectation-count-Zk} yields $n^{-k \alpha_k \tau_k} Z_{k, n} = {\phi^k \omega^k} / {k!} + o_{\bbp} (1)$. Since $T_{k, n} = Z_{k, n}$ on $\tilde{\Omega}_{k, n}$ and $\bbp (\tilde{\Omega}_{k, n}) \to 1$ by Lemma \ref{lem:unified-graph} \ref{item:unified-graph-components}, the same equation holds for $T_{k, n}$, proving \eqref{eqn:critical-star-count-limit}.
\end{proof}

Given an open set $A$, we write $C_c (A)$ for the space of continuous functions with compact support in $A$. For a locally finite measure $\Gamma$ on $A$ and $f \in C_c (A)$, write $\angles{f, \Gamma} \equiv \int_A f \, \rmd \Gamma$.

% *******************************************
\begin{proposition}[Star amplitude process]
\label{prop:collision-poisson}
% *******************************************
Under \eqref{cond:critical-collision},
\begin{equation}
    \scA_{k, n} \Rightarrow \scA_k,
    \qquad 
    \scA_k \sim \PRM (\Pi_{k, \phi, \omega}),
    \qquad 
    \Pi_{k, \phi, \omega}
    := \frac{\phi^k}{k!} \varrho_{k, \omega}^{\otimes k},
    \label{eqn:unified-amplitude-poisson}
\end{equation}
in the vague topology on $\bbr_+^k$. Equivalently, for every nonnegative $f \in C_c (\bbr_+^k)$,
\begin{equation}
    \bbe \braks*{\exp \{-\angles{f, \scA_{k, n}} \}}
    \to
    \exp \curls[\bigg]{-\int_{\bbr_+^k}
    \braks[\big]{1 - \exp \{-f (\bfa)\}}
    \, \Pi_{k, \phi, \omega} (\rmd \bfa)},
    \label{eqn:conditional-amplitude-laplace-limit}
\end{equation}
\end{proposition}

% *******************************************
\begin{proof}[Proof of Proposition \ref{prop:collision-poisson}]
% *******************************************
Fix a nonnegative $f \in C_c (\bbr_+^k)$ and put
\begin{equation*}
    \fkJ_{k, n} (f) := \int_{\bbr_+^k}
    \braks[\big]{1 - \exp \{-f (\bfa)\} }
    \, \rho_n^{\otimes k} (\rmd \bfa).
\end{equation*}
Conditional on $\caH_n$, the count $T_{k, n} = \abs{\bbg_{k\star}}$ is fixed, and the amplitude vectors $\{\bfa (\caC)\}_{\caC \in \bbg_{k\star}}$ are independent with common law $\rho_n^{\otimes k}$. Consequently,
\begin{equation}
    \bbe^{\caH} \braks[\big]{\exp \curls{-\angles{f, \scA_{k, n}}}}
    = (1 - \fkJ_{k, n} (f))^{T_{k, n}}.
    \label{eqn:conditional-amplitude-laplace}
\end{equation}
Since $\bfa \mapsto 1 - \exp \{-f (\bfa)\}$ is continuous and compactly supported, the vague convergence \eqref{eqn:conditional-amplitude-vague} implies that
\begin{equation*}
    n^{k \alpha_k \tau_k} \fkJ_{k, n} (f)
    \to \omega^{-k} \int_{\bbr_+^k}
    \braks[\big]{1 - \exp \{-f (\bfa)\}}
    \, \varrho_{k, \omega}^{\otimes k} (\rmd \bfa).
\end{equation*}
Combining this with \eqref{eqn:critical-star-count-limit} and the definition of $\Pi_{k, \phi, \omega}$ in \eqref{eqn:unified-amplitude-poisson}, we obtain
\begin{equation*}
    T_{k, n} \fkJ_{k, n} (f)
    \pconv \int_{\bbr_+^k} \braks[\big]{1 - \exp \{-f (\bfa)\}}
    \, \Pi_{k, \phi, \omega} (\rmd \bfa),
    \qquad
    T_{k, n} \fkJ_{k, n} (f)^2 \pconv 0.
\end{equation*}
Applying the expansion $\log (1 - u) = -u + O (u^2)$ to \eqref{eqn:conditional-amplitude-laplace} gives
\begin{equation*}
    \bbe^{\caH} \braks[\big]{\exp \curls{-\angles{f, \scA_{k, n}}}}
    \pconv
    \exp \curls[\bigg]{-\int_{\bbr_+^k}
    \braks[\big]{1 - \exp \{-f (\bfa)\}}
    \, \Pi_{k, \phi, \omega} (\rmd \bfa)}.
\end{equation*}
Here the conditional Laplace functionals lie in $[0, 1]$, and thus taking expectations yields \eqref{eqn:conditional-amplitude-laplace-limit}. The right-hand side is the Laplace functional of $\PRM (\Pi_{k, \phi, \omega})$, which proves \eqref{eqn:unified-amplitude-poisson}.
\end{proof}

To pass from the amplitude process $\scA_{k, n}$ to the eigenvalue process of $\bfL_{\caB}$, we need to study the intensity measure $\Pi_{k, \phi, \omega}$ under the map $\gamma_k^{+}$. Although $\Pi_{k, \phi, \omega}$ has infinite mass near the origin, the next lemma shows that its restriction to $\{\gamma_k^{+} > \beta_\phi (k - 1)\}$ has finite positive mass.

% *******************************************
\begin{lemma}
\label{lem:critical-intensity}
% *******************************************
The amplitude intensity $\Pi_{k, \phi, \omega}$ in \eqref{eqn:unified-amplitude-poisson} has the following properties.
\begin{enumerate}[label = (\roman*)]
    \item \label{item:critical-intensity-finite} The set $\{\gamma_k^{+} > \beta_\phi (k - 1)\}$ has finite positive $\Pi_{k, \phi, \omega}$-measure,
    \begin{equation}
        \Pi_{k, \phi, \omega} \pars[\big]{
        \{\bfa \in \bbr_+^k: 
        \gamma_k^{+} (\bfa) > \beta_\phi (k - 1)\} }
        \in (0, \infty).
        \label{eqn:finite-critical-intensity}
    \end{equation}
    \item \label{item:critical-intensity-level-sets} Every level set $\{\bfa \in \bbr_+^k: 
        \gamma_k^{+} (\bfa) = t\}$ with $t > 1$ has zero $\Pi_{k, \phi, \omega}$-measure.
\end{enumerate}
\end{lemma}

% *******************************************
\begin{proof}[Proof of Lemma \ref{lem:critical-intensity}]
% *******************************************
We first prove part \ref{item:critical-intensity-finite}. For every $k \geq 2$, positivity follows by considering the set where $\min_{r \in \dbraks{k}} a_r$ is sufficiently large. Indeed, Lemma \ref{lem:star-spectrum} \ref{item:star-spectrum-bounds} implies $\gamma_k^{+} (\bfa) \to k > \beta_\phi (k - 1)$ as $\min_{r \in \dbraks{k}} a_r \to \infty$. 

To prove finiteness, first suppose $k = 2$. Since $k \geq k_* (\phi)$, we have $\phi \in (0,1)$ and hence $\beta_\phi (k - 1) = 1 + \sqrt{\phi}$. A direct computation gives the explicit formula $\gamma_2^{+} (\bfa) = 1 + \varphi (a_1) \varphi (a_2)$. The definition of $\varphi$ in \eqref{def:star-amplitude-normalization} therefore yields 
\begin{equation*}
    \gamma_2^{+} (\bfa) > \beta_\phi (k - 1)
    \quad \Longrightarrow \quad
    \varphi (a_1) \varphi (a_2) > \sqrt{\phi}
    \quad \Longrightarrow \quad
    a_1 \wedge a_2 > \sqrt{\phi}.
\end{equation*}
Hence, $\{\gamma_2^{+} (\bfa) > \beta_\phi (1)\}$ has finite $\Pi_{2, \phi, \omega}$-measure, proving finiteness in \eqref{eqn:finite-critical-intensity} for $k = 2$.

For $k \geq 3$, it suffices to control the $\Pi_{k, \phi, \omega}$-measure of the larger set $\{\gamma_k^{+} (\bfa) > k-1\}$. Set
\begin{equation*}
    s_r = \varphi (a_r) = \frac{a_r}{\sqrt{1 + a_r^2}},
    \qquad
    h (s) = \frac{s^2}{k - 2 + s^2}.
\end{equation*}
Each summand on the left-hand side of \eqref{eqn:unified-star-secular} is decreasing in $x \in (1, \infty)$. Lemma \ref{lem:star-spectrum} \ref{item:star-spectrum-secular} therefore gives
\begin{equation*}
    \gamma_k^{+} (\bfa) > k - 1
    \quad \Longrightarrow \quad
    \sum\nolimits_{r = 1}^k h (s_r) 
    = \sum\nolimits_{r = 1}^k 
    \frac{a_r^2}{k - 2 + (k-1) a_r^2}
    > 1.
\end{equation*}
By symmetry, we may assume that $a_k = \min_{r \in \dbraks{k}} a_r$, so that $s_k = \min_{r \in \dbraks{k}} s_r$. Since $(k - 1) h (1) = 1$, the preceding inequality shows that $\gamma_k^{+} (\bfa) > k - 1$ further implies
\begin{equation*}
    \frac{k - 2}{(k - 1)^2}
    \sum\nolimits_{r=1}^{k-1} (1 - s_r^2)
    \leq \sum\nolimits_{r=1}^{k-1} [h (1) - h (s_r)] 
    < h (s_k)
    \leq \frac{s_k^2}{k - 2}.
\end{equation*}
Hence $1 - s_r^2 \leq 4 s_k^2$ for every $r < k$. If $a_k \leq 1/2$, then $s_r = \varphi (a_r)$ gives $a_r^2 \geq (4 a_k)^{-2}$ for every $r < k$. Consequently, $\{\gamma_k^{+} (\bfa) > \beta_\phi (k - 1)\} \subset \bba \cup [1/2, \infty)^k$, where
\begin{align*}
    \bba := \bigcup\nolimits_{i = 1}^k
    \curls[\big]{a_i \in (0, 1/2),
    \, a_r \geq (4 a_i)^{-1},
    \forall \, r \neq i }.
\end{align*}
The definition of $\Pi_{k, \phi, \omega}$ gives $\Pi_{k, \phi, \omega} ([1/2, \infty)^k) < \infty$. It remains to bound $\Pi_{k, \phi, \omega} (\bba)$. By \eqref{def:critical-tail-measure}, we have $\varrho_{k, \omega} ([t, \infty)) = \omega t^{-\alpha_k}$ for every $t > 0$. The union bound and the product form of $\Pi_{k, \phi, \omega}$ therefore yield
\begin{equation*}
    \Pi_{k, \phi, \omega} (\bba)
    \leq \frac{k \phi^k}{k!}
    \int_0^{1/2}
    \curls[\big]{\varrho_{k, \omega} ([(4t)^{-1}, \infty))}^{k - 1}
    \, \varrho_{k, \omega} (\rmd t) 
    \leq C \int_0^{1/2} t^{\alpha_k (k - 2) - 1} \, \rmd t 
    < \infty,    
\end{equation*}
where the last step uses $k \geq 3$ and $\alpha_k > 0$. This proves finiteness in \eqref{eqn:finite-critical-intensity} for $k \geq 3$.

For part \ref{item:critical-intensity-level-sets}, fix $t > 1$ and $(a_1, \ldots, a_{k-1}) \in \bbr_+^{k-1}$. By Lemma \ref{lem:star-spectrum} \ref{item:star-spectrum-monotonicity}, the map $a_k \mapsto \gamma_k^{+} (\bfa)$ is strictly increasing on $(0, \infty)$. Thus $\{a_k > 0: \gamma_k^{+} (\bfa) = t\}$ contains at most one point. Since $\varrho_{k, \omega}$ has a density, this level set has zero $\varrho_{k, \omega}$-measure. Tonelli's theorem and the product form of $\Pi_{k, \phi, \omega}$ then imply that the level set $\{\bfa \in \bbr_+^k: \gamma_k^{+} (\bfa) = t\}$ has zero $\Pi_{k, \phi, \omega}$-measure.
\end{proof}

% *******************************************
\subsection{The effective eigenvalue process}
\label{subsec:collision-effective-process}
% *******************************************

We now combine the block diagonal representation \eqref{eqn:unified-component-cap} with the limit of the star amplitude process $\scA_{k, n}$ to describe the eigenvalues of $\bfL_{\caB}$ above $\beta_\phi (k - 1)$. 

Define the upper effective eigenvalue process on $(\beta_\phi (k - 1), \infty)$ by
\begin{subequations}
\begin{equation}
    \scP_{k, n}^{+} := \sum\nolimits_{i = 1}^{\abs{\caB}}
    \bbone \{\lambda_i (\bfL_{\caB}) > \beta_\phi (k - 1)\}
    \, \delta_{\lambda_i (\bfL_{\caB})}.
    \label{def:point-process-scP}
\end{equation}
For $\phi \in (0,1)$ and $k = 2$, define the lower effective eigenvalue process on $(0, 1 - \sqrt{\phi})$ by
\begin{equation}
    \scP_{2, n}^{-} := \sum\nolimits_{i = 1}^{\abs{\caB}}
    \bbone \{0 < \lambda_i (\bfL_{\caB}) < 1 - \sqrt{\phi}\}
    \, \delta_{\lambda_i (\bfL_{\caB})}.
    \label{def:point-process-scP-lower}
\end{equation}    
\end{subequations}
We first establish the Poisson limit of $\scP_{k, n}^{+}$ in the next proposition. For the lower process $\scP_{2, n}^{-}$, we associate it with $\scP_{2, n}^{+}$ and thus its convergence could be obtained as a corollary.

% *******************************************
\begin{proposition}[Upper effective eigenvalue process]
\label{prop:collision-spectrum}
% *******************************************
Assume \eqref{eqn:basic-assumption} and suppose that \eqref{cond:critical-collision} holds for a fixed integer $k \geq k_* (\phi)$. In the vague topology on $(\beta_\phi (k - 1), \infty)$,
\begin{equation}
    \scP_{k, n}^{+} \Rightarrow \scP_k^{+},
    \qquad
    \scP_k^{+} \sim \PRM (\Gamma_{k, \phi, \omega}^{+}),
    \label{eqn:critical-effective-process}
\end{equation}
where, for Borel sets $B \subset (\beta_\phi (k - 1), \infty)$,
\begin{equation}
    \Gamma_{k, \phi, \omega}^{+} (B)
    := \Pi_{k, \phi, \omega} \circ (\gamma_k^{+})^{-1} (B)
    = \frac{\phi^k}{k!} \int_{\bbr_+^k}
    \bbone \{\gamma_k^{+} (\bfa) \in B\} \,
    \varrho_{k, \omega}^{\otimes k} (\rmd \bfa).
    \label{def:critical-effective-intensity}
\end{equation}
\end{proposition}

% *******************************************
\begin{proof}[Proof of Proposition \ref{prop:collision-spectrum}]
% *******************************************
Recall \eqref{eqn:unified-component-cap}. On $\tilde{\Omega}_{k, n}$, Lemma \ref{lem:collision-components} shows that a block $\bfL_{\caI (\caC)}$ can have an eigenvalue above $\beta_\phi (k - 1) \geq k - 1$ only if $\caC \in \bbg_{k\star}$. For such a block, \eqref{eqn:star-sign-conjugation} and Lemma \ref{lem:star-spectrum} \ref{item:star-spectrum-bounds} identify its largest eigenvalue as $\gamma_k^{+} (\bfa (\caC))$ and bound all remaining eigenvalues by one. Thus, on $\tilde{\Omega}_{k, n}$, we have
\begin{equation}
    \scP_{k, n}^{+}
    = \sum\nolimits_{\caC \in \bbg_{k\star}}
    \bbone \{\gamma_k^{+} (\bfa (\caC)) > \beta_\phi (k - 1)\}
    \, \delta_{\gamma_k^{+} (\bfa (\caC))}
    = \sum\nolimits_{\bfa \in \scA_{k, n}}
    \bbone \{\gamma_k^{+} (\bfa) > \beta_\phi (k - 1)\}
    \, \delta_{\gamma_k^{+} (\bfa)},
    \label{eqn:critical-effective-mark-representation}
\end{equation}
where the sum over $\bfa \in \scA_{k, n}$ runs over the atoms of $\scA_{k, n}$, counted with multiplicity.

Fix a nonnegative function $h \in C_c ((\beta_\phi (k - 1), \infty))$ and extend it by zero outside $(\beta_\phi (k - 1), \infty)$. Define
\begin{equation*}
    g_h (\bfa)
    := 1 - \exp \{-h (\gamma_k^{+} (\bfa))\},
    \qquad
    \fkJ_{k, n} (h) := \int_{\bbr_+^k} g_h (\bfa) \,
    \rmd \rho_n^{\otimes k} (\bfa).
\end{equation*}
Since $h$ vanishes outside $(\beta_\phi (k - 1), \infty)$, the representation \eqref{eqn:critical-effective-mark-representation} implies that $\angles{h, \scP_{k, n}^{+}} = \sum_{\caC \in \bbg_{k\star}} h ({\gamma_k^{+} (\bfa (\caC))})$ on $\tilde{\Omega}_{k, n}$. On the other hand, conditional on any realization of $\caH_n$ in this event, the amplitude vectors $\{\bfa (\caC)\}_{\caC \in \bbg_{k\star}}$ are independent with common law $\rho_n^{\otimes k}$. Recalling that $T_{k, n} = \abs{\bbg_{k\star}}$, we therefore obtain
\begin{equation}
    \bbe^{\caH} \braks[\big]{\exp \curls{-\angles{h, \scP_{k, n}^{+}}}}
    = (1 - \fkJ_{k, n} (h))^{T_{k, n}}.
    \label{eqn:conditional-effective-laplace}
\end{equation}

We next determine the asymptotics of $\fkJ_{k, n} (h)$. The function $g_h$ is continuous and takes values in $[0, 1]$. Since the support of $h$ is a compact subset of $(\beta_\phi (k - 1), \infty) \subset (k - 1, \infty)$, Lemma \ref{lem:star-spectrum} \ref{item:star-spectrum-amplitudes} gives a constant $c > 0$ such that $g_h (\bfa) = 0$ whenever $\min_{r \in \dbraks{k}} a_r \leq c$. For $L \geq 1 \vee c$, the product convergence \eqref{eqn:conditional-amplitude-vague} gives
\begin{equation}
    n^{k \alpha_k \tau_k} \int_{(c, L]^k} 
    g_h (\bfa) \, \rmd \rho_n^{\otimes k} (\bfa)
    \to \omega^{-k} \int_{(c, L]^k} 
    g_h (\bfa) \, \rmd \varrho_{k, \omega}^{\otimes k} (\bfa).
    \label{eqn:convg-Jk-with-boundaries}
\end{equation}
It remains to control the integral over
\begin{equation*}
    \bba_L := \{\bfa \in \bbr_+^k: 
    \min\nolimits_{r \in \dbraks{k}} a_r > c, \ 
    \max\nolimits_{r \in \dbraks{k}} a_r > L \}.
\end{equation*}
For $L \geq 1 \vee c$, the critical tail condition \eqref{cond:critical-collision}, the asymptotic formula \eqref{eqn:critical-truncation-probability} for $\fkq_n$, and the conditional large-entry law \eqref{eqn:alpha3-resampled-psi} give, uniformly for all sufficiently large $n$,
\begin{equation*}
    n^{\alpha_k \tau_k} \rho_n ((c, \infty)) 
    = \frac{n^{\alpha_k \tau_k}}{\fkq_n}
    \bar{F}_{\xi} (c \sqrt{n})
    \leq C,
    \qquad
    n^{\alpha_k \tau_k} \rho_n ((L, \infty))
    = \frac{n^{\alpha_k \tau_k}}{\fkq_n}
    \bar{F}_{\xi} (L \sqrt{n})
    \leq C L^{-\alpha_k}.    
\end{equation*}
Similarly, \eqref{def:critical-tail-measure} gives $\varrho_{k, \omega} ((c, \infty)) \leq C$ and $\varrho_{k, \omega} ((L, \infty)) \leq C L^{-\alpha_k}$. Consequently,
\begin{equation*}
    n^{k \alpha_k \tau_k} \rho_n^{\otimes k} (\bba_L)
    \leq C L^{-\alpha_k},
    \qquad
    \varrho_{k, \omega}^{\otimes k} (\bba_L)
    \leq C L^{-\alpha_k}.
\end{equation*}
Combining these tail bounds with \eqref{eqn:convg-Jk-with-boundaries} and letting $L \to \infty$, we obtain
\begin{equation*}
    n^{k \alpha_k \tau_k} \fkJ_{k, n} (h)
    \to \omega^{-k}
    \int_{(c, \infty)^k}  g_h (\bfa) \, 
    \rmd \varrho_{k, \omega}^{\otimes k} (\bfa)
    = \omega^{-k}
    \int_{\bbr_+^k}  g_h (\bfa) \, 
    \rmd \varrho_{k, \omega}^{\otimes k} (\bfa).
\end{equation*}
We combine this limit with the asymptotic of $T_{k, n}$ in Lemma \ref{lem:critical-star-count} \ref{item:critical-star-count}. Taking logarithms in \eqref{eqn:conditional-effective-laplace} and expanding $\log (1 - u) = -u + O (u^2)$, as in the proof of Proposition \ref{prop:collision-poisson}, we obtain
\begin{equation}
    \bbe^{\caH} \braks[\big]{\exp \curls{-\angles{h, \scP_{k, n}^{+}}}}
    \pconv \exp \curls[\bigg]{-\int_{\beta_\phi (k - 1)}^{\infty}
    \braks[\big]{1 - \exp \{-h (t)\}}
    \, \Gamma_{k, \phi, \omega}^{+} (\rmd t)}.
    \label{eqn:conditional-effective-laplace-limit}
\end{equation}
Here we use the definition of $\Gamma_{k, \phi, \omega}^{+}$ in \eqref{def:critical-effective-intensity} and the fact that $\tilde{\Omega}_{k, n}$ is $\caH_n$-measurable with $\bbp (\tilde{\Omega}_{k, n}) \to 1$. Since the conditional Laplace functionals lie in $[0, 1]$, convergence in probability in \eqref{eqn:conditional-effective-laplace-limit} also gives convergence of their expectations. The limit is the Laplace functional of $\PRM (\Gamma_{k, \phi, \omega}^{+})$, which proves \eqref{eqn:critical-effective-process}.
\end{proof}

% *******************************************
\begin{corollary}[Lower effective eigenvalue process]
\label{cor:critical-effective-lower-counts}
% *******************************************
Assume \eqref{eqn:basic-assumption} with $0 < \phi < 1$ and \eqref{cond:wide-tail-critical}, and construct $\bfL_{\caB}$ at order $k = 2$. In the vague topology on $(0, 1 - \sqrt{\phi})$,
\begin{equation}
    \scP_{2, n}^{-} \Rightarrow \scP_2^{-},
    \qquad
    \scP_2^{-} \sim \PRM (\Gamma_{2, \phi, \omega}^{-}),
    \label{eqn:critical-lower-effective-process}
\end{equation}
where, for Borel sets $B \subset (0, 1 - \sqrt{\phi})$,
\begin{equation}
    \Gamma_{2, \phi, \omega}^{-} (B)
    := \Pi_{2, \phi, \omega} \circ (\gamma_2^{-})^{-1} (B)
    = \frac{\phi^2}{2} \int_{\bbr_+^2}
    \bbone \{\gamma_2^{-} (\bfa) \in B\} \,
    \varrho_{2, \omega}^{\otimes 2} (\rmd \bfa).
    \label{def:critical-lower-effective-intensity}
\end{equation}
\end{corollary}

% *******************************************
\begin{proof}[Proof of Corollary \ref{cor:critical-effective-lower-counts}]
% *******************************************
Note that the assumptions of this corollary specialize those of Proposition \ref{prop:collision-spectrum} to $k = 2$ and $\phi \in (0, 1)$. On $\tilde{\Omega}_{2, n}$, the effective Gram matrix $\bfL_{\caB}$ is block diagonal, with scalar blocks equal to $1$ and $2 \times 2$ blocks whose eigenvalues are $1 \pm \abs{L_{ij}}$, where $\abs{L_{ij}} \in (0, 1)$; see \eqref{eqn:2by2-block}. Consequently, on this event,
\begin{equation}
    \scP_{2, n}^{-}
    = \sum\nolimits_{t \in \scP_{2, n}^{+}} \delta_{2 - t}.
    \label{eqn:reflection-process}
\end{equation}
Since $\gamma_2^{+} (a_1, a_2) = 1 + \varphi (a_1) \varphi (a_2) \in (1, 2)$, all atoms of $\scP_2^{+}$ lie in $(1 + \sqrt{\phi}, 2)$ almost surely by \eqref{def:critical-effective-intensity}. The identity $\gamma_2^{-} (\bfa) = 2 - \gamma_2^{+} (\bfa)$ and the mapping theorem for Poisson processes (see, e.g. \cite[Proposition 3.7]{resnickExtremeValuesRegular1987}) allow us to realize the process $\scP_2^{-}$ as
\begin{equation}
    \scP_2^{-}
    = \sum\nolimits_{t \in \scP_2^{+}} \delta_{2 - t}.
    \label{eqn:reflection-limit}
\end{equation}
Therefore, Proposition \ref{prop:collision-spectrum} and $\bbp (\tilde{\Omega}_{2, n}) \to 1$ give $\scP_{2, n}^{-} \Rightarrow \scP_2^{-}$.
\end{proof}

% *******************************************
\subsection{Critical spectral counts}
\label{subsec:critical-spectral-proofs}
% *******************************************

We now transfer from $\scP_{k, n}^{+}$ to $\scN_{k, n}^{+}$ using Proposition \ref{prop:spectral-transfer}.

% *******************************************
\begin{proof}[Proof of Theorem \ref{thm:upper-poisson}]
% *******************************************
Fix $\ell \geq 1$ and $z_1 > \cdots > z_\ell > \vartheta_\phi (k - 1)$. Set 
\begin{equation*}
    x_i = \fkx_\phi (z_i) > \beta_\phi (k - 1).
\end{equation*}
By Proposition \ref{prop:spectral-transfer}, there is a deterministic sequence $\eps_n \downarrow 0$ such that, with probability tending to one, the following holds simultaneously for all $1 \leq i \leq \ell$,
\begin{equation}
    \fkn_+^{\bfL_{\caB}} (x_i + \eps_n)
    \leq \fkn_+^{\bfR_n} (z_i)
    \leq \fkn_+^{\bfL_{\caB}} (x_i - \eps_n).
    \label{eqn:sandwich-eigs-count}
\end{equation}
Choose $0 < \delta < (x_\ell - \beta_\phi (k - 1)) / 2$. Each interval $[x_i - \delta, x_i + \delta]$ is relatively compact in $(\beta_\phi (k - 1), \infty)$, and its endpoints have zero $\Gamma_{k, \phi, \omega}^{+}$-measure by Lemma \ref{lem:critical-intensity} \ref{item:critical-intensity-level-sets}. For all sufficiently large $n$, we have $\eps_n < \delta$. On the event \eqref{eqn:sandwich-eigs-count}, the definition of $\scP_{k, n}^{+}$ in \eqref{def:point-process-scP}, together with the Poisson convergence \eqref{eqn:critical-effective-process}, gives 
\begin{equation*}
    \abs[\big]{\fkn_+^{\bfR_n} (z_i) - \fkn_+^{\bfL_{\caB}} (x_i)}
    \leq \scP_{k, n}^{+} ([x_i - \delta, x_i + \delta])
    \Rightarrow
    \mathrm{Poisson} (\kappa_i (\delta)),
\end{equation*}
where the Poisson parameter satisfies
\begin{equation*}
    \kappa_i (\delta)
    = \Gamma_{k, \phi, \omega}^{+} ([x_i - \delta, x_i + \delta])
    \downarrow \Gamma_{k, \phi, \omega}^{+} (\{x_i\})
    = 0,
    \qquad
    \delta \downarrow 0.
\end{equation*}
Since the counts are integer-valued, the union bound and the Poisson convergence above give
\begin{align*}
    & ~ \limsup\nolimits_{n \to \infty}
    \bbp \curls[\big]{\max\nolimits_{1 \leq i \leq \ell}
    \abs[\big]{\fkn_+^{\bfR_n} (z_i) - \fkn_+^{\bfL_{\caB}} (x_i)} > 0} \\
    \leq & ~ \limsup\nolimits_{n \to \infty} \sum\nolimits_{i = 1}^\ell 
    \bbp \curls[\big]{ \scP_{k, n}^{+} ([x_i - \delta, x_i + \delta]) > 0 } 
    = \sum\nolimits_{i = 1}^\ell
    \braks{1 - \exp \curls{ -\kappa_i (\delta) }}.
\end{align*}
Here we have also used that the bounds \eqref{eqn:sandwich-eigs-count} hold with probability tending to one. Since $\ell$ is fixed, letting $\delta \downarrow 0$ yields the asymptotic equivalence of the counts
\begin{equation}
    \pars[\big]{ \fkn_+^{\bfR_n} (z_i)
    - \fkn_+^{\bfL_{\caB}} (x_i) }_{i = 1}^\ell \pconv \mathbf{0}.
    \label{eqn:proof-critical-count-equivalence}
\end{equation}

On $\tilde{\Omega}_{k, n}$, the upper bound in \eqref{eqn:basic-prop-diag-block} gives $\fkn_+^{\bfL_{\caB}} (x_i) = \scP_{k, n}^{+} ((x_i, k])$. By \eqref{def:critical-effective-intensity} and Lemma \ref{lem:star-spectrum} \ref{item:star-spectrum-monotonicity}, the limiting process $\scP_k^{+}$ has no points in $[k, \infty)$ almost surely. Each nonempty interval $(x_i, k]$ is relatively compact in $(\beta_\phi (k - 1), \infty)$. Lemma \ref{lem:critical-intensity} \ref{item:critical-intensity-level-sets} shows that its boundary has zero $\Gamma_{k, \phi, \omega}^{+}$-measure, and hence it is almost surely a continuity set of $\scP_k^{+}$. As $\bbp (\tilde{\Omega}_{k, n}) \to 1$, Proposition \ref{prop:collision-spectrum}, the continuous mapping theorem, and \eqref{eqn:proof-critical-count-equivalence} yield
\begin{equation}
    \pars[\big]{\fkn_+^{\bfR_n} (z_i)}_{i = 1}^\ell
    \Rightarrow \pars[\big]{\scP_k^{+} ((x_i, \infty))}_{i = 1}^\ell
    \overset{\mathrm{d}}{=} \pars[\big]{\scN_k^{+} ((z_i, \infty))}_{i = 1}^\ell,
    \label{eqn:convg-upper-counts-Rn}
\end{equation}
where the last equality is in distribution. Indeed, \eqref{def:critical-effective-intensity} and \eqref{def:critical-spectral-intensity} give
\begin{equation}
    \Lambda_{k, \phi, \omega}^{+} 
    = \Gamma_{k, \phi, \omega}^{+} \circ (\theta_\phi)^{-1}.
    \label{eqn:intensity-pushforward}
\end{equation}
The mapping theorem for Poisson processes therefore gives $\scN_k^{+} = \scP_k^{+} \circ (\theta_\phi)^{-1}$ in distribution. Taking $\ell = 1$ in \eqref{eqn:convg-upper-counts-Rn} proves \eqref{eqn:fixed-gap-poisson}. Since $\theta_\phi$ is strictly increasing, $\Lambda_{k, \phi, \omega}^{+} (\{z\}) = \Gamma_{k, \phi, \omega}^{+} (\{\fkx_\phi (z)\}) = 0$ for every $z > \vartheta_\phi (k - 1)$, proving the atomlessness of $\Lambda_{k, \phi, \omega}^{+}$. The identity \eqref{eqn:intensity-pushforward} and Lemma \ref{lem:critical-intensity} \ref{item:critical-intensity-finite} also imply the finiteness
\begin{equation*}
    \Lambda_{k, \phi, \omega}^{+} ((\vartheta_\phi (k - 1), \infty))
    = \Gamma_{k, \phi, \omega}^{+} ((\beta_\phi (k - 1), k))
    \in (0, \infty).
\end{equation*}

Finally, for $\vartheta_\phi (k - 1) < w_\ell < z_{\ell} < \cdots < w_1 < z_1 < \infty$, taking differences in \eqref{eqn:convg-upper-counts-Rn} gives
\begin{equation*}
    \pars[\big]{\scN_{k, n}^{+} ((w_i, z_i])}_{i=1}^\ell
    = \pars[\big]{\fkn_+^{\bfR_n} (w_i) - \fkn_+^{\bfR_n} (z_i)}_{i=1}^\ell 
    \Rightarrow \pars[\big]{\scN_{k}^{+} ((w_i, z_i])}_{i=1}^\ell.
\end{equation*}
It remains to verify tightness of $\{\scN_{k, n}^{+}\}_{n \geq 1}$ in the vague topology. By the tightness criterion for random measures (see, e.g. \cite[Theorem 23.15]{kallenbergFoundationsModernProbability2002}), it suffices to show that $\{\scN_{k, n}^{+} (\bbk)\}_{n \geq 1}$ is tight for every nonempty compact set $\bbk \subset (\vartheta_\phi (k - 1), \infty)$. Fix such a set and choose $z$ with $\vartheta_\phi (k - 1) < z < \inf \bbk$. Since $\scN_{k, n}^{+} (\bbk) \leq \fkn_+^{\bfR_n} (z)$ and all counts are integer-valued, \eqref{eqn:convg-upper-counts-Rn} and Markov's inequality give
\begin{equation*}
    \limsup_{n \to \infty}
    \bbp \{\scN_{k, n}^{+} (\bbk) > L\}
    \leq
    \bbp \{\scN_k^{+} ((z, \infty)) > L\} 
    \leq 
    L^{-1} \Lambda_{k, \phi, \omega}^{+} ((z, \infty)).
\end{equation*}
Here $\Lambda_{k, \phi, \omega}^{+} ((z, \infty)) < \infty$ by finiteness of the intensity $\Lambda_{k, \phi, \omega}^{+}$. Hence, letting $L \to \infty$ establishes tightness and therefore proves the vague convergence $\scN_{k, n}^{+} \Rightarrow \scN_{k}^{+}$ in \eqref{eqn:upper-spectral-poisson}.
\end{proof}

We next use Theorem \ref{thm:upper-poisson} to derive the limiting law of $\lambda_1 (\bfR_n)$. In terms of the effective process $\scP_k^{+}$, the critical law $H_{k, \phi, \omega}^{+}$ in \eqref{eqn:critical-maximum-cdf} is the distribution of
\begin{equation*}
    \vartheta_\phi (k - 1) \vee
    \sup\nolimits_{x \in \scP_k^{+}} \theta_\phi (x).
\end{equation*}
Here the supremum runs over the atoms of $\scP_k^{+}$. This random variable equals the baseline $\vartheta_\phi (k - 1)$ if $\scP_k^{+}$ has no atoms, and otherwise equals the image of its largest atom under the outlier map $\theta_\phi$. In particular, its mass at the baseline is the probability that $\scP_k^{+}$ has no atoms.

% *******************************************
\begin{proof}[Proof of Theorem \ref{thm:critical-maximum}]
% *******************************************
First, Theorem \ref{thm:upper-poisson} shows that $\Lambda_{k, \phi, \omega}^{+}$ is atomless with finite positive total mass. Hence,
\begin{equation*}
    H_{k, \phi, \omega}^{+} (\vartheta_\phi (k - 1))
    =  \exp \curls[\big]{-\Lambda_{k, \phi, \omega}^{+} ((\vartheta_\phi (k - 1), \infty))} \in (0, 1).
\end{equation*}
which yields the asserted atom of $H_{k, \phi, \omega}^{+}$ and its mass. In addition, by Lemma \ref{lem:star-spectrum} \ref{item:star-spectrum-monotonicity}, we have $\gamma_k^{+} (\bfa) < k$ for $\bfa \in \bbr_+^k$, so the definition \eqref{def:critical-spectral-intensity} gives $\Lambda_{k, \phi, \omega}^{+} ([\theta_\phi (k), \infty)) = 0$. This implies the continuity of $H_{k, \phi, \omega}^{+}$ at $\theta_\phi  (k)$.

We now prove \eqref{eqn:weak-convergence-critical}. Fix $z > \vartheta_\phi (k - 1)$. The Poisson convergence \eqref{eqn:fixed-gap-poisson} yields
\begin{equation*}
    \bbp \{\lambda_1 (\bfR_n) \leq z\}
    = \bbp \{\fkn_+^{\bfR_n} (z) = 0\} 
    \to \exp \{-\Lambda_{k, \phi, \omega}^{+} ((z, \infty))\}
    = H_{k, \phi, \omega}^{+} (z).
\end{equation*}
Note that this also applies to $z \geq \theta_\phi (k)$. To prove \eqref{eqn:weak-convergence-critical}, it remains to establish 
\begin{equation}
    \bbp \{\lambda_1 (\bfR_n) \leq z\} \to 0,
    \qquad z < \vartheta_\phi (k - 1).
    \label{eqn:proof-critical-baseline}
\end{equation}
When $\vartheta_\phi (k - 1) = \lambda_+$, this follows from the edge convergence $\lambda_1 (\bfR_{\caT}) = \lambda_+ (\phi) + o_{\bbp} (1)$ of the typical block deduced in \eqref{eqn:eig-typical-to-edge} and Cauchy's interlacing inequality $\lambda_1 (\bfR_{n}) \geq \lambda_1 (\bfR_{\caT})$.

Now suppose $\vartheta_\phi (k - 1) > \lambda_+$, which implies $\beta_\phi (k - 1) = k - 1 > 1 + \sqrt{\phi}$. In particular, $k \geq 3$. Since $\alpha_{k - 1} > \alpha_k$, condition \eqref{cond:critical-collision} implies the supercritical condition $t^{\alpha_{k - 1}} \bar{F}_{\xi} (t) \to \infty$ at order $k-1$. Therefore, we can apply Lemma \ref{lem:strong-stars} at order $k - 1$ to obtain
\begin{equation*}
    \lambda_1 (\bfL_{\caB}) \geq k - 1 - o_{\bbp} (1).
\end{equation*}
For $\lambda_+ < z < \vartheta_\phi (k - 1)$, we have $\fkx_\phi (z) < k - 1$. Hence the spectral reduction \eqref{eqn:spectral-count-sandwich} gives
\begin{equation*}
    \bbp \{\lambda_1 (\bfR_n) \leq z\}
    \leq \bbp \{\lambda_1 (\bfL_{\caB})
    \leq \fkx_\phi (z) + \eps_n\} + o (1)
    \to 0.
\end{equation*}
Monotonicity covers $z \leq \lambda_+$, proving \eqref{eqn:proof-critical-baseline}. This completes the proof of \eqref{eqn:weak-convergence-critical}.
\end{proof}

We now prove the two lower critical results. Throughout these proofs, we take $k = 2$ and $0 < \phi < 1$, and construct $\bfL_{\caB}$ at order $2$ under \eqref{cond:wide-tail-critical}.

% *******************************************
\begin{proof}[Proof of Theorem \ref{thm:lower-poisson}]
% *******************************************
We follow the proof of Theorem \ref{thm:upper-poisson}, using the lower count comparison \eqref{eqn:lower-spectral-count-sandwich} and Corollary \ref{cor:critical-effective-lower-counts} in place of their upper counterparts. The identity $\gamma_2^{-} = 2 - \gamma_2^{+}$ and Lemma \ref{lem:critical-intensity} show that $\Gamma_{2, \phi, \omega}^{-}$ is atomless with finite positive total mass.

Fix $0 < z_1 < \cdots < z_\ell < \lambda_-$ and put $x_i = \fkx_\phi (z_i)$. The argument used to prove \eqref{eqn:proof-critical-count-equivalence} applies to the lower comparison \eqref{eqn:lower-spectral-count-sandwich} and gives
\begin{equation}
    \pars[\big]{\fkn_-^{\bfR_n} (z_i)
    - \fkn_-^{\bfL_{\caB}} (x_i)}_{i = 1}^\ell
    \pconv \mathbf{0}.
    \label{eqn:critical-lower-count-equivalence}
\end{equation}
On the event $\tilde{\Omega}_{2, n}$, the reflection identity \eqref{eqn:reflection-process} gives $\scP_{2, n}^{-} ((0, x_i)) = \scP_{2, n}^{+} ((2 - x_i, 2])$. Each interval $(2 - x_i, 2]$ is relatively compact in $(1 + \sqrt{\phi}, \infty)$, and its endpoints have zero $\Gamma_{2, \phi, \omega}^{+}$-measure. Proposition \ref{prop:collision-spectrum} and \eqref{eqn:critical-lower-count-equivalence} therefore yield
\begin{equation}
    \pars[\big]{\fkn_-^{\bfR_n} (z_i)}_{i = 1}^\ell
    \Rightarrow
    \pars[\big]{\scP_2^{-} ((0, x_i))}_{i = 1}^\ell
    \overset{\mathrm{d}}{=}
    \pars[\big]{\scN_2^{-} ((0, z_i))}_{i = 1}^\ell.
    \label{eqn:convg-lower-counts-Rn}
\end{equation}
Here we have also used the reflection identity $\scP_2^{-} ((0, x_i)) = \scP_2^{+} ((2 - x_i, 2])$ from \eqref{eqn:reflection-limit}. The last equality in distribution follows from the intensity relation $\Lambda_{2, \phi, \omega}^{-} = \Gamma_{2, \phi, \omega}^{-} \circ (\theta_\phi)^{-1}$ and the mapping theorem for Poisson processes, since $\theta_\phi$ is strictly increasing. 

As in the proof of Theorem \ref{thm:upper-poisson}, taking differences in \eqref{eqn:convg-lower-counts-Rn} gives joint convergence of the counts on intervals compactly contained in $(0, \lambda_-)$. Combining this convergence with the analogous tightness argument proves \eqref{eqn:lower-spectral-poisson}. Finally, the homeomorphism $\theta_\phi$ transfers atomlessness and finite positive total mass from $\Gamma_{2, \phi, \omega}^{-}$ to $\Lambda_{2, \phi, \omega}^{-}$, proving \eqref{eqn:lower-intensity-properties}.

It remains to prove \eqref{eqn:lower-fixed-gap-poisson}. We first control eigenvalues near zero, including exact zeros. For each fixed sufficiently small $\eps \in (0, \lambda_- / 2)$, the lower count comparison \eqref{eqn:lower-spectral-count-sandwich} gives, 
\begin{equation*}
    \bbp \{\lambda_{p_n} (\bfR_n) \leq \eps\}
    \leq \bbp \{\fkn_-^{\bfR_n} (2 \eps) > 0\} 
    \leq \bbp \{\fkn_-^{\bfL_{\caB}}
    (\fkx_\phi (2 \eps) + \eps_n) > 0\} + o (1) 
    \leq \bbp \{\fkn_-^{\bfL_{\caB}}
    (C \eps) > 0\} + o (1).
\end{equation*}
Here the last inequality uses $\eps_n \to 0$, \eqref{def:upper-outlier-map} and \eqref{def:inverse-map}: for each fixed sufficiently small $\eps$, we have $\fkx_\phi (2 \eps) + \eps_n \leq C \eps < 1$ once $n$ is sufficiently large, with $C > 0$ independent of $\eps$.

On $\tilde{\Omega}_{2, n}$, the block diagonal representation \eqref{eqn:unified-component-cap} shows that $\bfL_{\caB}$ consists of scalar blocks equal to $1$ and $2 \times 2$ blocks with lower eigenvalues $\gamma_2^{-} (\bfa (\caC))$, where $\caC \in \bbg_{2\star}$. Thus an eigenvalue below $C \eps < 1$ must come from a star $\caC \in \bbg_{2\star}$ with $\gamma_2^{-} (\bfa (\caC)) \leq C \eps$. By \eqref{def:lower-star-eigenvalue}, this forces $a_1 (\caC) \wedge a_2 (\caC) > c \eps^{-1/2}$ for some constant $c > 0$ independent of $\eps$. Conditioning on any realization of $\caH_n$ in $\tilde{\Omega}_{2, n}$ fixes the number $T_{2, n} = \abs{\bbg_{2\star}}$ of stars with two row leaves, while the amplitude vectors $\{\bfa (\caC)\}_{\caC \in \bbg_{2\star}}$ are independent with common law $\rho_n^{\otimes 2}$. The union bound therefore gives
\begin{equation*}
    \bbp^{\caH} \{\fkn_-^{\bfL_{\caB}} (C \eps) > 0\}
    \leq \bbp^{\caH} \curls[\big]{\exists \, \caC \in \bbg_{2\star}:
    a_1 (\caC) \wedge a_2 (\caC) > c \eps^{-1/2}}
    \leq T_{2, n}
    \braks[\big]{\rho_n ((c \eps^{-1/2}, \infty))}^2.
\end{equation*}
The definition of $\rho_n$ in \eqref{def:conditional-amplitude-law}, the conditional large-entry law \eqref{eqn:alpha3-resampled-psi}, and the critical tail condition \eqref{cond:wide-tail-critical}, together with \eqref{eqn:critical-truncation-probability}, give, for all sufficiently large $n$,
\begin{equation*}
    \rho_n ((c \eps^{-1/2}, \infty))
    = \fkq_n^{-1}
    \bbp \{\abs{\xi} > c \sqrt{n / \eps}\}
    \leq C \eps^{3/2} n^{-3 \tau_2},
\end{equation*}
where $C > 0$ is independent of $\eps$. Fix $M > \phi^2 \omega^2 / 2$. Applying \eqref{eqn:critical-star-count-limit} at order $2$ gives $\bbp \{T_{2, n} > M n^{6 \tau_2}\} = o (1)$. Averaging the conditional bound above and using $\bbp (\tilde{\Omega}_{2, n}) \to 1$, we obtain
\begin{equation*}
    \bbp \{\lambda_{p_n} (\bfR_n) \leq \eps\}
    \leq M n^{6 \tau_2}
    \braks[\big]{\rho_n ((c \eps^{-1/2}, \infty))}^2 + o (1)
    \leq C \eps^3 + o (1).
\end{equation*}
Taking the limit superior as $n \to \infty$ and then letting $\eps \downarrow 0$, we obtain
\begin{equation}
    \lim_{\eps \downarrow 0} \limsup_{n \to \infty}
    \bbp \{\lambda_{p_n} (\bfR_n) \leq \eps\} = 0.
    \label{eqn:critical-lower-no-points-near-zero}
\end{equation}

Now fix $\lambda \in (0, \lambda_-)$. Since $\Lambda_{2, \phi, \omega}^{-}$ is atomless, \eqref{eqn:lower-spectral-poisson} gives, for each fixed $\eps \in (0, \lambda)$,
\begin{equation*}
    \scN_{2, n}^{-} ((\eps, \lambda])
    \Rightarrow \operatorname{Poisson}
    \pars[\big]{\Lambda_{2, \phi, \omega}^{-} ((\eps, \lambda])}.
\end{equation*}
Note that the full eigenvalue count in \eqref{eqn:lower-fixed-gap-poisson} differs from this truncated count precisely when $\lambda_{p_n} (\bfR_n) \leq \eps$. Therefore, \eqref{eqn:critical-lower-no-points-near-zero} implies
\begin{equation*}
    \begin{split}
        & \lim_{\eps \downarrow 0} \limsup_{n \to \infty}
        \bbp \curls*{\abs[\big]{\{i \in \dbraks{p_n} : \lambda_i (\bfR_n) \leq \lambda\}}
        \neq \scN_{2, n}^{-} ((\eps, \lambda])} 
        = 0.
    \end{split}
\end{equation*}
Finally, continuity from below of the intensity gives
\begin{equation*}
    \lim_{\eps \downarrow 0}
    \Lambda_{2, \phi, \omega}^{-} ((\eps, \lambda])
    = \Lambda_{2, \phi, \omega}^{-} ((0, \lambda]).
\end{equation*}
Combining these three limits proves \eqref{eqn:lower-fixed-gap-poisson}.
\end{proof}

% *******************************************
\begin{proof}[Proof of Theorem \ref{thm:critical-minimum}]
% *******************************************
First, Theorem \ref{thm:lower-poisson} shows that $\Lambda_{2, \phi, \omega}^{-}$ is atomless with finite positive total mass. By \eqref{eqn:critical-minimum-cdf}, the distribution function $H_{2, \phi, \omega}^{-}$ is therefore continuous away from $\lambda_-$. Here continuity at zero follows from $\Lambda_{2, \phi, \omega}^{-} ((0, z)) \to 0$ as $z \downarrow 0$. Moreover,
\begin{equation*}
    1 - \lim_{z \uparrow \lambda_-} H_{2, \phi, \omega}^{-} (z)
    = \exp \{-\Lambda_{2, \phi, \omega}^{-} ((0, \lambda_-))\}
    \in (0, 1).
\end{equation*}
Since $H_{2, \phi, \omega}^{-} (\lambda_-) = 1$, this gives the mass of the atom at $\lambda_-$.

We now establish \eqref{eqn:weak-convergence-critical-minimum}. Fix $z \in (0, \lambda_-)$. The Poisson convergence \eqref{eqn:lower-fixed-gap-poisson} yields
\begin{equation*}
    \bbp \{\lambda_{p_n} (\bfR_n) > z\}
    = \bbp \curls*{\abs[\big]{\{i \in \dbraks{p_n} :
    \lambda_i (\bfR_n) \leq z\}} = 0} 
    \to \exp \{-\Lambda_{2, \phi, \omega}^{-} ((0, z])\}
    = 1 - H_{2, \phi, \omega}^{-} (z),    
\end{equation*}
where the last equality uses the atomlessness of $\Lambda_{2, \phi, \omega}^{-}$. Taking complements gives
\begin{equation*}
    \bbp \{\lambda_{p_n} (\bfR_n) \leq z\}
    \to H_{2, \phi, \omega}^{-} (z),
    \qquad 0 < z < \lambda_-.
\end{equation*}
Next, the lower edge convergence \eqref{eqn:eig-typical-to-edge} for typical block and Cauchy's interlacing inequality imply $\lambda_{p_n} (\bfR_n) \leq \lambda_{\abs{\caT}} (\bfR_{\caT}) = \lambda_- + o_{\bbp} (1)$. Consequently,
\begin{equation*}
    \bbp \{\lambda_{p_n} (\bfR_n) \leq z\} \to 1 = H_{2, \phi, \omega}^{-} (z),
    \qquad
    z > \lambda_-.
\end{equation*}
To prove \eqref{eqn:weak-convergence-critical-minimum}, it remains to verify convergence at $z = 0$. By \eqref{eqn:critical-lower-no-points-near-zero}, we have
\begin{equation*}
    \bbp \{\lambda_{p_n} (\bfR_n) \leq 0\}
    = \bbp \{\rank (\bfY_n) < p_n\}
    \to 0 = H_{2, \phi, \omega}^{-} (0).
\end{equation*}
Therefore, we have established convergence of the distribution functions at every continuity point of $H_{2, \phi, \omega}^{-}$, which concludes the proof of \eqref{eqn:weak-convergence-critical-minimum}.
\end{proof}

\appendix
% *******************************************
\section{The tall regime}
\label{sec:tall}
% *******************************************

This section establishes Theorem \ref{thm:tall} by adapting Tikhomirov's argument \cite{tikhomirovLimitSmallestSingular2015} for the smallest eigenvalue of the sample covariance matrix $\bfS_n = (1/n) \, \bfX_n \bfX_n^\top$. We briefly recall the high-level idea therein. Let $\bfX_n \in \bbr^{p_n \times n}$ be the random matrix of i.i.d. centered, unit-variance entries, and suppose that $p_n / n \to \phi \in (1, \infty)$. Tikhomirov decomposes $\bfX_n = \tilde{\bfX}_n + \hat{\bfX}_n$ by the centered fixed-level truncation
\begin{subequations} \label{eqn:T-truncation}
\begin{align}
    \tilde{x}_{i \mu}
    & := x_{i \mu} \mathbb{1} \curls{\abs{x_{i \mu}} \leq T}
    - \bbe \braks{x_{i \mu} \mathbb{1} \curls{\abs{x_{i \mu}} \leq T}},
    \label{def:tilde-x} \\
    \hat{x}_{i \mu}
    & := x_{i \mu} \mathbb{1} \curls{\abs{x_{i \mu}} > T}
    - \bbe \braks{x_{i \mu} \mathbb{1} \curls{\abs{x_{i \mu}} > T}},
    \label{def:hat-x}
\end{align}
\end{subequations}
where $T > 0$ is fixed independently of $n$. Since the entries of $\tilde{\bfX}_n$ are i.i.d. and bounded, the Bai--Yin theorem \cite{baiLimitSmallestEigenvalue1993} controls the smallest singular value $s_n(\tilde{\bfX}_n)$ for the truncated component. In contrast, under only a second moment assumption, one cannot require $\norm{\hat{\bfX}_n} = o(\sqrt{p_n})$, and thus the usual operator-norm perturbation bound is unavailable.

Tikhomirov overcomes this difficulty by a row-selection argument. Fix $\delta > 0$, and let $\mathbb{S}^{n - 1}$ denote the Euclidean unit sphere in $\bbr^n$. Applied to the normalized tail $\hat{\bfX}_n$, the row filtering estimate \cite[Proposition 13]{tikhomirovLimitSmallestSingular2015} selects, for every $\bfv \in \mathbb{S}^{n - 1}$, at least $(1 - \varepsilon) p_n$ rows on which $\hat{\bfX}_n \bfv$ is $O(\delta \sqrt{p_n})$. The projection stability estimate \cite[Proposition 14]{tikhomirovLimitSmallestSingular2015} then controls the loss in the smallest singular value of $\tilde{\bfX}_n$ after those rows are discarded. Together, these results give the high-probability bounds
\begin{align}
\begin{split}
    s_n (\bfX_n) 
    & \geq \min_{\abs{\caI} \geq p_n - \varepsilon p_n}
    s_n(\bfP_{\caI}\tilde{\bfX}_n) - C \delta \sqrt{p_n}, \\
    s_n(\tilde{\bfX}_n)
    & \leq \min_{\abs{\caI} \geq p_n - \varepsilon p_n}
    s_n(\bfP_{\caI}\tilde{\bfX}_n) 
    + \delta \sqrt{p_n}.    
\end{split} \label{eqn:compare-X-row-select}
\end{align}
Combining these bounds and letting $\delta \downarrow 0$ gives the required almost-sure lower bound.

We now turn to the smallest singular value $s_n(\bfY_n)$ of the self-normalized matrix and use the same fixed-level truncation \eqref{eqn:T-truncation}. Let $\tilde{\bfx}_i = (\tilde{x}_{i 1}, \ldots, \tilde{x}_{i n})^\top$ denote the $i$-th row of $\tilde{\bfX}_n$. The row-normalized truncated matrix $\tilde{\bfY}_n \in \bbr^{p_n \times n}$ is defined by
\begin{equation*}
    \tilde{\bfY}_n^\top
    := (\tilde{\bfy}_1, \ldots, \tilde{\bfy}_{p_n}),
    \qquad
    \tilde{\bfy}_i
    := {\tilde{\bfx}_i} / {\norm{\tilde{\bfx}_i}}.
\end{equation*}
Analogously to \eqref{eqn:compare-X-row-select}, we prove
\begin{align}
\begin{split}
    s_n (\bfY_n) 
    & \geq \min_{\abs{\caI} \geq p_n - \varepsilon p_n}
    s_n (\bfP_{\caI} \tilde{\bfY}_n) - C \delta, \\
    s_n(\tilde{\bfY}_n)
    & \leq \min_{\abs{\caI} \geq p_n - \varepsilon p_n}
    s_n (\bfP_{\caI} \tilde{\bfY}_n) + \delta.    
\end{split} \label{eqn:compare-Y-row-select}
\end{align}
The two inequalities are established in Propositions \ref{prop:gap-truncation} and \ref{prop:gap-projection}, respectively. For the self-normalized matrix, the identity $\bfX_n = \tilde{\bfX}_n + \hat{\bfX}_n$ does not induce an additive decomposition of $\bfY_n$, because the denominators also change. The main additional task here is therefore to control the normalizing factors $\norm{\bfx_i}$ and $\norm{\tilde{\bfx}_i}$ and their difference; see Lemmas \ref{lemma:norm-bounded-vectors} and \ref{lemma:tail-vectors}.

For any fixed sufficiently small $\delta > 0$, we choose $T \equiv T(\delta) > 0$ sufficiently large that
\begin{equation}
    \bbe \abs{\tilde{x}_{i \mu}}^2 \geq (1 - \delta)^2,
    \qquad
    \bbe \abs{\hat{x}_{i \mu}}^2 \leq \bbe \braks[\big]{\abs{x_{i \mu}}^2
    \mathbb{1} {\curls{\abs{x_{i \mu}} > T}}} \leq \delta^2.
    \label{eqn:variance}
\end{equation}
By enlarging $T$ if necessary, we may assume $T \geq 1$. Throughout this section, we assume $n$ is sufficiently large such that $n \leq p_n \leq 2 \phi n$. In particular, since $p_n \geq n$, we may use the following variational characterization of the smallest singular value,
\begin{equation}
    s_{n} (\bfY_n)
    := \inf\nolimits_{\norm{\bfv} = 1} \norm{\bfY_n \bfv}.
\end{equation}

\begin{remark}
In principle, the spectral reduction in Section \ref{sec:truncation} could also be extended to analyze the tall regime under the additional tail bound \eqref{eqn:analytic-tail-envelope} for some fixed integer $k \geq 2$. Indeed, when $\phi > 1$, the companion MP law has no atom at zero and is supported on $[\lambda_-, \lambda_+]$. Its Stieltjes transform $\fkm_\phi$ is therefore positive on $(0, \lambda_-)$. Since the effective Gram matrix $\bfL_{\caB}$ in \eqref{eqn:unified-component-matrix} is positive definite, we have
\begin{equation*}
    \bfI + \fkm_\phi (z) \bfL_{\caB} \geq \bfI,
    \qquad 0 < z < \lambda_-.
\end{equation*}
Consequently, the outlier equation $\det [\bfI + \fkm_\phi (z) \bfL_{\caB}] = 0$ has no solution in $(0, \lambda_-)$, even when $\lambda_{\abs{\caB}} (\bfL_{\caB})$ approaches zero. That is, small eigenvalues of $\bfL_{\caB}$ cannot generate lower nonzero outliers of $\bfR_n$. This mechanism has already been noted by Baik and Silverstein \cite[Theorem 1.2]{baikEigenvaluesLargeSample2006} for spiked sample covariance matrices: when $\phi > 1$, small population spikes do not change the limiting smallest sample eigenvalue.

However, the spectral reduction estimates in Section \ref{sec:truncation} require the polynomial tail bound \eqref{eqn:analytic-tail-envelope}, which is stronger than a finite second moment. Therefore, here we choose to adapt Tikhomirov's argument \cite{tikhomirovLimitSmallestSingular2015} to directly establish the almost sure convergence in Theorem \ref{thm:tall} under $\bbe \abs{\xi}^2 < \infty$ alone.
\end{remark}

% *******************************************
\subsection{Projection stability}
% *******************************************

We first establish a lower bound for the row normalizers using Bernstein's inequality.

% *******************************************
\begin{lemma}
\label{lemma:norm-bounded-vectors}
% *******************************************
Suppose $\delta \leq 1 / 4$. Then, for all sufficiently large $n$,
\begin{equation}
    \bbp \curls[\big]{
    \min\nolimits_{i \in \dbraks{p_n}}
    \pars[\big]{\norm{\tilde{\bfx}_i}^2
    \wedge \norm{\bfx_i}^2}
    \geq {n} / {2}}
    \geq 1 - \exp(-c_{\mathrm{\ref{eqn:lower-row-norms}}} n / T^2),
    \label{eqn:lower-row-norms}
\end{equation}
where $c_{\mathrm{\ref{eqn:lower-row-norms}}} > 0$ is a universal constant.
\end{lemma}

% *******************************************
\begin{proof}[Proof of Lemma \ref{lemma:norm-bounded-vectors}]
% *******************************************
Recall \eqref{def:tilde-x}. We have $\abs{\tilde{x}_{i \mu}} \leq 2 T$ and $\bbe \abs{\tilde{x}_{i \mu}}^2 \leq \bbe \abs{{x}_{i \mu}}^2 = 1$. Hence,
\begin{equation*}
    \abs[\big]{\abs{\tilde{x}_{i \mu}}^2 - \bbe \abs{\tilde{x}_{i \mu}}^2}
    \leq 8 T^2,
    \qquad
    \bbe \abs*{\abs{\tilde{x}_{i \mu}}^2 - \bbe \abs{\tilde{x}_{i \mu}}^2}^2
    \leq \bbe \abs{\tilde{x}_{i \mu}}^4
    \leq 4 T^2.
\end{equation*}
Bernstein's inequality (\cite[Theorem 2.8.4]{vershyninHighdimensionalProbabilityIntroduction2018}) gives, for a universal constant $c > 0$,
\begin{equation*}
    \bbp \curls*{\abs*{
    \sum\nolimits_{\mu = 1}^{n}
    \pars{\abs{\tilde{x}_{i \mu}}^2 - \bbe \abs{\tilde{x}_{i \mu}}^2}} \geq \frac{n}{16}}
    \leq 2 \exp(-c n / T^2),
    \qquad
    \forall i \in \dbraks{p_n}.
\end{equation*}
Note that $\bbe \abs{\tilde{x}_{i \mu}}^2 \geq (1 - \delta)^2 \geq 9 / 16$ by \eqref{eqn:variance} and $\delta \leq 1 / 4$. Thus the above estimate implies
\begin{equation*}
    \bbp \curls{\norm{\tilde{\bfx}_i}^2 \geq n / 2}
    \geq 1 - 2 \exp(-c n / T^2).
\end{equation*}
For the lower bound on $\norm{\bfx_i}$, put $a_{i \mu} = x_{i \mu} \mathbb{1} \curls{\abs{x_{i \mu}} \leq T}$. Then we have $\abs{a_{i \mu}} \leq T$ and $1 - \delta^2 \leq \bbe \abs{a_{i \mu}}^2 \leq 1$. Therefore, the same argument gives
\begin{equation}
    \bbp \curls{\norm{\bfx_i}^2 \geq n / 2}
    \geq \bbp \curls{\norm{\mathbf{a}_i}^2 \geq n / 2}
    \geq 1 - 2 \exp(-c n / T^2).
    \label{eqn:lower-norm-x}
\end{equation}
The union bound and $p_n = O(n)$ complete the proof.
\end{proof}

The following proposition adapts \cite[Proposition 14]{tikhomirovLimitSmallestSingular2015} to self-normalized matrices. We use the following standard inequality for the Orlicz $\psi_2$-norm (see, e.g., \cite[Proposition 2.6.1]{vershyninHighdimensionalProbabilityIntroduction2018}). If $\{ a_\mu \}_{\mu = 1}^n$ are independent, mean-zero sub-Gaussian random variables, then
\begin{equation}
    \norm*{\sum\nolimits_{\mu = 1}^{n} a_{\mu}}_{\psi_{2}}^{2}
    \leq C_{\mathrm{\ref{eqn:sum-subgaussian}}}
    \sum\nolimits_{\mu = 1}^{n} \norm{a_{\mu}}_{\psi_{2}}^{2},
    \label{eqn:sum-subgaussian}
\end{equation}
where $C_{\mathrm{\ref{eqn:sum-subgaussian}}} > 0$ is a universal constant.

% *******************************************
\begin{proposition}
\label{prop:gap-projection}
% *******************************************
Suppose that the assumptions of Theorem \ref{thm:tall} hold, $\delta \leq 1 / 4$, and $\delta^2 \leq 8 c_{\mathrm{\ref{eqn:lower-row-norms}}} C_{\mathrm{\ref{eqn:sum-subgaussian}}}$. Choose $\varepsilon > 0$ sufficiently small such that $\varepsilon < (\phi - 1) / (2 \phi)$ and
\begin{equation}
    \varepsilon \ln (3 e) + \varepsilon \ln (2 e / \varepsilon)
    \leq \frac{(\ln 2) \delta^2}{16 C_{\mathrm{\ref{eqn:sum-subgaussian}}} \phi T^2}.
    \label{eqn:choice-eps}
\end{equation}
Then, there exists $N \in \bbn$ such that, for all $n \geq N$,
\begin{equation}
    \bbp \curls[\Big]{ \,
    \min_{\abs{\caI} \geq p_n - \varepsilon p_n}
    s_n(\bfP_{\caI} \tilde{\bfY}_n)
    \geq s_n(\tilde{\bfY}_n) - \delta \, }
    \geq 1 - \exp(-\varepsilon p_n).
    \label{eqn:gap-estimate-I}
\end{equation}
\end{proposition}

% *******************************************
\begin{proof}[Proof of Proposition \ref{prop:gap-projection}]
% *******************************************
We can choose $N \in \bbn$ sufficiently large such that the estimate \eqref{eqn:lower-row-norms} and the following inequalities hold for all $n \geq N$:
\begin{equation*}
    \varepsilon p_n \geq 1,
    \qquad
    p_n \leq 2 \phi n,
    \qquad
    p_n - \varepsilon p_n > n,
    \qquad
    p_n - \lceil p_n - \varepsilon p_n \rceil \geq \varepsilon p_n / 2.
\end{equation*}
Following \cite[Proposition 14]{tikhomirovLimitSmallestSingular2015}, we prove \eqref{eqn:gap-estimate-I} by contradiction. Suppose that it fails for some $n \geq N$. Then
\begin{equation*}
    \bbp \curls[\Big]{ \,
    \min_{\abs{\caI} \geq p_n - \varepsilon p_n}
    s_n(\bfP_{\caI} \tilde{\bfY}_n)
    < s_n(\tilde{\bfY}_n) - \delta \, }
    > \exp(-\varepsilon p_n).
\end{equation*}
Consequently, there exists $\caI_0 \subset \dbraks{p_n}$ such that $\abs{\caI_0} = \lceil p_n - \varepsilon p_n \rceil$ and
\begin{equation}
    \bbp \curls[\big]{s_n(\tilde{\bfY}_n)
    > s_n(\bfP_{\caI_0} \tilde{\bfY}_n) + \delta}
    > \binom{p_n}{\lceil p_n - \varepsilon p_n \rceil}^{-1}
    \exp(-\varepsilon p_n)
    \geq \frac{\exp(-\varepsilon p_n)}{(2 e / \varepsilon)^{\varepsilon p_n}}.
    \label{eqn:contradiction-1}
\end{equation}
We next bound the left-hand side of \eqref{eqn:contradiction-1}. Since $\abs{\caI_0} = \lceil p_n - \varepsilon p_n \rceil > n$, the projected matrix $\bfP_{\caI_0} \tilde{\bfY}_n$ retains more than $n$ rows. Hence, we can choose a measurable random vector $\bfv_0 \in \mathbb{S}^{n - 1}$, depending only on $\bfP_{\caI_0} \tilde{\bfY}_n$, such that $\norm{\bfP_{\caI_0} \tilde{\bfY}_n \bfv_0} = s_n (\bfP_{\caI_0} \tilde{\bfY}_n)$. The variational characterization then gives $\norm{\tilde{\bfY}_n\bfv_0} \geq s_n(\tilde{\bfY}_n)$. Therefore,
\begin{align*}
    s_n(\tilde{\bfY}_n)
    > s_n(\bfP_{\caI_0} \tilde{\bfY}_n) + \delta
    & \qquad \Longrightarrow \qquad
    \norm{\tilde{\bfY}_n \bfv_0} 
    > \norm{\bfP_{\caI_0} \tilde{\bfY}_n \bfv_0} + \delta \\
    & \qquad \Longrightarrow \qquad
    \norm{\bfP_{\caI_0^c} \tilde{\bfY}_n \bfv_0}
    > \delta.
\end{align*}
Here $\caI_0$ is deterministic. Since $\bfv_0$ is a measurable function of $\bfP_{\caI_0} \tilde{\bfY}_n$ and the rows of $\tilde{\bfY}_n$ are independent, the random vector $\bfv_0$ is independent of $\bfP_{\caI_0^c} \tilde{\bfY}_n$. Conditioning on $\bfv_0$ then leads to the upper bound
\begin{align}
\begin{split}
    & ~ \bbp \curls[\big]{s_n(\tilde{\bfY}_n) > s_n(\bfP_{\caI_0} \tilde{\bfY}_n) + \delta} \\
    \leq & ~ \bbe \braks[\big]{\bbp \curls[\big]{\norm{\bfP_{\caI_0^c} \tilde{\bfY}_n \bfv_0} > \delta \!\mid\! \bfv_0}} 
    \leq \sup\nolimits_{\bfv \in \mathbb{S}^{n - 1}} 
    \bbp \curls[\big]{\norm{\bfP_{\caI_0^c} \tilde{\bfY}_n \bfv} > \delta}.
    \label{eqn:contradiction-2}    
\end{split}
\end{align}
Fix $\bfv \in \mathbb{S}^{n - 1}$. Lemma \ref{lemma:norm-bounded-vectors} implies that
\begin{align} \label{eqn:contradiction-3}
\begin{split}
    \bbp \curls[\big]{\norm{\bfP_{\caI_0^c} \tilde{\bfY}_n \bfv} > \delta}
    & = \bbp \curls*{\sum\nolimits_{i \notin \caI_0}
    {\abs{\angles{\tilde{\bfx}_i, \bfv}}^2}
    / {\norm{\tilde{\bfx}_i}^2}
    > \delta^2} \\
    & \leq \bbp \curls[\Big]{\sum\nolimits_{i \notin \caI_0} \abs{\angles{\tilde{\bfx}_i, \bfv}}^2
    > {\delta^2 n} / {2}}
    + \exp(-c_{\mathrm{\ref{eqn:lower-row-norms}}} n / T^2).
\end{split}
\end{align}
Since $\abs{\tilde{x}_{i \mu}} \leq 2 T$, we have $\norm{\tilde{x}_{i \mu}}_{\psi_2} \leq 2 T / \sqrt{\ln 2}$. Consequently, \eqref{eqn:sum-subgaussian} gives, 
\begin{equation*}
    \norm{\angles{\tilde{\bfx}_i, \bfv}}_{\psi_2}^2
    \leq 4 C_{\mathrm{\ref{eqn:sum-subgaussian}}} T^2 / \ln 2,
    \qquad
    \forall \bfv \in \mathbb{S}^{n - 1}.
\end{equation*}
By definition of the Orlicz $\psi_2$-norm, this means that
\begin{equation*}
    \bbe \exp\pars{\lambda \abs{\angles{\tilde{\bfx}_i, \bfv}}^2} \leq 2,
    \qquad
    1 / \lambda = 4 C_{\mathrm{\ref{eqn:sum-subgaussian}}} T^2 / \ln 2.
\end{equation*}
A standard exponential Markov bound now gives
\begin{equation}
    \bbp \curls[\Big]{\sum\nolimits_{i \notin \caI_0}
    \abs{\angles{\tilde{\bfx}_i, \bfv}}^2
    > {\delta^2 n} / {2}}
    \leq \frac{\prod_{i \notin \caI_0} \bbe \exp\pars{\lambda \abs{\angles{\tilde{\bfx}_i, \bfv}}^2}}{\exp(\lambda \delta^2 n / 2)}
    \leq \frac{2^{\varepsilon p_n}}{\exp(\lambda \delta^2 n / 2)}.
    \label{eqn:contradiction-4}
\end{equation}
Combining the estimates in \eqref{eqn:contradiction-1}--\eqref{eqn:contradiction-4} yields
\begin{align*}
    (2 e / \varepsilon)^{-\varepsilon p_n}
    \exp(-\varepsilon p_n)
    & < 2^{\varepsilon p_n}
    \exp(- \lambda \delta^2 n / 2)
    + \exp(-c_{\mathrm{\ref{eqn:lower-row-norms}}} n / T^2) \\
    & \leq (2^{\varepsilon p_n} + 1)
    \exp(- \lambda \delta^2 n / 2)
    \leq 3^{\varepsilon p_n}
    \exp \braks{- \lambda \delta^2 p_n / (4 \phi)}.
\end{align*}
Here the second step uses $\delta^2 \leq 8 c_{\mathrm{\ref{eqn:lower-row-norms}}} C_{\mathrm{\ref{eqn:sum-subgaussian}}}$, and the final inequality uses both $p_n \leq 2 \phi n$ and $2^{\varepsilon p_n} + 1 \leq 3^{\varepsilon p_n}$, where the latter follows from $\varepsilon p_n \geq 1$. Consequently,
\begin{equation*}
    - \varepsilon - \varepsilon \ln (2 e / \varepsilon)
    < - \frac{(\ln 2) \delta^2}{16 C_{\mathrm{\ref{eqn:sum-subgaussian}}} \phi T^2} + \varepsilon \ln 3.
\end{equation*}
However, this contradicts the choice of $\varepsilon$ in \eqref{eqn:choice-eps}.
\end{proof}

% *******************************************
\subsection{Truncation comparison}
% *******************************************

For each $i \in \dbraks{p_n}$, define the difference between the original and truncated squared row norms by
\begin{equation*}
    \Delta_i = \norm{\bfx_i}^2
    - \norm{\tilde{\bfx}_i}^2.
\end{equation*}
We control the number of rows for which $\Delta_i$ deviates substantially from its expectation.

% *******************************************
\begin{lemma}
\label{lemma:tail-vectors}
% *******************************************
Suppose that the assumptions of Theorem \ref{thm:tall} hold. Fix $\varepsilon \in (0, \delta]$. Then
\begin{equation*}
    \bbp \curls*{
    \abs*{\curls{i \in \dbraks{p_n} : \abs{\Delta_i} \geq 3 \delta n}}
    \geq \varepsilon p_n}
    \leq \exp (- c \varepsilon p_n),
\end{equation*}
for all sufficiently large $n$, where $c > 0$ is a universal constant.
\end{lemma}

% *******************************************
\begin{proof}[Proof of Lemma \ref{lemma:tail-vectors}]
% *******************************************
Set $\theta := \bbe \abs{x_{i \mu}}^2 - \bbe \abs{\tilde{x}_{i \mu}}^2$. Then
\begin{equation*}
    \abs{\theta}
    = \abs[\big]{\bbe \braks{\hat{x}_{i \mu} \pars{x_{i \mu} + \tilde{x}_{i \mu}}}}
    \leq \pars[\big]{\bbe \abs{\hat{x}_{i \mu}}^2 \cdot \bbe \abs{x_{i \mu} + \tilde{x}_{i \mu}}^2}^{1 / 2}
    \leq \pars[\big]{\bbe \abs{\hat{x}_{i \mu}}^2 \cdot 4 \bbe \abs{x_{i \mu}}^2}^{1 / 2}
    \leq 2 \delta.
\end{equation*}
Recall the following form of the weak law of large numbers (see, e.g., \cite[Theorem 6.17]{kallenbergFoundationsModernProbability2002}). If $\{ a_\mu \}_{\mu \geq 1}$ are i.i.d. integrable random variables with mean zero, then, for every $r, s > 0$, there exists $N \in \bbn$, depending only on $r$, $s$, and the distribution of $a_\mu$, such that
\begin{equation*}
    \bbp \curls[\bigg]{\abs[\Big]{\frac{1}{n}
    \sum\nolimits_{\mu = 1}^{n} a_\mu} \geq r} \leq s,
    \qquad \forall n \geq N.
\end{equation*}
Applying this result with $a_\mu = \abs{x_{i \mu}}^2 - \abs{\tilde{x}_{i \mu}}^2 - \theta$, $r = \varepsilon$, and $s = \varepsilon / (2 e)$, we obtain
\begin{equation*}
    \bbp \curls{\abs{\Delta_i - \theta n} \geq \varepsilon n}
    = \bbp \curls[\bigg]{\abs[\Big]{\frac{1}{n} \sum\nolimits_{\mu = 1}^n \pars[\big]{\abs{x_{i \mu}}^2 - \abs{\tilde{x}_{i \mu}}^2 - \theta}} \geq \varepsilon}
    \leq \frac{\varepsilon}{2 e},
    \qquad \forall i \in \dbraks{p_n},
\end{equation*}
for all sufficiently large $n$. Since $\varepsilon \leq \delta$ and $\abs{\theta} \leq 2 \delta$, the above estimate implies
\begin{equation*}
    \bbp \curls{\abs{\Delta_i} \geq 3 \delta n}
    \leq \bbp \curls{\abs{\Delta_i} \geq \pars{\abs{\theta} + \varepsilon} n}
    \leq \varepsilon / (2 e),
    \qquad \forall i \in \dbraks{p_n}.
\end{equation*}
Set $\chi_i := \mathbb{1} \curls{\abs{\Delta_i} \geq 3 \delta n}$. Note that $(\chi_i)_{i=1}^{p_n}$ are independent Bernoulli random variables with success probability $\bbe \chi_i \leq \varepsilon / (2 e)$. Hence, Chernoff's bound (\cite[Theorem 2.3.1]{vershyninHighdimensionalProbabilityIntroduction2018}) yields
\begin{equation*}
    \bbp \curls*{\sum\nolimits_{i = 1}^{p_n} \chi_i \geq \varepsilon p_n}
    \leq \pars{{e \bbe \chi_i} / {\varepsilon}}^{\varepsilon p_n}
    \leq 2^{-\varepsilon p_n}.
\end{equation*}
This proves the claim.
\end{proof}

The following result is taken from \cite[Proposition 13]{tikhomirovLimitSmallestSingular2015}.

% *******************************************
\begin{proposition}
\label{prop:tail-matrix}
% *******************************************
Suppose $p_n \geq n$, and let $\mathbf{A} = (a_{i \mu})$ be a $p_n \times n$ random matrix with i.i.d. mean-zero, unit-variance entries. For every $\varepsilon \in (0, 1]$, there exists $w \equiv w(\varepsilon) > 0$ such that, for all sufficiently large $n$,
\begin{equation*}
    \bbp \curls[\Big]{\,
    \sup_{\bfv \in \mathbb{S}^{n - 1}}
    \min_{\abs{\caI}
    \geq p_n - \varepsilon p_n}
    \norm{\bfP_\caI \mathbf{A} \bfv}
    \leq C \sqrt{p_n} \, }
    \geq 1 - \exp(-w p_n),
\end{equation*}
where $C > 0$ is a universal constant.
\end{proposition}

% *******************************************
\begin{proposition}
\label{prop:gap-truncation}
% *******************************************
Suppose that the assumptions of Theorem \ref{thm:tall} hold, $\delta \leq 1 / 4$ and $\delta^2 \leq 8 c_{\mathrm{\ref{eqn:lower-row-norms}}} C_{\mathrm{\ref{eqn:sum-subgaussian}}}$. Fix $\varepsilon \in (0, \delta]$. Then there exists $w \equiv w(\varepsilon, \delta) > 0$ such that
\begin{equation*}
    \bbp \curls[\Big]{ \,
    s_n (\bfY_n)
    \geq \min_{\abs{\caI} \geq p_n - \varepsilon p_n}
    s_n (\bfP_{\caI} \tilde{\bfY}_n)
    - C_\phi \delta \, }
    \geq 1 - \exp(-w p_n),
\end{equation*}
for all sufficiently large $n$. Here $C_\phi > 0$ depends only on $\phi$.
\end{proposition}

% *******************************************
\begin{proof}[Proof of Proposition \ref{prop:gap-truncation}]
% *******************************************
We apply Lemma \ref{lemma:norm-bounded-vectors}, Lemma \ref{lemma:tail-vectors}, and Proposition \ref{prop:tail-matrix} to $\hat{\bfX}_n$ and $\tilde{\bfX}_n$. For the last application, rescale $\hat{\bfX}_n$ and $\tilde{\bfX}_n$ by their entrywise standard deviations. If the tail variance is zero, the estimate for $\hat{\bfX}_n$ is trivial. By \eqref{eqn:variance}, rescaling back gives the factors $\delta$ and $1$, respectively. Hence, there exist universal constants $c, C > 0$ and $w_0 \equiv w_0(\varepsilon) > 0$ such that, for all sufficiently large $n$,
\begin{alignat*}{2}
    \bbp (\Omega_{0 n})
    & \geq 1 - \exp(-c n / T^2),
    & \qquad
    \Omega_{0 n}
    & := \curls*{
    \min\nolimits_{i \in \dbraks{p_n}}
    \pars*{\norm{\tilde{\bfx}_i}^2
    \wedge \norm{\bfx_i}^2}
    \geq {n} / {2}}, \\
    \bbp (\Omega_{1 n})
    & \geq 1 - \exp (- c \varepsilon p_n),
    & \qquad
    \Omega_{1 n}
    & := \curls*{
    \abs*{\curls{i \in \dbraks{p_n} : \abs{\Delta_i} \geq 3 \delta n}}
    < \varepsilon p_n / 3}, \\
    \bbp (\Omega_{2 n})
    & \geq 1 - \exp(-w_0 p_n),
    & \qquad
    \Omega_{2 n}
    & := \curls[\Big]{\, \sup_{\bfv \in \mathbb{S}^{n - 1}}
    \min_{\abs{\caI} \geq p_n - \varepsilon p_n / 3}
    \norm{\bfP_\caI \hat{\bfX}_n \bfv}
    \leq C \delta \sqrt{p_n} \, }, \\
    \bbp (\Omega_{3 n})
    & \geq 1 - \exp(-w_0 p_n),
    & \qquad
    \Omega_{3 n}
    & := \curls[\Big]{\, \sup_{\bfv \in \mathbb{S}^{n - 1}}
    \min_{\abs{\caI} \geq p_n - \varepsilon p_n / 3}
    \norm{\bfP_\caI \tilde{\bfX}_n \bfv}
    \leq C \sqrt{p_n} \, }.
\end{alignat*}
Set $\Omega_n := \bigcap_{\ell = 0}^3 \Omega_{\ell n}$. Since $p_n \leq 2 \phi n$, we can fix a sufficiently small constant $w \equiv w(\varepsilon, \delta) > 0$ such that
\begin{equation*}
    \bbp (\Omega_n) \geq
    1 - \braks*{\exp(-c n / T^2)
    + \exp (- c \varepsilon p_n)
    + 2 \exp(-w_0 p_n)}
    \geq 1 - \exp(-w p_n).
\end{equation*}
for all sufficiently large $n$. Fix a realization $\omega \in \Omega_n$. By the definition of the event $\Omega_{1 n}$, there exists an index set $\caI_1 \equiv \caI_1(\omega) \subset \dbraks{p_n}$ such that $\abs{\caI_1} \geq p_n - \varepsilon p_n / 3$ and
\begin{equation}
    \abs{\Delta_i} < 3 \delta n,
    \qquad \forall\, i \in \caI_1.
    \label{eqn:set-property-1}
\end{equation}
Moreover, for every $\bfv \in \mathbb{S}^{n - 1}$, the definitions of $\Omega_{2 n}$ and $\Omega_{3 n}$ provide sets $\caI_{2, \bfv} \equiv \caI_{2, \bfv}(\omega)$ and $\caI_{3, \bfv} \equiv \caI_{3, \bfv}(\omega)$ contained in $\dbraks{p_n}$ such that $\abs{\caI_{2, \bfv}} \wedge \abs{\caI_{3, \bfv}} \geq p_n - \varepsilon p_n / 3$ and
\begin{equation}
    \norm{\bfP_{\caI_{2, \bfv}} \hat{\bfX}_n \bfv}
    \leq C \delta \sqrt{p_n},
    \qquad
    \norm{\bfP_{\caI_{3, \bfv}} \tilde{\bfX}_n \bfv}
    \leq C \sqrt{p_n}.
    \label{eqn:set-property-2}
\end{equation}
Now define
\begin{equation*}
    \caI_{\bfv}
    \equiv \caI_{\bfv} (\omega)
    = \caI_1 (\omega)
    \cap \caI_{2, \bfv} (\omega)
    \cap \caI_{3, \bfv} (\omega).
\end{equation*}
Then $\abs{\caI_{\bfv}} \geq p_n - \varepsilon p_n$. Since $\omega \in \Omega_{0 n}$, for each $i \in \dbraks{p_n}$,
\begin{align*}
    \abs*{\angles{\bfy_i, \bfv} - \angles{\tilde{\bfy}_i, \bfv}}^2
    & = \abs*{\frac{1}{\norm{\bfx_i}}
    \angles{\bfx_i, \bfv}
    - \frac{1}{\norm{\tilde{\bfx}_i}}
    \angles{\tilde{\bfx}_i, \bfv} }^2
    = \abs*{\frac{1}{\norm{\bfx_i}}
    \angles{\hat{\bfx}_i, \bfv}
    - \frac{\norm{\bfx_i} - \norm{\tilde{\bfx}_i}}{\norm{\tilde{\bfx}_i} \, \norm{\bfx_i}} \angles{\tilde{\bfx}_i, \bfv}}^2 \\
    & \leq \frac{2}{\norm{\bfx_i}^2}
    \, \abs{\angles{\hat{\bfx}_i, \bfv}}^2
    + {\frac{2 (\norm{\bfx_i} - \norm{\tilde{\bfx}_i})^2}
    {\norm{\bfx_i}^2 \norm{\tilde{\bfx}_i}^2} }
    \abs{\angles{\tilde{\bfx}_i, \bfv}}^2 \\
    & \leq \frac{4}{n} \, \abs{\angles{\hat{\bfx}_i, \bfv}}^2
    + \frac{4}{n^3} \, \abs{\Delta_i}^2
    \abs{\angles{\tilde{\bfx}_i, \bfv}}^2,
\end{align*}
where the last step uses
\begin{equation*}
    \norm{\bfx_i} - \norm{\tilde{\bfx}_i}
    = \frac{\norm{\bfx_i}^2 - \norm{\tilde{\bfx}_i}^2}{\norm{\bfx_i} + \norm{\tilde{\bfx}_i}}
    = \frac{\Delta_i}{\norm{\bfx_i} + \norm{\tilde{\bfx}_i}}.
\end{equation*}
By \eqref{eqn:set-property-1} and \eqref{eqn:set-property-2},
\begin{align*}
    \norm{\bfP_{\caI_{\bfv}} \pars{\bfY_n - \tilde{\bfY}_n} \bfv}^2
    & = \sum\nolimits_{i \in \caI_{\bfv}} \abs[\big]{\angles{\bfy_i, \bfv} - \angles{\tilde{\bfy}_i, \bfv}}^2
    \leq \frac{C}{n} \sum\nolimits_{i \in \caI_{\bfv}} \pars[\big]{\abs{\angles{\hat{\bfx}_i, \bfv}}^2
    + \delta^2 \abs{\angles{\tilde{\bfx}_i, \bfv}}^2} \\
    & = \frac{C}{n} \pars[\big]{\norm{\bfP_{\caI_{\bfv}} \hat{\bfX}_n \bfv}^2
    + \delta^2 \norm{\bfP_{\caI_{\bfv}} \tilde{\bfX}_n \bfv}^2} \\
    & \leq \frac{C}{n} \pars[\big]{\norm{\bfP_{\caI_{2, \bfv}} \hat{\bfX}_n \bfv}^2
    + \delta^2 \norm{\bfP_{\caI_{3, \bfv}} \tilde{\bfX}_n \bfv}^2}
    \leq C_\phi \delta^2.
\end{align*}
Thus, on $\Omega_n$, for each $\bfv \in \mathbb{S}^{n - 1}$,
\begin{equation*}
    \norm{\bfY_n \bfv}
    \geq \norm{\bfP_{\caI_{\bfv}} \bfY_n \bfv}
    \geq \norm{\bfP_{\caI_{\bfv}} \tilde{\bfY}_n \bfv} - C_\phi \delta
    \geq \min_{\abs{\caI} \geq p_n - \varepsilon p_n}
    \norm{\bfP_{\caI} \tilde{\bfY}_n \bfv} - C_\phi \delta.
\end{equation*}
Taking the infimum over $\bfv \in \mathbb{S}^{n - 1}$ concludes the proof.
\end{proof}

% *******************************************
\subsection{Completing the proof}
% *******************************************

% *******************************************
\begin{proof}[Proof of Theorem \ref{thm:tall}]
% *******************************************
By the MP law for sample correlation matrices \cite[Theorem 1.2]{jiangLimitingDistributionsEigenvalues2004}, the ESD of $\bfR_n = \bfY_n \bfY_n^\top$ converges weakly to $\tilde{F}_\phi$ almost surely under \eqref{eqn:basic-assumption}. Recall that $\lambda_n (\bfR_n) = s_n(\bfY_n)^2$. Consequently,
\begin{equation*}
    \limsup\nolimits_{n \to \infty} s_n(\bfY_n) \leq \sqrt{\phi} - 1,
    \qquad \text{ a.s. }
\end{equation*}
Thus, it suffices to prove the lower estimate
\begin{equation}
    \liminf\nolimits_{n \to \infty} s_n(\bfY_n) \geq \sqrt{\phi} - 1,
    \qquad \text{ a.s. }
\end{equation}
Fix a sufficiently small constant $\delta > 0$ such that $\delta \leq 1 / 4$ and $\delta^2 \leq 8 c_{\mathrm{\ref{eqn:lower-row-norms}}} C_{\mathrm{\ref{eqn:sum-subgaussian}}}$, and choose $T \equiv T(\delta) > 0$ so that \eqref{eqn:variance} holds. Next choose $\varepsilon \equiv \varepsilon(\delta) > 0$ sufficiently small that $\varepsilon \leq \delta$, $\varepsilon < (\phi - 1) / (2 \phi)$, and \eqref{eqn:choice-eps} holds. Propositions \ref{prop:gap-projection} and \ref{prop:gap-truncation} then give $w = w(\varepsilon, \delta) > 0$ such that, for all sufficiently large $n$,
\begin{equation*}
    \bbp \curls[\big]{s_n (\bfY_n) \geq s_n(\tilde{\bfY}_n) - C_\phi \delta}
    \geq 1 - \exp(-\varepsilon p_n) - \exp(-w p_n),
\end{equation*}
where $C_\phi > 0$ depends only on $\phi$. The Borel--Cantelli lemma then ensures that $s_n(\bfY_n) \geq s_n(\tilde{\bfY}_n)- C_\phi\delta$ eventually almost surely. The truncated entries $\tilde{x}_{i \mu}$ have positive variance by \eqref{eqn:variance}. Rescaling them to unit variance does not change $\tilde{\bfY}_n$. As the rescaled entries are centered and bounded, Jiang's result for smallest eigenvalue \cite[Theorem 1.1 (ii)]{jiangLimitingDistributionsEigenvalues2004} applies and gives $\lim_{n \to \infty} s_n(\tilde{\bfY}_n) = \sqrt{\phi} - 1$ almost surely. Consequently,
\begin{equation*}
    \liminf_{n \to \infty} s_n(\bfY_n)
    \geq (\sqrt{\phi} - 1) - C_\phi \delta,
    \qquad \text{ a.s. }
\end{equation*}
Applying this estimate along a sequence $\delta \downarrow 0$ completes the proof.
\end{proof}

% *******************************************
\section{Extension to Pearson correlation matrices}
\label{sec:pearson}
% *******************************************

% *******************************************
\subsection{Spectral reduction}
% *******************************************

To extend the main results to the Pearson correlation matrix $\bfR_n^{\circ}$ defined in \eqref{def:pearson-correlation}, we first establish the following analogue of Proposition \ref{prop:spectral-transfer}.

% *******************************************
\begin{proposition}
\label{prop:pearson-transfer}
\label{prop:pearson-lower-transfer}
% *******************************************
Fix an integer $k \geq 2$ and assume \eqref{eqn:basic-assumption} and \eqref{eqn:analytic-tail-envelope}. Let $\bbk_+ \subset (\lambda_+, \infty)$ and $\bbk_- \subset (0, \lambda_-)$ be compact intervals. For the lower count bounds, additionally assume $\phi \in (0, 1)$. Then there is a deterministic sequence $\eps_n \downarrow 0$ such that, with probability tending to one, the following bounds hold simultaneously,
\begin{subequations} \label{eqn:pearson-spectral-sandwich}
\begin{align}
    \fkn_+^{\bfL_{\caB}} (\fkx_\phi (z) + \eps_n)
    & \leq \fkn_+^{\bfR_n^{\circ}} (z)
    \leq \fkn_+^{\bfL_{\caB}} (\fkx_\phi (z) - \eps_n),
    \qquad
    \forall z \in \bbk_+.
    \label{eqn:pearson-upper-spectral-sandwich} \\    
    \fkn_-^{\bfL_{\caB}} (\fkx_\phi (z) - \eps_n)
    & \leq \fkn_-^{\bfR_n^{\circ}} (z)
    \leq \fkn_-^{\bfL_{\caB}} (\fkx_\phi (z) + \eps_n),
    \qquad \forall z \in \bbk_-.
    \label{eqn:pearson-lower-spectral-sandwich}
\end{align}    
\end{subequations}
\end{proposition}

Let $\bfH_n$ denote the orthogonal projection onto the subspace perpendicular to $\mathbf{1}_n$,
\begin{equation*}
    \bfH_n := \bfI - \bfke_n \bfke_n^\top,
    \qquad
    \bfke_n := n^{-1/2} \mathbf{1}_n.
\end{equation*}
Then, $\bfX_n^\circ = \bfX_n \bfH_n$. Define
\begin{equation*}
    \bfv_n := \bfY_n \bfke_n = (v_i)_{i=1}^{p_n},
    \qquad
    v_i :=  \angles{\bfy_i, \bfke_n},
\end{equation*}
and the diagonal comparison matrix
\begin{equation*}
    \bfDelta_n := \diag \pars{
    \Delta_1, \ldots, \Delta_{p_n}},
    \qquad
    \Delta_i :=  \frac{1}{\sqrt{1 - \abs{v_i}^2}}.    
\end{equation*}
Since $\Delta_i = 1 / \norm{\bfH_n \bfy_i} = \norm{\bfx_i} / \norm{\bfx_i^\circ}$, the Pearson correlation matrix can be written as
\begin{equation}
    \bfR_n^{\circ} = \bfDelta_n \bfY_n \bfH_n \bfY_n^\top \bfDelta_n.
    \label{eqn:pearson-diagonal-comparison}
\end{equation}

% *******************************************
\begin{lemma}
\label{lem:pearson-means}
% *******************************************
Under \eqref{eqn:basic-assumption},
\begin{equation}
    \norm{\bfv}_{\max}
    \prec n^{-1/2},
    \qquad
    \norm{\bfDelta_n - \bfI} \prec n^{-1}.
    \label{eqn:pearson-means}
\end{equation}
\end{lemma}

% *******************************************
\begin{proof}[Proof of Lemma \ref{lem:pearson-means}]
% *******************************************
By a standard truncation argument and Bernstein's inequality \eqref{eqn:scalar-Bernstein-domination}, as in Lemma \ref{lemma:norm-bounded-vectors}, we have $\min_{i \in \dbraks{p_n}} \norm{\bfx_i}^2 \geq n/4$ with overwhelming probability. To control the row sums, define
\begin{equation*}
    A_i := \sum\nolimits_{\mu = 1}^n
    x_{i \mu} \bbone\{\abs{x_{i \mu}} \leq \sqrt{n}\},
    \qquad
    B_i := \sum\nolimits_{\mu = 1}^n
    \bbone\{\abs{x_{i \mu}} > \sqrt{n}\}.
\end{equation*}
The assumptions $\bbe \xi = 0$ and $\bbe \xi^2 = 1$, together with Markov's inequality, give
\begin{equation*}
    \abs{\bbe A_i}
    = n \abs[\big]{\bbe \braks[\big]{\xi \bbone\{\abs{\xi} > \sqrt{n}\}}}
    \leq \sqrt{n},
    \qquad
    \operatorname{Var} (A_i) \leq n,
    \qquad
    \bbe B_i \vee \operatorname{Var} (B_i) \leq 1,
\end{equation*}
The centered summands of $A_i$ are bounded by $2\sqrt{n}$, while those of $B_i$ are bounded by one. Hence, applying Bernstein's inequality \eqref{eqn:scalar-Bernstein-domination} and taking a union bound over $i \in \dbraks{p_n}$, yield
\begin{equation*}
    \max\nolimits_{i \in \dbraks{p_n}} \abs{A_i} \prec \sqrt{n},
    \qquad
    \max\nolimits_{i \in \dbraks{p_n}} B_i \prec 1.
\end{equation*}
On the event $\min_{i \in \dbraks{p_n}} \norm{\bfx_i}^2 \geq n/4$, the Cauchy--Schwarz inequality yields
\begin{equation*}
    \abs{v_i}
    = \frac{1}{\sqrt{n} \norm{\bfx_i}}
    \abs*{\sum\nolimits_{\mu = 1}^n x_{i \mu}}
    \leq \frac{1}{\sqrt{n} \norm{\bfx_i}} \abs{A_i}
    + \sqrt{{B_i} / {n}}
    \leq 2 \abs{A_i} / n
    + \sqrt{{B_i} / {n}}.
\end{equation*}
Combining these estimates proves the first bound in \eqref{eqn:pearson-means}. The second estimate follows immediately from the expansion $(1-t)^{-1/2} - 1 = O(t)$ as $t \to 0$.
\end{proof}

Lemma \ref{lem:pearson-means} controls the diagonal renormalization $\bfDelta_n$ in \eqref{eqn:pearson-diagonal-comparison}. However, centering also introduces the rank-one term in $\bfY_n \bfH_n \bfY_n^\top = \bfY_n \bfY_n^\top - \bfv_n \bfv_n^\top$, and the bound on $\norm{\bfv}_{\max}$ alone does not show that this term is small in operator norm. We therefore need to control its effect through the resolvent $\bfG_{\caT}$ of the typical block. The next two lemmas provide the additional estimates needed to complement Proposition \ref{prop:unified-centered-compressions}.

% *******************************************
\begin{lemma}
\label{lem:pearson-mean-direction}
% *******************************************
Suppose that \eqref{eqn:basic-assumption} and \eqref{eqn:analytic-tail-envelope} hold for a fixed $k \geq 2$.  Let $\bbk_+ \subset (\lambda_+, \infty)$ be a compact interval. When $\phi \in (0, 1)$, also let $\bbk_- \subset (0, \lambda_-)$ be a compact interval. Then, uniformly over all realizations of $\bar{\caH}_n$ in $\tilde{\Omega}_n$, the following holds under $\bbp^{\bar{\caH}}$,
\begin{equation}
    \sup\nolimits_{z \in \bbk_{\pm}} \!
    \curls[\big]{\norm{\fkS_{\caB} \bfG_{\caT}(z) \bfke_n}
    + \norm{\fkW_{\caB} \bfG_{\caT}(z) \bfke_n}} 
    \prec n^{-\tau}.
    \label{eqn:pearson-mean-direction}
\end{equation}
\end{lemma}

% *******************************************
\begin{proof}[Proof of Lemma \ref{lem:pearson-mean-direction}]
% *******************************************
Fix a realization of $\bar{\caH}_n$ in $\tilde{\Omega}_n$. Unless further conditioning is specified, all stochastic estimates below hold under $\bbp^{\bar{\caH}}$, uniformly over such realizations. As demonstrated in the proof of Proposition \ref{prop:unified-centered-compressions}, a standard $\varepsilon$-net argument reduces the proof to estimates for each fixed $z \in \bbk$. We henceforth fix $z$ and suppress it from the notation. 

We first show that
\begin{equation}
    \norm{\bfG_{\caT} \bfke_n}_{\max}
    = \max\nolimits_{\mu \in \dbraks{n}} 
    \abs{\angles{\bfe_\mu, \bfG_{\caT} \bfke_n}}
    \prec n^{-\tau}.
    \label{eqn:resolvent-averaging}
\end{equation}
Recall the i.i.d. proxy $\tilde{\bfY}_{\caT}$ from \eqref{def:iid-proxy}, and write $\tilde{\bfu}_\mu := \tilde{\bfY}_{\caT} \bfe_\mu$ for its $\mu$-th column. We denote the leave-one-out matrix obtained by removing this column, the resulting companion matrix, and the corresponding resolvent by
\begin{equation*}
    \tilde{\bfY}_{\caT}^{(\mu)}
    := (\tilde{y}_{i \rho})_{i \in \caT\!, \, \rho \in \dbraks{n} \backslash \{\mu\}},
    \qquad
    \tilde{\bfQ}_\caT^{(\mu)} 
    := \tilde{\bfY}_{\caT}^{(\mu) \top} \tilde{\bfY}_{\caT}^{(\mu)},
    \qquad
    \tilde{\bfG}_{\caT}^{(\mu)} (z)
    := ({\tilde{\bfQ}_\caT^{(\mu)} - z \bfI})^{-1}.
\end{equation*}
Let $\bfke_n^{(\mu)}$ denote $\bfke_n$ with its $\mu$-th coordinate deleted. The Schur complement formula then gives
\begin{equation}
    \angles{\bfe_\mu, \tilde{\bfG}_{\caT} \bfke_n}
    = \tilde{G}_{\mu \mu}
    \braks[\big]{n^{-1/2}-
    \angles[\big]{\tilde{\bfu}_\mu, \tilde{\bfY}_{\caT}^{(\mu)}
    \tilde{\bfG}_{\caT}^{(\mu)} \bfke_{n}^{(\mu)} }}
    =: \tilde{G}_{\mu \mu}
    \braks[\big]{n^{-1/2}-
    \angles{\tilde{\bfu}_\mu, \bfb^{(\mu)}}}.
    \label{eqn:leave-one-out}
\end{equation}

The argument leading to \eqref{eqn:alpha3-proxy-edge-rigidity} also applies to the leave-one-out matrix, with an $O(n^{-1})$ change in the deterministic edges. A union bound over $\mu$ gives
\begin{equation}
    \max_{\mu \in \dbraks{n}} \,
    \abs[\big]{\lambda_1 (\tilde{\bfQ}_{\caT}^{(\mu)})
    - \lambda_+ (\varphi_n)}
    \prec n^{-2 \tau}, 
    \qquad
    \max_{\mu \in \dbraks{n}} \,
    \abs[\big]{\lambda_{\abs{\caT}} (\tilde{\bfQ}_{\caT}^{(\mu)})
    - \lambda_- (\varphi_n)}
    \prec n^{-2 \tau},
    \label{eqn:pearson-lower-gap-minor}
\end{equation}
where the second estimate requires $0 < \phi < 1$. Consequently, for $z \in \bbk_\pm$, the vector $\bfb^{(\mu)}$ introduced in \eqref{eqn:leave-one-out} satisfies $\norm{\bfb^{(\mu)}} \leq C$ with overwhelming probability for a sufficiently large constant $C > 0$. On the other hand, conditional on $\tilde{\bfY}_{\caT}^{(\mu)}$, the entries of $\tilde{\bfu}_\mu = (\tilde{y}_{i \mu})_{i \in \caT}$ remain i.i.d., centered, and bounded by $O (n^{-\tau})$. Therefore, by Bernstein's inequality \eqref{eqn:scalar-Bernstein-domination},
\begin{equation*}
    \abs{\angles{\tilde{\bfu}_\mu,\bfb^{(\mu)}}}
    \prec n^{-\tau} + n^{-1/2} \leq 2 n^{-\tau}.
\end{equation*}
The rigidity estimate \eqref{eqn:alpha3-proxy-edge-rigidity} for the full matrix also gives $\abs{\tilde{G}_{\mu \mu}} \prec 1$ on $\bbk_\pm$. Substituting these bounds into \eqref{eqn:leave-one-out} and taking a union bound yields
\begin{equation*}
    \norm{\tilde{\bfG}_{\caT} \bfke_n}_{\max}
    = \max\nolimits_{\mu \in \dbraks{n}} 
    \abs{\angles{\bfe_\mu, \tilde{\bfG}_{\caT} \bfke_n}}
    \prec n^{-\tau}.
\end{equation*}
The comparison \eqref{eqn:unified-core-comparison} and the resolvent identity transfer this bound to $\bfG_{\caT}$, proving \eqref{eqn:resolvent-averaging}.

\shortpara{The $\fkS$-$\bfke$ estimate}. We next prove the bound involving $\fkS_{\caB}$ in \eqref{eqn:pearson-mean-direction}. Conditional on $\bar{\caF}_n$, the resolvent $\bfG_{\caT}$ is fixed, while the occupied column set $\caJ$ is uniformly distributed among subsets of $\dbraks{n}$ of size $\abs{\bar{\caJ}}$. Let $(\eta_\mu)_{\mu=1}^n$ be i.i.d. $\mathrm{Bern} (\abs{\bar{\caJ}}/n)$ variables, independent of all other randomness. Set $\caJ_\eta := \{\mu\in\dbraks{n}:\eta_\mu=1\}$. Then
\begin{equation}
    \caJ \,|\, \bar{\caF}_n
    \overset{\rmd}{=} 
    \caJ_\eta \,\big|\,
    \abs{\caJ_\eta} = \abs{\bar{\caJ}},
    \qquad
    \bbp \curls[\big]{\abs{\caJ_\eta} = \abs{\bar{\caJ}}}
    \geq ({n+1})^{-1}.
    \label{eqn:property-caJ-eta}
\end{equation}
The probability bound follows because $\abs{\bar{\caJ}}$ is a mode of $\abs{\caJ_\eta} \sim \mathrm{Binom} (n, \abs{\bar{\caJ}}/n)$. Before conditioning on the cardinality of $\caJ_\eta$, we apply Bernstein's inequality \eqref{eqn:scalar-Bernstein-domination} to the Bernoulli sum, conditional on $\bar{\caF}_n$. Averaging this estimate and using \eqref{eqn:resolvent-averaging} and the bound $\norm{\bfG_{\caT}} \prec 1$ from \eqref{eqn:unified-core-trace}, we obtain
\begin{align*}
    \norm{\bfP_{\caJ_\eta} \bfG_{\caT} \bfke_n}^2
    & = \sum\nolimits_{\mu=1}^n
    \abs{\angles{\bfe_\mu, \bfG_{\caT} \bfke_n}}^2
    \pars[\big]{\eta_\mu - \abs{\bar{\caJ}}/n}
    + (\abs{\bar{\caJ}}/n) \, \norm{\bfG_{\caT}\bfke_n}^2 \\
    & \prec 
    \norm{\bfG_{\caT}\bfke_n}_{\max}^2
    + (\abs{\bar{\caJ}}/n) \, \norm{\bfG_{\caT}\bfke_n}^2    
    \prec n^{-2\tau} + n^{-\sigma/2}
    \leq 2 n^{-2\tau},
\end{align*}
where we used $\abs{\bar{\caJ}}/n\leq n^{-\sigma/2}$ from Lemma \ref{lem:unified-graph} and $\tau \leq \sigma/4$ from \eqref{eqn:unified-truncation}. Conditioning on $\abs{\caJ_\eta} = \abs{\bar{\caJ}}$ multiplies failure probabilities by at most $n+1$ by \eqref{eqn:property-caJ-eta}, and thus it preserves stochastic domination. The equality in distribution in \eqref{eqn:property-caJ-eta} thus transfers the bound to $\caJ$. Using the norm bound of $\fkS_{\caB}$ in \eqref{eqn:unified-freezing-bounds}, we obtain
\begin{equation*}
    \norm{\fkS_{\caB}\bfG_{\caT}\bfke_n}
    = \norm[\big]{\fkS_{\caB} \bfP_{\caJ}^\top \bfP_{\caJ} \bfG_{\caT}\bfke_n}
    \leq k\norm{\bfP_{\caJ}\bfG_{\caT}\bfke_n}
    \prec n^{-\tau}.
\end{equation*}

\shortpara{The $\fkW$-$\bfke$ estimate}. Conditional on $\caF_n$, the summands in the decomposition
\begin{equation*}
    \fkW_{\caB} \bfG_{\caT} \bfke_n
    = \sum\nolimits_{i \in \caB} \sum\nolimits_{\mu = 1}^n \fkX_{i \mu},
    \qquad
    \fkX_{i \mu} := \fkw_{i \mu}
    \angles{\bfe_\mu, \bfG_{\caT} \bfke_n} \bfe_i
\end{equation*}
are independent and centered. By \eqref{eqn:unified-diffuse-entries} and $\abs{\caB}/n \leq n^{-\sigma/2}$ on $\tilde{\Omega}_n$, we have
\begin{align*}
    \norm{\fkX_{i \mu}}
    & \leq C n^{-\tau} \norm{\bfG_{\caT} \bfke_n}_{\max}, \\
    \norm*{\sum\nolimits_{i \in \caB}
    \sum\nolimits_{\mu = 1}^n
    \bbe^{\caF} [\fkX_{i \mu} \fkX_{i \mu}^{\top}] }
    & \leq C n^{-1} \norm{\bfG_{\caT} \bfke_n}^2, \\
    \norm*{\sum\nolimits_{i \in \caB}
    \sum\nolimits_{\mu = 1}^n
    \bbe^{\caF} [\fkX_{i \mu}^{\top} \fkX_{i \mu}] }
    & \leq C n^{-\sigma/2} \norm{\bfG_{\caT} \bfke_n}^2.
\end{align*}
We apply Tropp's matrix Bernstein inequality \eqref{eqn:matrix-Bernstein-domination} under $\bbp^{\caF}$ and then average the resulting conditional estimate. Together with \eqref{eqn:resolvent-averaging} and \eqref{eqn:unified-core-trace}, this gives
\begin{equation*}
    \norm{\fkW_{\caB}\bfG_{\caT}\bfke_n}
    \prec n^{-\tau} \norm{\bfG_{\caT}\bfke_n}_{\max}
    + \sqrt{\abs{\caB}/n} \norm{\bfG_{\caT}\bfke_n}
    \prec n^{-2\tau} + n^{-\sigma/4}
    \leq n^{-\tau},
\end{equation*}
which completes the proof of \eqref{eqn:pearson-mean-direction}.
\end{proof}

% *******************************************
\begin{lemma}
\label{lem:pearson-mean-scalar}
% *******************************************
Suppose that \eqref{eqn:basic-assumption} and \eqref{eqn:analytic-tail-envelope} hold for a fixed $k \geq 2$.  Let $\bbk_+ \subset (\lambda_+, \infty)$ be a compact interval. When $\phi \in (0, 1)$, also let $\bbk_- \subset (0, \lambda_-)$ be a compact interval. Then, uniformly over all realizations of $\bar{\caH}_n$ in $\tilde{\Omega}_n$, the following holds under $\bbp^{\bar{\caH}}$,
\begin{equation}
    \sup\nolimits_{z \in \bbk_\pm} \! \abs[\big]{
    \angles{\bfke_n, \bfG_{\caT} (z) \bfke_n} - \fkm_{\varphi_n}\!(z)}
    \prec n^{-\tau},
    \label{eqn:pearson-mean-scalar}
\end{equation}
\end{lemma}

% *******************************************
\begin{proof}[Proof of Lemma \ref{lem:pearson-mean-scalar}]
% *******************************************
Fix a realization of $\bar{\caH}_n$ in $\tilde{\Omega}_n$. All stochastic estimates below hold under $\bbp^{\bar{\caH}}$, uniformly over such realizations and $E \in \bbk_\pm$. We first fix $E \in \bbk_\pm$ and work with $z := E + \rmi \eta$, where $\eta = n^{-1 + \tau}$. For simplicity, we suppress the spectral parameter $z$.

We again apply the leave-one-out argument to the i.i.d. proxy $\tilde{\bfY}_{\caT}$, now removing one row at a time. For $i \in \caT$, let $\tilde{\bfy}_i \in \bbr^n$ denote its $i$-th row. We denote the matrix with row $i$ removed, its companion matrix, and the corresponding resolvent by
\begin{equation*}
    \tilde{\bfY}_{\caT}^{[i]}
    := (\tilde{y}_{j \rho})_{j \in \caT \backslash \{i\} \!, \, \rho \in \dbraks{n}},
    \qquad
    \tilde{\bfQ}_\caT^{[i]} 
    := \tilde{\bfY}_{\caT}^{[i] \top} \tilde{\bfY}_{\caT}^{[i]},
    \qquad
    \tilde{\bfG}_{\caT}^{[i]} (z)
    := \pars{ \tilde{\bfQ}_\caT^{[i]} - z \bfI}^{-1}.
\end{equation*}
Define the quadratic forms
\begin{equation*}
    \tilde{q} := \angles{\bfke_n, \tilde{\bfG}_{\caT} \bfke_n},
    \qquad
    \tilde{q}_i := \angles{\bfke_n, \tilde{\bfG}_{\caT}^{[i]} \bfke_n}.
\end{equation*}
Note that
\begin{equation*}
    \tilde{\bfQ}_\caT^{[i]} = \tilde{\bfQ}_{\caT} - \tilde{\bfy}_i \tilde{\bfy}_i^{\top},
    \qquad
    \bfI + z \tilde{\bfG}_{\caT} (z) 
    = \tilde{\bfQ}_{\caT} \tilde{\bfG}_{\caT} (z) 
    = \sum\nolimits_{i \in \caT} 
    \tilde{\bfy}_i \tilde{\bfy}_i^{\top}
    \tilde{\bfG}_{\caT} (z).
\end{equation*}
Consequently, the Sherman--Morrison formula gives the identities
\begin{equation}
    \tilde{q} - \tilde{q}_i = 
    -\frac{\angles{\tilde{\bfy}_i, \tilde{\bfG}_{\caT}^{[i]} \bfke_n}^2}
    {1 + \angles{\tilde{\bfy}_i, \tilde{\bfG}_{\caT}^{[i]} \tilde{\bfy}_i}},
    \qquad
    1 + z \tilde{q} = \sum\nolimits_{i \in \caT} \frac{
    \angles{\bfke_n, \tilde{\bfy}_i}
    \angles{\tilde{\bfy}_i, \tilde{\bfG}_{\caT}^{[i]} \bfke_n}}
    {1 + \angles{\tilde{\bfy}_i, \tilde{\bfG}_{\caT}^{[i]} \tilde{\bfy}_i}}.
    \label{eqn:pearson-scalar-row-identities}
\end{equation}
Conditional on the other rows, the entries of $\tilde{\bfy}_i$ remain i.i.d., with mean zero, variance $1/n$, and absolute value at most $C n^{-\tau}$. The standard large deviation bounds for random variables with bounded support (see, e.g., \cite[Lemma A.5]{dingNecessarySufficientCondition2018}) therefore give
\begin{subequations} \label{eqn:pearson-scalar-linear-deviation}
\begin{align}
    \abs{\angles{\tilde{\bfy}_i, \bfke_n}}
    & \prec n^{-\tau} \norm{\bfke_n}_{\max}
    + n^{-1/2} \norm{\bfke_n}
    \leq C n^{-1/2},
    \label{eqn:pearson-scalar-linear-deviation-a} \\
    \abs{\angles{\tilde{\bfy}_i, \tilde{\bfG}_{\caT}^{[i]} \bfke_n}}
    & \prec n^{-\tau} \norm{\tilde{\bfG}_{\caT}^{[i]} \bfke_n}_{\max}
    + n^{-1/2} \norm{\tilde{\bfG}_{\caT}^{[i]} \bfke_n} 
    \leq C n^{-\tau} \norm{\tilde{\bfG}_{\caT}^{[i]}},
    \label{eqn:pearson-scalar-linear-deviation-b} \\
    \abs[\big]{\angles{\tilde{\bfy}_i, \tilde{\bfG}_{\caT}^{[i]} \tilde{\bfy}_i}
    - n^{-1} \operatorname{Tr} \tilde{\bfG}_{\caT}^{[i]}} 
    & \prec n^{-\tau} \norm{\tilde{\bfG}_{\caT}^{[i]}}_{\max}
    + n^{-1} \norm{\tilde{\bfG}_{\caT}^{[i]}}_{\mathrm{F}}
    \leq C n^{-\tau} \norm{\tilde{\bfG}_{\caT}^{[i]}},
    \label{eqn:pearson-scalar-linear-deviation-c}
\end{align}
\end{subequations}
where we also used $\tau < 1/2$.

As explained around \eqref{eqn:pearson-lower-gap-minor}, the rigidity estimate \eqref{eqn:alpha3-proxy-edge-rigidity} and the entrywise law \eqref{eqn:alpha3-proxy-entrywise-local-law} apply to the minors with the same error rates. Together with $\Im \fkm_{\varphi_n} = O (\eta)$ from \eqref{eqn:imaginary-fkm-bound}, these estimates imply that, for each $i \in \caT$, the following holds with overwhelming probability and depends only on the rows other than $\tilde{\bfy}_i$,
\begin{equation}
    \norm{\tilde{\bfG}_{\caT}^{[i]}} \leq C,
    \qquad
    \abs[\big]{n^{-1} \operatorname{Tr} \tilde{\bfG}_{\caT}^{[i]} 
    - \fkm_{\varphi_n}}
    \leq n^{-\tau/2}.
    \label{eqn:event-leave-i-good}
\end{equation}
Since $\varphi_n \to \phi$ uniformly on $\tilde{\Omega}_n$, the assumptions on $\bbk_\pm$ imply
\begin{equation}
    \abs{\fkm_{\varphi_n} (E + \rmi \eta)} 
    + \abs{1 + \fkm_{\varphi_n} (E + \rmi \eta)}^{-1} \leq C.
    \label{eqn:stability-Stieltjes}
\end{equation}

We establish concentration of the quadratic form $\tilde{q} = \angles{\bfke_n, \tilde{\bfG}_{\caT} \bfke_n}$ using a moment version of the Efron--Stein inequality. Let $\bbe_i$ denote expectation over $\tilde{\bfy}_i$, conditional on the other rows. By \cite[Theorem 2]{boucheronMomentInequalitiesFunctions2005}, for each fixed $\ell \in \bbn$,
\begin{equation}
    \bbe \abs{\tilde{q} - \bbe \tilde{q}}^{2 \ell}
    \leq C_\ell \bbe \abs[\Big]{\sum\nolimits_{i \in \caT}
    \pars*{\abs{\tilde{q} - \tilde{q}_i}^2 
    + \bbe_i \abs{\tilde{q} - \tilde{q}_i}^2}}^\ell.
    \label{eqn:pearson-scalar-variance}
\end{equation}
It therefore remains to control the two sums on the right-hand side of \eqref{eqn:pearson-scalar-variance}.

First consider the sum without conditional expectations. For each $i \in \caT$, restrict the other rows to the event \eqref{eqn:event-leave-i-good}. The concentration bound \eqref{eqn:pearson-scalar-linear-deviation-c} and the deterministic bound \eqref{eqn:stability-Stieltjes} then imply that the denominator in \eqref{eqn:pearson-scalar-row-identities} is bounded below by a constant $c > 0$ with overwhelming conditional probability. Combining this with \eqref{eqn:pearson-scalar-linear-deviation-b} yields
\begin{equation}
    \sum\nolimits_{i \in \caT} \abs{\tilde{q} - \tilde{q}_i}^2
    \prec n^{-2\tau} \sum\nolimits_{i \in \caT}
    \abs{\angles{\tilde{\bfy}_i, \tilde{\bfG}_{\caT}^{[i]} \bfke_n}}^2
    \prec n^{-2\tau}
    \norm{\tilde{\bfY}_{\caT} \tilde{\bfG}_{\caT} \bfke_n}^2
    \prec n^{-2 \tau},
    \label{eqn:bound-without-condition}
\end{equation}
where the second step follows from
\begin{equation*}
    \abs[\big]{\angles{\tilde{\bfy}_i, \tilde{\bfG}_{\caT}^{[i]} \bfke_n}}
    = \abs[\big]{1 + \angles{\tilde{\bfy}_i, \tilde{\bfG}_{\caT}^{[i]} \tilde{\bfy}_i} }
    \abs[\big]{\angles{\tilde{\bfy}_i, \tilde{\bfG}_{\caT} \bfke_n}}
    \prec \abs[\big]{\angles{\tilde{\bfy}_i, \tilde{\bfG}_{\caT} \bfke_n}}.
\end{equation*}

For the sum with conditional expectations, note that on the event \eqref{eqn:event-leave-i-good},
\begin{equation}
    \bbe_i \abs[\big]{\angles{\tilde{\bfy}_i, \tilde{\bfG}_{\caT}^{[i]} \bfke_n}}^2
    = n^{-1} \norm{\tilde{\bfG}_{\caT}^{[i]} \bfke_n}^2
    \leq C n^{-1}.
    \label{eqn:pearson-scalar-second-moment}
\end{equation}
As above, the denominator in \eqref{eqn:pearson-scalar-row-identities} is bounded below by $c > 0$ with overwhelming probability over $\tilde{\bfy}_i$. Together with \eqref{eqn:pearson-scalar-linear-deviation-b} and \eqref{eqn:pearson-scalar-second-moment}, this gives
\begin{equation}
    \max\nolimits_{i \in \caT} \bbe_i \abs{\tilde{q} - \tilde{q}_i}
    \prec n^{-1},
    \qquad
    \max\nolimits_{i \in \caT} \bbe_i \abs{\tilde{q} - \tilde{q}_i}^2
    \prec n^{-1 - 2 \tau}.
    \label{eqn:pearson-scalar-increment-moments}
\end{equation}
Here the contribution of the complementary event to the conditional moments is negligible by the deterministic bound $\abs{\tilde{q} - \tilde{q}_i} \leq 2/\eta \leq 2n$.

Applying \eqref{eqn:bound-without-condition} and \eqref{eqn:pearson-scalar-increment-moments} to \eqref{eqn:pearson-scalar-variance}, and again using the deterministic bound $\abs{\tilde{q} - \tilde{q}_i} \leq 2/\eta$ to control the exceptional contributions, we obtain the moment bound $\bbe \abs{\tilde{q} - \bbe \tilde{q}}^{2 \ell} \prec n^{-2 \ell \tau}$ for every fixed $\ell \in \bbn$. Markov's inequality then yields
\begin{equation}
    \abs{\tilde{q} - \bbe \tilde{q}} \prec n^{-\tau}.
    \label{eqn:concentration-q-exp}
\end{equation}

It remains to identify the mean $\bbe \tilde{q}$. By the second identity in \eqref{eqn:pearson-scalar-row-identities},
\begin{equation}
    1 + z \tilde{q} = \frac{1}{1 + \fkm_{\varphi_n}}
    \sum\nolimits_{i \in \caT} 
    \angles{\bfke_n, \tilde{\bfy}_i}
    \angles{\tilde{\bfy}_i, \tilde{\bfG}_{\caT}^{[i]} \bfke_n}
    + \Upsilon,
    \label{eqn:self-consistent-eq-approx}
\end{equation}
where the remainder is
\begin{equation*}
    \Upsilon := \sum\nolimits_{i \in \caT} 
    \pars[\bigg]{\frac{1}{1 + \angles{\tilde{\bfy}_i, \tilde{\bfG}_{\caT}^{[i]} \tilde{\bfy}_i}}
    - \frac{1}{1 + \fkm_{\varphi_n}}}
    \angles{\bfke_n, \tilde{\bfy}_i}
    \angles{\tilde{\bfy}_i, \tilde{\bfG}_{\caT}^{[i]} \bfke_n}.
\end{equation*}
We again use \eqref{eqn:pearson-scalar-linear-deviation-c}, the entrywise law \eqref{eqn:alpha3-proxy-entrywise-local-law} applied to the minors, and \eqref{eqn:stability-Stieltjes} to control the replacement of the denominators. Cauchy--Schwarz then gives
\begin{align*}
    \abs{\Upsilon} 
    & \prec n^{-\tau} \sum\nolimits_{i \in \caT} 
    \abs[\big]{\angles{\bfke_n, \tilde{\bfy}_i}
    \angles{\tilde{\bfy}_i, \tilde{\bfG}_{\caT}^{[i]} \bfke_n}} \\
    & \prec n^{-\tau} \pars[\Big]{\sum\nolimits_{i \in \caT}
    \abs[\big]{\angles{\bfke_n, \tilde{\bfy}_i}}^2}^{1/2}
    \pars[\Big]{\sum\nolimits_{i \in \caT} 
    \abs[\big]{\angles{\tilde{\bfy}_i, \tilde{\bfG}_{\caT}^{[i]} \bfke_n}}^2}^{1/2} \\
    & \prec n^{-\tau} \norm{\tilde{\bfY}_{\caT} \bfke_n}
    \norm{\tilde{\bfY}_{\caT} \tilde{\bfG}_{\caT} \bfke_n}
    \prec n^{-\tau},
\end{align*}
where the second sum in the second line is controlled as in \eqref{eqn:bound-without-condition}. The deterministic bounds $\abs{\tilde{q}} \leq \eta^{-1}$, $\norm{\tilde{\bfG}_{\caT}^{[i]}} \leq \eta^{-1}$, and $\abs{\tilde{y}_{i \mu}} = O (n^{-\tau})$ ensure $\abs{\Upsilon} \leq n^{C}$ for some constant $C > 0$. Consequently, the preceding estimate gives $\bbe \abs{\Upsilon} \prec n^{-\tau}$. Taking expectations in \eqref{eqn:self-consistent-eq-approx} and using independence and the variance normalization, we obtain
\begin{equation*}
    1 + z \bbe \tilde{q}
    = \frac{1}{1 + \fkm_{\varphi_n}}
    \pars[\Big]{\frac{1}{n} \sum\nolimits_{i \in \caT} \bbe \tilde{q}_i}
    + O_{\prec} (n^{-\tau}).
\end{equation*}
On the other hand, averaging \eqref{eqn:pearson-scalar-increment-moments} over the other rows gives $\max_{i \in \caT} \bbe \abs{\tilde{q}_i - \tilde{q}} = O_{\prec} (n^{-1})$. Since $\abs{\caT}/n = \varphi_n$, the preceding identity and \eqref{eqn:companion-MP-equation} yield, for every fixed $\eps > 0$,
\begin{equation*}
    - \bbe \tilde{q} / \fkm_{\varphi_n}
    = \pars*{z - \frac{\varphi_n}{1 + \fkm_{\varphi_n}}} \bbe \tilde{q}
    = -1 + O (n^{-\tau + \eps}).
\end{equation*}
Since $\abs{\fkm_{\varphi_n}} = O (1)$ by \eqref{eqn:stability-Stieltjes}, this proves $\abs{\bbe \tilde{q} - \fkm_{\varphi_n}} \prec n^{-\tau}$. Combining this with \eqref{eqn:concentration-q-exp}, we obtain the desired estimate $\abs{\tilde{q} - \fkm_{\varphi_n}} \prec n^{-\tau}$ at $z = E + \rmi \eta$. The extension to the real line, the supremum over $E \in \bbk_\pm$, and the transfer to the original resolvent using \eqref{eqn:unified-core-comparison} follow as in the proof of Proposition \ref{prop:unified-core}. We omit these details. This proves \eqref{eqn:pearson-mean-scalar}.
\end{proof}

We now use these two lemmas to prove Proposition \ref{prop:pearson-transfer}.

% *******************************************
\begin{proof}[Proof of Proposition \ref{prop:pearson-transfer}]
% *******************************************
We begin by showing that the diagonal renormalization $\bfDelta_n$ in \eqref{eqn:pearson-diagonal-comparison} has a negligible effect on the spectrum. Proposition \ref{prop:unified-core} ensures that $\norm{\bfY_{\caT}} = O_{\bbp} (1)$. For the atypical block, \eqref{eqn:unified-resolvent-detailed-Gram} gives $\norm{\bfY_{\caB} \bfY_{\caB}^\top - \bfL_{\caB}} \prec n^{-\tau}$ under $\bbp^{\bar{\caH}}$, uniformly over realizations of $\bar{\caH}_n$ in $\tilde{\Omega}_n$. Since $\norm{\bfL_{\caB}} = O (1)$ on $\tilde{\Omega}_n$ by \eqref{eqn:Lb-trivial-bound} and $\bbp (\tilde{\Omega}_n) \to 1$, we also have $\norm{\bfY_{\caB}} = O_{\bbp} (1)$. Consequently,
\begin{equation}
    \norm{\bfY_n} = O_{\bbp}(1),
    \qquad
    \norm{\bfY_n \bfH_n \bfY_n^\top} = O_{\bbp}(1),
    \qquad
    \norm{\bfR_n^{\circ} - \bfY_n \bfH_n \bfY_n^\top} = o_{\bbp}(1),
    \label{eqn:pearson-small-renormalization}
\end{equation}
where the last estimate follows from \eqref{eqn:pearson-diagonal-comparison} and $\norm{\bfDelta_n - \bfI} \prec n^{-1}$ in Lemma \ref{lem:pearson-means}.

By Weyl's inequality and \eqref{eqn:pearson-small-renormalization}, it therefore suffices to prove \eqref{eqn:pearson-spectral-sandwich} for $\bfY_n \bfH_n \bfY_n^\top$. Up to a simultaneous permutation of rows and columns,
\begin{equation}
    \bfY_n\bfH_n\bfY_n^\top-z\bfI
    = \begin{bmatrix}
        \bfY_{\caT}\bfH_n\bfY_{\caT}^{\top}-z\bfI
        & \bfY_{\caT}\bfH_n\bfY_{\caB}^{\top} \\
        \bfY_{\caB}\bfH_n\bfY_{\caT}^{\top}
        & \bfY_{\caB}\bfH_n\bfY_{\caB}^{\top}-z\bfI
    \end{bmatrix}.
    \label{eqn:pearson-blocks}
\end{equation}
We next establish the invertibility of the upper-left block in \eqref{eqn:pearson-blocks} for $z \in \bbk_\pm$.

Since $\lambda_1 (\bfH_n \bfQ_{\caT} \bfH_n) \leq \lambda_1 (\bfQ_{\caT})$, the upper edge bound \eqref{eqn:unified-core-trace} ensures that, with probability tending to one, the upper-left block is negative definite for every $z \in \bbk_+$. For the lower interval $\bbk_- \subset (0, \lambda_-)$, assume $\phi \in (0, 1)$ and set $\bfv_{\caT} := \bfY_{\caT} \bfke_n$. Then
\begin{equation*}
    \bfY_{\caT} \bfH_n \bfY_{\caT}^{\top} = \bfR_{\caT} - \bfv_{\caT} \bfv_{\caT}^\top.
\end{equation*}
Since $\phi < 1$, the explicit formula for the companion MP transform gives
\begin{equation*}
    \fkm_\phi (z)
    = \frac{\phi - 1 - z - \sqrt{(\lambda_- - z)(\lambda_+ - z)}}{2z}
    < 0,
    \qquad
    z \in \bbk_-.
\end{equation*}
By Lemma \ref{lem:pearson-mean-scalar} and the uniform convergence $\fkm_{\varphi_n} \to \fkm_\phi$ on $\bbk_-$ over realizations in $\tilde{\Omega}_n$, there exists some constant $c \equiv c (\bbk_-) \in (0, 1)$ such that $z \angles{\bfke_n, \bfG_{\caT} (z) \bfke_n} \leq -c$ uniformly on $\bbk_-$ with probability tending to one. On the intersection of this event with the lower gap event in \eqref{eqn:unified-core-trace}, the matrix determinant lemma gives, simultaneously for $z \in \bbk_-$ and $t \in [0, 1]$,
\begin{align}
\begin{split}
    \frac{\det (\bfR_{\caT} - z \bfI - t \bfv_{\caT} \bfv_{\caT}^{\top})}
    {\det (\bfR_{\caT} - z \bfI)}
    & = 1 - t \angles{\bfv_{\caT}, (\bfR_{\caT} - z \bfI)^{-1} \bfv_{\caT}} \\
    & = 1 - t - t z \angles{\bfke_n, \bfG_{\caT} (z) \bfke_n} 
    \geq 1 - (1-c) t \geq c.
\end{split} \label{eqn:pearson-centered-core-positivity}
\end{align}
Here we used $\bfY_{\caT}^{\top} (\bfR_{\caT} - z \bfI)^{-1} \bfY_{\caT} = \bfI + z \bfG_{\caT} (z)$. On the aforementioned event, $\bfR_{\caT} - z \bfI$ is positive definite for every $z \in \bbk_-$. Hence \eqref{eqn:pearson-centered-core-positivity} ensures that no eigenvalue of $\bfR_{\caT} - z \bfI - t \bfv_{\caT} \bfv_{\caT}^{\top}$ crosses zero as $t$ varies from zero to one. The matrix therefore remains positive definite throughout $t \in [0, 1]$. Taking $t = 1$ proves that the upper-left block in \eqref{eqn:pearson-blocks} is positive definite throughout $\bbk_-$.

To summarize, with probability tending to one, the upper-left block in \eqref{eqn:pearson-blocks} is invertible for every $z \in \bbk_\pm$. Its Schur complement is
\begin{align*}
    & ~ \pars{ \bfY_{\caB} \bfH_n \bfY_{\caB}^{\top} - z \bfI }
    - \bfY_{\caB} \bfH_n \bfY_{\caT}^{\top}
    \pars{\bfY_{\caT} \bfH_n \bfY_{\caT}^{\top} - z \bfI}^{-1}
    \bfY_{\caT} \bfH_n \bfY_{\caB}^{\top} \\
    = & - z \braks[\big]{ \bfI
    +\bfY_{\caB}\bfH_n(\bfH_n\bfQ_{\caT}\bfH_n-z\bfI)^{-1}
    \bfH_n\bfY_{\caB}^{\top} }.
\end{align*}
The same application of Sylvester's law of inertia and Weyl's inequality, as in the proof of Proposition \ref{prop:spectral-transfer}, reduces \eqref{eqn:pearson-spectral-sandwich} to
\begin{equation}
    \sup\nolimits_{z \in \bbk_\pm}
    \norm*{\bfY_{\caB}\bfH_n
    (\bfH_n\bfQ_{\caT}\bfH_n-z\bfI)^{-1}
    \bfH_n\bfY_{\caB}^\top
    - \fkm_\phi(z)\bfL_{\caB}}
    \pconv 0.
    \label{eqn:pearson-block-resolvent}
\end{equation}
To compare the centered and uncentered resolvents, we use the following consequence of the Sherman--Morrison formula,
\begin{equation*}
    \bfH_n (\bfH_n \bfQ_{\caT} \bfH_n - z \bfI)^{-1} \bfH_n
    = \bfG_{\caT} (z)
    - \frac{\bfG_{\caT}(z) \bfke_n \bfke_n^\top \bfG_{\caT}(z)}
    {\angles{\bfke_n, \bfG_{\caT}(z) \bfke_n}}.
\end{equation*}
Proposition \ref{prop:unified-resolvent} controls $\bfY_{\caB} \bfG_{\caT} \bfY_{\caB}^\top - \fkm_\phi \bfL_{\caB}$ uniformly on $\bbk_\pm$. In addition, the denominator in the correction term is also uniformly bounded away from zero with probability tending to one. Indeed, for $z \in \bbk_+$, the nonnegative definiteness of $\bfQ_{\caT}$ ensures that $-\angles{\bfke_n, \bfG_{\caT} (z) \bfke_n} \geq 1/z \geq 1/(\max \bbk_+)$ on the gap event \eqref{eqn:unified-core-trace}. While $z \in \bbk_-$, the preceding bound on $z \angles{\bfke_n, \bfG_{\caT} (z) \bfke_n}$ deduced from Lemma \ref{lem:pearson-mean-scalar} gives the same conclusion. It therefore remains to show that $\norm{\bfY_{\caB} \bfG_{\caT} (z) \bfke_n}$ tends to zero uniformly on $\bbk_\pm$.

We decompose $\bfY_{\caB} = \boldsymbol{\Theta}_{\caB} (\fkS_{\caB} + \fkW_{\caB} + \fkD_{\caB})$ as in \eqref{eqn:unified-atypical-decomposition}. The comparison \eqref{eqn:unified-normalizer-comparison} controls $\boldsymbol{\Theta}_{\caB}$, and Lemma \ref{lem:pearson-mean-direction} controls the terms involving $\fkS_{\caB}$ and $\fkW_{\caB}$. The remaining term is bounded using \eqref{eqn:unified-freezing-bounds} and the resolvent bound \eqref{eqn:unified-core-lower-gap}. Together, these estimates give, under $\bbp^{\bar{\caH}}$ and uniformly over realizations of $\bar{\caH}_n$ in $\tilde{\Omega}_n$,
\begin{equation*}
    \sup\nolimits_{z \in \bbk_\pm}
    \norm{\bfY_{\caB} \bfG_{\caT}(z) \bfke_n}
    \prec n^{-\tau}.
\end{equation*}
The rank-one correction therefore has norm $O_{\prec} (n^{-2\tau})$ under the same conditioning. Averaging over $\bar{\caH}_n$ and using $\bbp (\tilde{\Omega}_n) \to 1$ proves \eqref{eqn:pearson-block-resolvent} and completes the proof.
\end{proof}

% *******************************************
\subsection{Extension of the main results}
% *******************************************

The spectral comparisons in Proposition \ref{prop:pearson-transfer} control eigenvalue counts outside the MP bulk. We also need the edge limits for the principal submatrices of $\bfR_n^{\circ}$ corresponding to the typical rows, which give complementary bounds for the extreme eigenvalues by interlacing.

% *******************************************
\begin{lemma}[Typical Pearson edges]
\label{lem:pearson-typical-edges}
% *******************************************
Fix an integer $k \geq 2$ and assume \eqref{eqn:basic-assumption} and \eqref{eqn:analytic-tail-envelope}. Let $\bfP_{\caT} \in \bbr^{\abs{\caT} \times p_n}$ denote the coordinate projection onto the typical rows $\caT$. Then
\begin{equation}
    \lambda_1 \pars[\big]{\bfP_{\caT} \bfR_n^{\circ} \bfP_{\caT}^{\top}}
    \pconv \lambda_+,
    \qquad
    \lambda_{\abs{\caT}}
    \pars[\big]{\bfP_{\caT} \bfR_n^{\circ} \bfP_{\caT}^{\top}}
    \pconv \lambda_-,
    \label{eqn:pearson-typical-edge}
\end{equation}
where the second convergence additionally requires $\phi \in (0,1)$.
\end{lemma}

% *******************************************
\begin{proof}[Proof of Lemma \ref{lem:pearson-typical-edges}]
% *******************************************
By \eqref{eqn:pearson-small-renormalization} and Weyl's inequality, it suffices to prove both limits for $\bfY_{\caT} \bfH_n \bfY_{\caT}^{\top}$. First consider the upper edge. Recall the i.i.d. proxy $\tilde{\bfY}_{\caT}$ from \eqref{def:iid-proxy}. By \cite[Theorem 3.10]{baiSpectralAnalysisLarge2010} and Lemma \ref{lem:unified-graph} \ref{item:unified-graph-edge-count}, the ESD of $\tilde{\bfY}_{\caT} \tilde{\bfY}_{\caT}^{\top}$ converges weakly almost surely to the MP law with aspect ratio $\phi$. Here the bound $\tilde{y}_{i \mu} = O (n^{-\tau})$ ensures the required Lindeberg condition. This weak convergence gives the lower bound $\lambda_2 (\tilde{\bfY}_{\caT} \tilde{\bfY}_{\caT}^{\top}) \geq \lambda_+ - o_{\bbp} (1)$. Together with  \eqref{eqn:eig-typical-to-edge}, \eqref{eqn:unified-core-comparison}, and rank-one interlacing, we arrive at
\begin{equation*}
    \lambda_+ - o_{\bbp} (1)
    \leq \lambda_2 (\bfR_{\caT})
    \leq \lambda_1 (\bfY_{\caT} \bfH_n \bfY_{\caT}^{\top})
    \leq \lambda_1 (\bfR_{\caT})
    \leq \lambda_+ + o_{\bbp} (1).
\end{equation*}
This proves the first limit in \eqref{eqn:pearson-typical-edge}.

For the lower edge estimate, assume $\phi \in (0,1)$ and fix $z \in (0, \lambda_-)$. According to \eqref{eqn:pearson-centered-core-positivity}, the matrix $\bfY_{\caT} \bfH_n \bfY_{\caT}^{\top} - z \bfI$ is positive definite with probability tending to one. Hence,
\begin{equation*}
    z < \lambda_{\abs{\caT}} (\bfY_{\caT} \bfH_n \bfY_{\caT}^{\top})
    \leq \lambda_{\abs{\caT}} (\bfR_{\caT})
    = \lambda_- + o_{\bbp} (1),
\end{equation*}
with probability tending to one, where the last equality follows from \eqref{eqn:eig-typical-to-edge}. Letting $n \to \infty$ and then $z \uparrow \lambda_-$ proves the second limit in \eqref{eqn:pearson-typical-edge}.
\end{proof}

% *******************************************
\begin{proof}[Proof of Theorem \ref{cor:pearson-extension}]
% *******************************************
The spectral comparisons \eqref{eqn:pearson-spectral-sandwich} and the edge limits \eqref{eqn:pearson-typical-edge} replace the corresponding estimates in Proposition \ref{prop:spectral-transfer} and \eqref{eqn:eig-typical-to-edge}. The effective Gram matrix $\bfL_{\caB}$ remains the same. We may therefore repeat the proof of Corollary \ref{cor:extremal-transfer} and the analysis of $\bfL_{\caB}$ in Sections \ref{sec:collisions} and \ref{sec:critical-collisions} with these substitutions. This proves the extension of Theorems \ref{thm:upper-edge-convergence}, \ref{thm:collision-layer}, \ref{thm:upper-poisson}, \ref{thm:critical-maximum}, \ref{thm:wide-subcritical}, \ref{thm:lower-poisson}, and \ref{thm:critical-minimum} to the Pearson correlation matrix $\bfR_n^\circ$. We now address the remaining assertions, which require additional arguments.

\shortpara{Theorem \ref{thm:all-orders-divergence} (The Ultra-supercritical regime)}. We compare between principal blocks of $\bfR_n^\circ$ and $\bfR_n$ of a fixed size. Fix an integer $k \geq k_*(\phi)$. For every $\caI \subset \dbraks{p_n}$ with $\abs{\caI} = k$, we have
\begin{equation*}
    0 \leq \bfP_{\caI} \bfY_n \bfH_n \bfY_n^\top \bfP_{\caI}^{\top} \leq \bfP_{\caI} \bfR_n \bfP_{\caI}^{\top},
    \qquad
    \norm[\big]{\bfP_{\caI} \bfR_n \bfP_{\caI}^{\top}}
    \leq \operatorname{Tr} \pars[\big]{\bfP_{\caI} \bfR_n \bfP_{\caI}^{\top}}
    \leq k.
\end{equation*}
Together with Lemma \ref{lem:pearson-means} and \eqref{eqn:pearson-diagonal-comparison}, these bounds give the uniform comparison
\begin{equation}
    \sup\nolimits_{\abs{\caI}=k}
    \norm[\big]{\bfP_{\caI} (\bfR_n^{\circ}-\bfR_n) \bfP_{\caI}^{\top}}
    \leq k (1+\norm{\bfDelta_n}) \norm{\bfDelta_n-\bfI}
    + k \norm{\bfv}_{\max}^2
    = o_{\bbp}(1).
    \label{eqn:comparison-center}
\end{equation}
The proof of Lemma \ref{lem:strong-stars} shows that, with probability tending to one, there exist a set of rows $\caI \subset \dbraks{p_n}$ with $\abs{\caI}=k$ and a column $\mu \in \dbraks{n}$ such that $\Xi_{i\mu}$ holds for every $i \in \caI$. Here the events $\Xi_{i \mu}$ are defined in \eqref{eqn:good-strong-star}. For these rows, the argument leading to \eqref{eqn:lower-Rn-k} gives
\begin{equation*}
    \lambda_1(\bfP_{\caI} \bfR_n \bfP_{\caI}^{\top})
    \geq \sum\nolimits_{i \in \caI} \abs{y_{i\mu}}^2
    \geq k-o(1).
\end{equation*}
The uniform comparison \eqref{eqn:comparison-center} and interlacing now imply $\lambda_1 (\bfR_n^{\circ}) \geq k-o_{\bbp}(1)$. Since $k$ can be arbitrarily large, Theorem \ref{thm:all-orders-divergence} also holds for $\bfR_n^{\circ}$.

\shortpara{Theorem \ref{thm:tall} (The tall regime)}. Suppose $\phi > 1$ under the assumptions of Theorem \ref{thm:tall}. Lemma \ref{lem:pearson-means} and the Borel--Cantelli lemma give $\norm{\bfDelta_n - \bfI} \to 0$ almost surely. Theorem \ref{thm:tall} implies that, eventually almost surely, we have $\rank (\bfY_n) = n$ and therefore $\rank(\bfY_n \bfH_n) = \rank (\bfR_n^{\circ}) = n-1$. Recall \eqref{eqn:pearson-diagonal-comparison}, where $\bfDelta_n$ is diagonal. Cauchy's interlacing gives
\begin{equation}
    \pars[\big]{\min\nolimits_{i \in \dbraks{p_n}} \Delta_i}^2 
    \lambda_n (\bfR_n)
    \leq \lambda_{n - 1} (\bfR_n^{\circ})
    \leq \pars[\big]{\max\nolimits_{i \in \dbraks{p_n}} \Delta_i}^2 
    \lambda_{n - 1} (\bfR_n).
    \label{eqn:pearson-tall-compression}
\end{equation}
The MP law of $\bfR_n$ \cite[Theorem 1.2]{jiangLimitingDistributionsEigenvalues2004} and Theorem \ref{thm:tall} imply $\lambda_{n - 1} (\bfR_n) \to \lambda_-$ almost surely. Applying \eqref{lem:pearson-means} and taking limits in \eqref{eqn:pearson-tall-compression} gives $\lambda_{n - 1} (\bfR_n^{\circ}) \to \lambda_-$ almost surely.

\shortpara{Theorem \ref{thm:wide-supercritical} (The supercritical wide regime)}. Since $0 \leq \bfY_n \bfH_n \bfY_n^{\top} \leq \bfR_n$, the diagonal representation \eqref{eqn:pearson-diagonal-comparison} and the variational characterization of eigenvalues give
\begin{equation*}
    0 \leq \lambda_{p_n} (\bfR_n^{\circ})
    \leq \norm{\bfDelta_n}^2
    \lambda_{p_n} (\bfY_n \bfH_n \bfY_n^{\top})
    \leq \norm{\bfDelta_n}^2 \lambda_{p_n} (\bfR_n).
\end{equation*}
Lemma \ref{lem:pearson-means} gives $\norm{\bfDelta_n} = 1 + o_{\bbp} (1)$, so Theorem \ref{thm:wide-supercritical} implies $\lambda_{p_n} (\bfR_n^{\circ}) = o_{\bbp} (1)$.
\end{proof}

% *******************************************
\section{Norms of random principal submatrix}
\label{sec:tech-lemmas}
% *******************************************

% *******************************************
\begin{proof}[Proof of Lemma \ref{lemma:fixed-cardinality-random-compression}]
% *******************************************
For any $\caI \subset \dbraks{n}$, write $\bfP_\caI : \bbr^n \to \bbr^{\abs{\caI}}$ for the corresponding coordinate projection and abbreviate $\bfH[\caI] \equiv \bfP_\caI \bfH \bfP_\caI^{\top}$. Let $\caJ \subset \dbraks{n}$ be uniformly distributed over all subsets of cardinality $d$ and set $\beta = d/n$. It suffices to prove the moment bound
\begin{equation}
    \pars[\big]{\bbe \norm{\bfH[\caJ]}^\ell}^{1 / \ell}
    \leq C \pars[\big]{
    \ell \norm{\bfH}_{\max}
    + \sqrt{\beta \ell} \norm{\bfH}},
    \qquad \ell \geq 2 \log n,
    \label{eq:fixed-cardinality-compression-moment-bound}
\end{equation}
where $C > 0$ is some universal constant. In fact, taking $\ell = \lceil 2 \log n \rceil$ in \eqref{eq:fixed-cardinality-compression-moment-bound} and a standard application of Markov's inequality yields \eqref{eq:fixed-cardinality-compression-moment}. 

If $d > n / 2$, then $\norm{\bfH[\caJ]} \leq \norm{\bfH} \leq 2 \beta \norm{\bfH}$. Therefore, it suffices to prove \eqref{eq:fixed-cardinality-compression-moment-bound} for $\beta = d / n \leq 1 / 2$. On an enlarged probability space, let
\begin{equation*}
    \pi \sim \operatorname{Unif} (\caS_n),
    \qquad
    k \sim \operatorname{Binom} (n, \beta)
\end{equation*}
be independent, where $\caS_n$ denotes the set of all permutations of $\dbraks{n}$. Define the random sets
\begin{equation}
    \caJ_\sharp := \curls*{\pi (1), \ldots, \pi (d)},
    \qquad
    \caI_\sharp := \curls*{\pi (1), \ldots, \pi (k)}.
    \label{eq:fixed-cardinality-bernoulli-coupling}
\end{equation}
Since $\pi$ is uniform, $\caJ_\sharp$ is uniformly distributed over the subsets of $\dbraks{n}$ of cardinality $d$. In particular, $\caJ_\sharp$ has the same distribution as $\caJ$. Hence, it suffices to prove \eqref{eq:fixed-cardinality-compression-moment-bound} for the random subset $\caJ_\sharp$. On the other hand, for every deterministic $\caI \subset \dbraks{n}$ with $\abs{\caI} = m$,
\begin{equation*}
        \bbp \{\caI_\sharp = \caI\}
        = \bbp \{\caI_\sharp = \caI \mid k = m\}
        \, \bbp \{k = m\} \\
        = \beta^m (1 - \beta)^{n - m}.
\end{equation*}
Hence, the indicators $\{ \bbone \{\mu \in \caI_\sharp\} \}_{\mu \in \dbraks{n}}$ are independent $\mathrm{Bern} (\beta)$ variables. In particular,
\begin{equation}
    \norm{\bfH[\caI_\sharp]}
    \overset{\mathrm{d}}{=}
    \norm{\bfPi \bfH \bfPi},
    \qquad
    \bfPi = \diag (\Pi_1, \ldots, \Pi_n),
    \qquad
    \Pi_\mu \overset{\mathrm{i.i.d.}}{\sim} \operatorname{Bern} (\beta).
    \label{eq:bernoulli-compression-law}
\end{equation}
Consider the event $\Xi := \{k \geq d\}$. Since $n \beta = d$ is a median of the binomial random variable $k \sim \operatorname{Binom} (n, \beta)$, we have $\bbp (\Xi) \geq 1 / 2$. On the event $\Xi$, the construction \eqref{eq:fixed-cardinality-bernoulli-coupling} gives $\caJ_\sharp \subset \caI_\sharp$. Since $\bfH$ is symmetric, the Cauchy interlacing theorem yields $\norm{\bfH[\caJ_\sharp]} \leq \norm{\bfH[\caI_\sharp]}$ on $\Xi$. Here the set $\caJ_\sharp$ depends only on the random permutation $\pi$, whereas $\Xi$ depends only on $k$. Hence, $\caJ_\sharp$ and $\Xi$ are independent, and it follows that
\begin{equation*}
    \begin{aligned}
        \bbe \norm{\bfH[\caJ_\sharp]}^\ell
        & = \bbe \braks[\big]{\norm{\bfH[\caJ_\sharp]}^\ell \mid \Xi}
        \leq \bbe \braks[\big]{\norm{\bfH[\caI_\sharp]}^\ell \mid \Xi} 
        \leq 2 \bbe \braks[\big]{\norm{\bfH[\caI_\sharp]}^\ell \bbone_\Xi}
        \leq 2 \bbe \norm{\bfH[\caI_\sharp]}^\ell.
    \end{aligned}
\end{equation*}
Therefore, by the fact that $\caJ \overset{\mathrm{d}}{=} \caJ_\sharp$ and \eqref{eq:bernoulli-compression-law}, we have, for $\ell \geq 1$,
\begin{equation*}
    \pars[\big]{\bbe \norm{\bfH[\caJ]}^\ell}^{1 / \ell}
    = \pars[\big]{\bbe \norm{\bfH[\caJ_\sharp]}^\ell}^{1 / \ell}
    \leq 2^{1 / \ell}
    \pars[\big]{\bbe \norm{\bfH[\caI_\sharp]}^\ell}^{1 / \ell}
    \leq 2 \pars[\big]{\bbe \norm{\bfPi \bfH \bfPi}^\ell}^{1 / \ell}.
\end{equation*}
Tropp's random principal submatrix estimate \cite[Theorem 1.1]{troppNormsRandomSubmatrices2008} yields, 
\begin{equation*}
    \pars[\big]{\bbe \norm{\bfPi \bfH \bfPi}^\ell}^{1 / \ell}
    \leq C \pars[\big]{
    \ell \norm{\bfH}_{\max}
    + \sqrt{\beta \ell} \, \norm{\bfH}_{1 \to 2}
    + \beta \norm{\bfH} },
    \qquad \ell \geq 2 \log n,
\end{equation*}
where $C > 0$ is a constant and $\norm{\bfH}_{1 \to 2} := \max_{\mu \in \dbraks{n}} \norm{\bfH \bfe_\mu}$ is the maximal column norm. Note that $\norm{\bfH}_{1 \to 2} \leq \norm{\bfH}$. Recall $\beta \leq 1/2$. Combining the preceding two estimates and concludes the proof of \eqref{eq:fixed-cardinality-compression-moment-bound}. 
\end{proof}

% --- Acknowledgements ---

\section*{Acknowledgements}
The authors are sincerely grateful to Professors Guangming Pan and Wang Zhou for introducing us to this problem and for their helpful discussions and suggestions during the preparation of this manuscript.

% --- Bibliography ---

\printbibliography

% --- Document ends here ---

\end{document}